\documentclass[a4paper,10pt,reqno]{amsart}
\usepackage[english,activeacute]{babel}
\usepackage[utf8]{inputenc}
\usepackage{amsmath,amssymb,amsfonts,amsthm,amscd}
\usepackage[mathscr]{eucal}
\usepackage{hyperref}

\usepackage{enumitem}
\setlist[itemize]{topsep=0.2em, itemsep=0.2em, leftmargin=2em}
\setlist[enumerate]{topsep=0.2em, itemsep=0.2em, leftmargin=2em}
\setlist[description]{topsep=0.2em, itemsep=0.2em, leftmargin=2em}

\usepackage{latexsym}

\usepackage{graphicx}
\usepackage{color}
\usepackage{graphicx}

\newcommand{\N}{\mathbb{N}}
\newcommand{\Z}{\mathbb{Z}}

\newcommand{\R}{\mathbb{R}}

\newcommand{\df}{\mathrm{d}}

\newcommand{\E}{\mathbb{E}}

\newcommand{\X}{\mathfrak{X}}

\newcommand{\Length}{\mathop{\rm Length}\nolimits}

\newcommand{\Flux}{\mathrm{Flux}}

\newcommand{\sm}{-}

\newtheorem{theorem}{Theorem}[section]
\newtheorem{proposition}[theorem]{Proposition}
\newtheorem{corollary}[theorem]{Corollary}
\newtheorem{lemma}[theorem]{Lemma}

\theoremstyle{definition}
\newtheorem{definition}[theorem]{Definition}

\theoremstyle{remark}
\newtheorem{remark}[theorem]{Remark}

\numberwithin{equation}{section}
\title{The Jenkins--Serrin problem in 3-manifolds with a Killing vector field}
\date{}

\author{Andrea Del Prete}
\address{Dipartimento di Ingegneria e Scienze dell'Informazione e Matematica\\
	Universit\`{a} dell'Aquila\\
	Via Vetoio Loc. Coppito -- 67100 L'Aquila (Italy)}
\email{andrea.delprete@graduate.univaq.it}

\author{Jos\'{e} M. Manzano}
\address{Departamento de Matemáticas\\Universidad de Ja\'{e}n\\Campus Las Lagunillas -- 23071 Jaén (Spain)}
\email{jmprego@ujaen.es}

\author{Barbara Nelli}
\address{Dipartimento di Ingegneria e Scienze dell'Informazione e Matematica\\
	Universit\`{a} dell'Aquila\\
	Via Vetoio Loc. Coppito -- 67100 L'Aquila (Italy)}
\email{nelli@univaq.it}



\begin{document}
	
\begin{abstract}
We consider a Riemannian submersion from a 3-manifold $\E$ to a surface $M$, both connected and orientable, whose fibers are the integral curves of a Killing vector field without zeros, not necessarily unitary. We solve the Jenkins--Serrin problem for the minimal surface equation in $\mathbb{E}$ over a relatively compact open domain $\Omega\subset M$ with prescribed finite or infinite values on some arcs of the boundary under the only assumption that the same value $+\infty$ or $-\infty$ cannot be prescribed on two adjacent components of $\partial\Omega$ forming a convex angle. The domain $\Omega$ can have reentrant corners as well as closed curves in its boundary. We show that the solution exists if and only if some generalized Jenkins--Serrin conditions (in terms of a conformal metric in $M$) are fulfilled. We develop further the theory of divergence lines to study the convergence of a sequence of minimal graphs. We also provide maximum principles that guarantee the uniqueness of the solution. Finally, we obtain new examples of minimal surfaces in $\mathbb{R}^3$ and in other homogeneous $3$-manifolds.
\end{abstract}

\maketitle
	
\section{Introduction}
	
We aim at proving a generalization to Killing submersions of the so called Jenkins--Serrin Theorem over relatively compact domains. This problem was firstly treated in Euclidean space $\R^3$ by Jenkins and Serrin~\cite[Thm.~3 and~4]{JS}, who considered bounded domains $\Omega\subset\R^2$ with $\partial\Omega$ composed of straight segments and convex arcs. They found necessary and sufficient elementary conditions on the lengths of the sides of polygons inscribed in $\Omega$ (see Definition~\ref{def:mu-polygon}) that guarantee the existence of a minimal graph in $\R^3$ over $\Omega$ with prescribed values on the regular components of $\partial\Omega$, as well as its uniqueness (possibly up to vertical translations). They incorporated into the classical Dirichlet problem the possible asymptotic infinite values on some straight components of $\partial\Omega$. Over the years, analogous results over bounded domains have been proven in other Riemannian 3-manifolds: in $\mathbb H^2\times\R$ by the third author and Rosenberg~\cite{NR,NR-err}; in $M\times\R$ by Pinheiro~\cite{Pin} (geodesically convex domains), Mazet, Rodriguez and Rosenberg~\cite{MRR} (general case), and Eichmair and Metzger~\cite{EM} (under milder assumptions and also allowing closed geodesics as part of the boundary); in $\widetilde{\mathrm{PSL}}_2(\R)$ by Younes~\cite{Younes}; and in $\mathrm{Sol}_3$ by Nguyen~\cite{Ng}. All these problems can be treated together by noticing that they deal with surfaces transverse to a Killing vector field. There are really few approaches to Dirichlet type problems with respect to non-Killing directions, e.g. see the solution to the Radó problem in homogeneous semidirect products in~\cite{MMP}. There are also many works on the Jenkins--Serrin problem for positive constant mean curvature graphs (starting with the work of Spruck~\cite{Spruck}, see also~\cite{EM} and the references therein) as well as for graphs over unbounded domains in $M\times\R$, being $M$ a Hadamard surface, and $\widetilde{\mathrm{PSL}}_2(\R)$ (starting with the work of Collin and Rosenberg in $\mathbb{H}^2\times\mathbb{R}$~\cite{CR}).

A Killing submersion is a Riemannian submersion $\pi:\mathbb{E}\to M$ whose fibers are the integral curves of a (not necessarily unitary) Killing vector field $\xi$ that never vanishes. These geometric structures describe locally any $3$-manifold endowed with a Killing vector field without zeros and provide a global framework for all homogeneous $3$-manifolds (among many other ambient spaces). Killing submersions were completely classified by the second author and Lerma~\cite{LerMan} in terms of the base surface $M$ and two geometric functions on $M$, namely the bundle curvature $\tau$ and the norm of the Killing vector field $\mu$. The triplet $(M,\tau,\mu)$ can be prescribed arbitrarily and encodes the geometry and topology of the submersion, which can be recovered uniquely if $\E$ is assumed simply connected~\cite[Thm.~2.6 and~2.9]{LerMan}.

If $\tau\equiv 0$, then $\E$ is isometric to the warped product $M\times_\mu\R$ (indeed a Riemannian product if $\mu\equiv 1$) where there is a natural notion of graph. If $\tau\not\equiv 0$, although the horizontal distribution orthogonal to $\xi$ is not integrable, we can still define a Killing graph in $\E$ as a section of the submersion everywhere transverse to the fibers of $\pi$. After prescribing a zero section at will, a graph $\Sigma\subset\E$ can be identified with a function $u\in\mathcal C^\infty(\Omega)$, where $\Omega=\pi(\Sigma)\subset M$. The fact that $\xi$ is Killing leads naturally to a divergence type expression for the mean curvature of $\Sigma$ involving the generalized gradient $Gu$. This operator $Gu$ outputs a vector field in $\Omega$ only depending on $\Sigma$, not on the choice of the zero section~\cite[Lemma~3.1]{LerMan} (see also Lemma~\ref{lemma:H} below). Killing graphs have played a central role in the theory of minimal and constant mean curvature surfaces, attracting the attention of many authors, see for instance the general solutions to the Dirichlet problem given by Dacjzer, De Lira and Ripoll~\cite{DD,DDR}, as well as other classification results such as the work of Alías, Dacjzer and Ripoll~\cite{ADR}. We have highlighted these papers since they treat the general case, yet in specific Killing submersions such as product spaces or homogeneous spaces, there is a vast literature on Killing graphs.

In a Killing submersion, a minimal vertical cylinder (i.e., a surface everywhere tangent to $\xi$) projects to a curve in $M$ which is geodesic for the conformal metric $\mu^2\df s_M^2$ in $M$, see Proposition~\ref{prop:H-invariant}. Every surface invariant by a continuous $1$-parameter group of isometries in any $3$-manifold is locally a vertical cylinder for some Killing submersion, so this gives a geometric characterization of the curves in the orbit space that generate any invariant surface with constant mean curvature in any $3$-manifold with a Killing vector field. In particular, such curves are characterized by their initial data, see Corollary~\ref{coro:invariant-H-surfaces}. 

This viewpoint also reveals the existence of minimal open book foliations of a neighborhood of any vertical fiber of any Killing submersion (with binding the fiber), see Corollary~\ref{coro:open-book-decomposition}. Furthermore, we shall see that the fact that minimal vertical cylinders project to $\mu$-geodesics (we will use $\mu$ as a prefix to indicate that we are using the conformal metric $\mu^2\df s_M^2$) implies that infinite values in the Jenkins--Serrin problem can only be prescribed along $\mu$-geodesic components of $\partial\Omega$ as in~\cite[Thm.~3.3]{RST}. Moreover, continuous finite boundary values can be prescribed on $\mu$-convex components of $\partial\Omega$, because $\mu$-convexity is the natural assumption to construct the usual barriers (see Theorem~\ref{thm:generalized-existence} and Proposition~\ref{prop:boundary-value}). It is important to remark that $\mu$-convexity will be always understood as \emph{not strict}, in the sense that a $\mu$-geodesic is also considered to be $\mu$-convex.

In spite of the very diverse behaviors of Killing submersions, our main result (Theorem~\ref{thm:JS}) shows that the necessary and sufficient conditions for the existence of a minimal graph over $\Omega$ are the very same as in the original Jenkins--Serrin result (using the $\mu$-metric of $M$ in the computation of lengths. It is quite satisfactory to realize that the original statement in $\R^3$ still applies with minor changes in this very general setting, but indeed our approach needs less assumptions. There are typically two conditions on a Jenkins--Serrin problem:
\begin{description}
	\item[(C1)] The value $+\infty$ or $-\infty$ is not assigned to two adjacent components of $\partial\Omega$ that meet at a convex corner.
	\item[(C2)] If no continuous finite values are assigned, then the subsets of $\partial\Omega$ where $+\infty$ and $-\infty$ are assigned are both disconnected.
\end{description}
Condition (C1) is necessary for the existence of solutions  as we shall prove in Proposition~\ref{prop:admissibility}. Condition (C2) was used in Jenkins and Serrin's original argument and has been required in the case of $M\times\R$ in~\cite{Pin} or~\cite{MRR} (but not in~\cite{EM}). Note that (C2) is automatically satisfied in $\R^3$, $\mathbb{H}^2\times\mathbb{R}$, $\widetilde{\mathrm{PSL}}_2(\R)$ or $\mathrm{Sol}_3$, but it discards some configurations when the $\mu$-metric has positive Gauss curvature as in the case of $\mathbb{S}^2\times\R$ or $\mathbb{S}^3$ (that fibers over $\mathbb{S}^2$ via the Hopf fibration and the $\mu$-metric is round). Indeed, some symmetric configurations in $\mathbb{S}^2\times\R$ show that (C2) is not strictly necessary, as pointed out in~\cite[Rmk.~3.5]{MRR}. We give a counterexample in Section~\ref{sec:examples} (see Figure~\ref{fig:counterexample}) that reveals a mistake in the proof of existence in~\cite{MRR} (in the general case of $M\times\R$). This example suggests that condition (C2) cannot be dropped if one uses Jenkins and Serrin's approach.

Consequently, we have decided to use and extend the theory of divergence lines introduced by Mazet~\cite{Mazet} and developed in~\cite{MRR} in $\mathbb{H}^2\times\mathbb{R}$ (see Section~\ref{sec:divergence-lines}). We will not prove the results of~\cite{MRR} that literally extend to Killing submersions, but we will need a bunch of new arguments to deal with situations that cannot arise in $\mathbb{H}^2\times\R$. 

Besides simplifying some arguments in some of the cited papers, the present work makes several contributions that is worth highlighting:
\begin{enumerate}
	\item We prove that there are minimal graphs over any relatively compact domain (Lemma~\ref{lemma:minimal-sections}). These graphs act as zero sections and come in handy in the Perron process to solve the Dirichlet problem with finite boundary values (Perron is needed as long as we allow reentrant corners in the domain $\Omega$).
	\item Douglas criterion is commonly used to obtain a family of minimal annuli in the construction of Scherk barriers (as in~\cite{NR}), but this is not possible in a general Killing submersion since minimal vertical cylinders are not necessarily area-minimizing. We use the Meeks-Yau solution of the Plateau problem to get minimal disks instead of annuli (see Proposition~\ref{prop:scherk}).

	\item Contrary to the case of $\mathbb{H}^2\times\mathbb{R}$, divergence lines might accummulate on $\overline\Omega$, but we will prove that they are actually properly embedded (Lemma~\ref{lem:divline-properly-embedded}). We also need to provide a new argument to prove that no divergence line ends at the interior of a boundary component (Lemma~\ref{lem:divline-boundary-values}) because~\cite{MRR} uses the symmetries of $\mathbb{H}^2\times\mathbb{R}$, which are not available in a Killing submersion.
	\item We have to deal with the fact that there can be uncountably many divergence lines (again contrary to the case of $\mathbb{H}^2\times\R$ in which this number is finite). We show that, up to a subsequence, they are disjoint and belong to finitely many nonempty isotopy classes (which can be understood rather well separately) and define different \emph{divergence heights} (Proposition~\ref{prop:disjoint-lines}). This settles a comment in~\cite[Rmk.~4.5]{MRR} and reveals that the number of relevant inscribed $\mu$-polygons and convergence components is actually finite (Corollary~\ref{coro:components-are-polygons}).
\end{enumerate}
We must point out that most of the ideas developed in Section~\ref{subsec:divlines} also apply (or can be adapted) to very general bounded or unbounded domains which are not of Jenkins--Serrin type. For instance, Lee and the first two authors~\cite[\S4]{DLM} use divergence lines in unbounded domains of Killing submersions where the ambient geometry is not even bounded. The proof of our Jenkins--Serrin theorem also applies to cases more general than those explicitly stated. We will give examples of graphs over immersed domains and examples in Killing submersions with fibers of finite length, see Section~\ref{sec:examples}. We will also produce new examples of minimal surfaces with boundary in $\R^3$ which are Jenkins--Serrin graphs with respect to rotations and accumulate on catenoids and planes (which are the vertical cylinders in this setting), as well as a complete Scherk type surface in $\mathrm{Nil}_3$ which is neither embedded nor proper (by the effect of the holonomy).

The contents of this manuscript are organized as follows. In Section~\ref{sec:killing}, we give local models for Killing submersions and prove some generalities for graphs and vertical cylinders. In Section~\ref{sec:convergence}, we adapt some of the gradient and curvature estimates given by Rosenberg, Souam and Toubiana~\cite{RST} (as well as their Uniform Graph Lemma) to the case of Killing submersions, and provide some further estimates towards a Compactness Theorem and a Monotone Convergence Theorem. Section~\ref{sec:dirichlet} is devoted to prove the existence and uniqueness of solution for the Dirichlet problem with finite boundary values (the convex case will be proved using Meeks--Yau solution of the Plateau problem~\cite{MY1}, whereas the general case will be treated by the Perron process following Sa Earp and Toubiana~\cite{ST1,ST3}). The construction of local Scherk type surfaces will be discussed in Section~\ref{sec:local-barriers}, as well as how they are used to prove that a convergent sequence of graphs achieves the right boundary values. In Section~\ref{sec:JS} we state the Jenkins--Serrin Theorem and develop the flux arguments needed in the proof; in particular, we prove that the Jenkins--Serrin conditions are necessary and give a general maximum principle that implies the uniqueness of solutions. Section~\ref{sec:divergence-lines} recalls some results for divergence lines in~\cite{MRR} and proves the new ones needed for the existence of solutions. Although we assume that the domain is a Jenkins--Serrin domain, the Jenkins--Serrin conditions are not applied until Section~\ref{subsec:existence-JS}, where we finish the proof. Finally, some examples are discussed in Section~\ref{sec:examples}.
		
\medskip	
\noindent\textbf{Acknowledgement.} This work is part of the first author's PhD thesis and has been partially supported by PRIN-2010NNBZ78-009 and INdAM-GNSAGA. The second author was supported by the project PID2019.111531GA.I00 funded by MCIN/AEI/10.13039/501100011033 and by a FEDER-UJA project (Ref.\ 1380860). The second and third authors would like to thank Massimiliano Pontecorvo for his hospitality when the grounds for this work arose.

\section{Preliminaries on Killing submersions}\label{sec:killing}

Let $\pi:\E\to M$ be a Killing submersion from a Riemannian 3-manifold $\E$ to Riemannian surface $M$, both connected and oriented. This means that the fibers of $\pi$ are the integral curves of a Killing vector field $\xi\in\X(\E)$ without zeros. In this paper we will assume that the fibers of $\xi$ have infinite length unless differently specified, which is a natural assumption for the Jenkins--Serrin problem. In Section~\ref{sec:examples}, we will discuss how to deal with the case the fibers are compact.

We will consider two distinguished geometric functions, namely the Killing length $\mu=\|\xi\|\in\mathcal{C}^\infty(\E)$ and the bundle curvature $\tau\in\mathcal{C}^\infty(\E)$, given by 
\begin{equation}\label{BundleCurvature}
	\tau(q)=-\frac{1}{\mu(q)}\langle\overline\nabla_u\xi,v\rangle,
\end{equation}		 
where $\{u,v,\xi_q/\mu(q)\}$ is a positively oriented orthonormal basis of $T_q\E$ and $\overline\nabla$ denotes the Levi-Civita connection in $\E$. The elements of the $1$-parameter group of isometries $\{\phi_t\}_{t\in\R}$ associated to $\xi$ will be called \emph{vertical translations}. Vector fields in $\E$ parallel (resp.\ orthogonal) to $\xi$ are called \emph{vertical} (resp.\ \emph{horizontal}). Since $\tau$ and $\mu$ are invariant by vertical translations, they are constant along the fibers of $\pi$ and induce smooth functions in $M$, which will be also denoted by $\tau,\mu\in\mathcal{C}^\infty(M)$ by an abuse of notation.		

These functions characterize the geometry of the submersion in the sense that, if $M$ is a simply connected Riemannian surface and $\tau,\mu\in C^\infty(M)$ are such that $\mu>0$, then there is a Killing submersion $\pi:\E\to M$ with bundle curvature $\tau$ and Killing length $\mu$. This Killing submersion is unique if we assume that $\E$ is simply connected~\cite[Thms.~2.6 and~2.9]{LerMan}. Furthermore, if $\E$ or $M$ are not simply connected, then existence still holds true but uniqueness might fail. In that case, if $\rho\colon\widetilde{M}\to M$ and $\sigma\colon\widetilde{\E}\to\E$ are the universal Riemannian covering maps of $M$ and $\E$, respectively, then there exist a unique Killing submersion $\widetilde{\pi}\colon\widetilde{\E}\to\widetilde{M}$ such that $\rho\circ\widetilde{\pi}=\pi\circ\sigma$ and a group of Killing isometries $G$ acting properly discontinuously on $\widetilde{\E}$ such that $\widetilde{\E}=\E/G$. In other words, $\pi:\E\to M$ is the quotient of another Killing submersion with simply connected total space~\cite[Thm~2.14]{LerMan}.

\subsection{Existence of minimal sections}
The existence of a global smooth section of a Killing submersion $\pi:\E\to M$ is useful since it yields a trivialization of the bundle, in which case $\E$ is diffeomorphic to $M\times\R$. Under our assumption that fibers have infinite length, the existence of global sections is always guaranteed~\cite[Thm.~12.2]{Steenrod}. If fibers were compact, then the existence of global sections is characterized in~\cite[Prop.~3.3]{LerMan}. However, it is an open question to determine if a Killing submersion whose fibers have infinite length over a non-compact base admits a global \emph{minimal} section. Some entire minimal graphs in unit Killing submersions with base $M=\R^2$ have been obtained in~\cite{DLM}. Next lemma gives minimal sections that are large enough for our Jenkins--Serrin constructions.

\begin{lemma}\label{lemma:minimal-sections}
Let $\pi:\E\to M$ a Killing submersion whose fibers have infinite length. If $U\subset M$ is open and relatively compact, then there is a minimal section over $U$.
\end{lemma}
	
\begin{proof}
If $M$ is compact, since the fibers of $\pi$ have infinite length, then $\pi$ admits a global smooth section~\cite[Thm.~12.2]{Steenrod}, so $\int_M\frac{\tau}{\mu}=0$ by~\cite[Prop.~3.3]{LerMan}. Hence, there is a minimal section over all $M$ by~\cite[Thm.~3.6]{LerMan} and we are done. This means that we can assume $M$ is not compact in what follows. Therefore, there is an increasing sequence of open subsets $G_n\subset M$ such that $\cup_{n\in\N}G_n=M$ and the boundary of each $G_n$ consists of finitely many smooth Jordan curves. Since $\overline U$ is compact, there will be some $n_0\in\N$ such that $G=G_{n_0}$ contains $\overline U$. 

Let $\gamma_1,\ldots,\gamma_r:[0,1]\to M$ be the boundary components of $G$. Each $\gamma_k$ can be lifted to a horizontal curve $\widehat\gamma_k:[0,1]\to\E$ and let $d_k\in\R$ be the difference of heights of its endpoints, i.e., $\widehat\gamma_k(1)=\phi_{d_k}(\widehat\gamma_k(0))$. Let us attach smoothly a disk $D_k$ to $\overline G$ such that $\partial D_k=\gamma_k$ and extend smoothly the Riemannian metric of $G$ to $\overline G\cup D_k$. Let us also extend smoothly $\tau$ and $\mu$ to $\overline G\cup D_k$ in such a way that $\int_{D_k}\frac{\tau}{\mu}=2d_k$. By uniqueness of Killing submersions~\cite{LerMan}, this implies that the total space $D_k\times\R$ of the Killing submersion over $D_k$ with bundle curvature $\tau$ and Killing length $\mu$ can be glued smoothly with $\pi^{-1}(\overline G)$ along $\pi^{-1}(\gamma_k)$, by just making a horizontal geodesic in $\partial D_k\times\R$ coincide with $\widehat\gamma_k$. After repeating this for all boundary components of $G$, we find a Killing submersion $\pi':\E'\to M'$ whose fibers have infinite length, $M'=\overline G\cup D_1\cup\ldots\cup D_r$ is compact, and induces on $(\pi')^{-1}(G)$ the same Riemannian metric as in $\pi^{-1}(G)$. The problem is therefore reduced to the compact case.
\end{proof}

\begin{remark}
In the sequel, we will always consider graphs over a relatively compact domain $U\subset M$ with respect to a fixed zero section $F_0$ which is defined on a suitable neighborhood of $\overline U$. By Lemma~\ref{lemma:minimal-sections}, this section can be assumed minimal if needed, which will simplify some of the arguments.
\end{remark}

\subsection{Local coordinates}
We will now recall a local description of the metric of $\E$ in terms of $\mu$ and $\tau$, which is enough for local computations. Consider a simply connected non-compact neighborhood $U$ of $p\in M$ parametrized conformally as $\varphi\colon\left(\Omega,\df s^2_\lambda\right)\to U$, where $\Omega\subset\R^2$ is an open disk and $\df s^2_\lambda=\lambda^2(\df x^2+\df y^2)$ for some positive $\lambda\in\mathcal C^\infty(\Omega)$. We can produce the global diffeomorphism
\[
	\psi\colon\Omega\times\R\longrightarrow\pi^{-1}(U),\qquad \psi(x,y,t)=\phi_t(F_0(\varphi(x,y))),
\]
where $\{\phi_t\}_{t\in\R}$ is the family of vertical translations associated to $\xi$ and $F_0\colon U\to \E$ is the zero section. The map $\psi$ transforms the original Killing submersion into the natural projection $\pi_1\colon\Omega\times\R\to\Omega$ in view of the relation $\pi\circ\psi=\varphi\circ\pi_1$. We will consider a metric $\df s^2$ on $\Omega\times\R$ that makes $\psi$ an isometry. A global orthonormal frame $\{e_1=\lambda^{-1}\partial_x,e_2=\lambda^{-1}\partial_y\}$ in $(\Omega,\df s^2_\lambda)$ can be lifted to an orthonormal frame $\left\{E_1,E_2\right\}$ of the horizontal distribution, which is orthogonal to $E_3=\frac{1}{\mu}\partial_t$. Since $\pi_1(x,y,t)=(x,y)$, there exist $a,b\in\mathcal{C}^\infty(\Omega)$ such that
\begin{equation}\label{LocalFrame}
	\begin{aligned}
		(E_1)_{(x,y,t)}=&\tfrac{1}{\lambda(x,y)}\partial_x+a(x,y)\partial_t,\\
		(E_2)_{(x,y,t)}=&\tfrac{1}{\lambda(x,y)}\partial_y+b(x,y)\partial_t,\\
		(E_3)_{(x,y,t)}=&\tfrac{1}{\mu(x,y)}\partial_t.
	\end{aligned}
\end{equation} 
This is equivalent to saying that $\df s^2$ is expressed in coordinates as
\begin{equation}\label{MetricInCoordinates}
	\df s^2=\lambda^2(\df x^2+\df y^2)+\mu^2(\df t-\lambda(a\df x+b\df y))^2.
\end{equation}
From~\eqref{BundleCurvature} and~\eqref{LocalFrame}, it is easy to deduce that
\begin{equation}\label{eqn:taumu}
\tfrac{2\tau}{\mu}=\tfrac{1}{\mu}\langle\left[E_1,E_2\right],E_3\rangle=\tfrac{1}{\lambda^2}\left((\lambda b)_x-(\lambda a)_y\right),
\end{equation}
see also~\cite[Eq.\ (2.5)]{LerMan}. Observe that~\eqref{eqn:taumu} is the only condition on $a$ and $b$ to produce a Killing submersion with bundle curvature $\tau$ and Killing length $\mu$. Moreover, using Equation~\eqref{eqn:taumu} and Koszul formula, we can work out the Levi-Civita connection in the frame~\eqref{LocalFrame} to deduce that
\begin{equation}\label{LeviCivitaConnection}
	\begin{array}{lll}
		\overline\nabla_{E_1}E_1=-\frac{\lambda_y}{\lambda^2}E_2, &\overline\nabla_{E_1}E_2=\frac{\lambda_y}{\lambda^2}E_1+\tau E_3, &\overline\nabla_{E_1}E_3=-\tau E_2, \\
		\overline\nabla_{E_2}E_1=\frac{\lambda_x}{\lambda^2}E_2-\tau E_3, &\overline\nabla_{E_2}E_2=-\frac{\lambda_x}{\lambda^2}E_1, &\overline\nabla_{E_2}E_3=\tau E_1, \\
		\overline\nabla_{E_3}E_1=-\tau E_2+\frac{\mu_x}{\lambda\mu}E_3, &\overline\nabla_{E_3}E_2=\tau E_1+\frac{\mu_y}{\lambda\mu}E_3, & \overline\nabla_{E_3}E_3=-\frac{1}{\mu}\overline\nabla\mu.
	\end{array}
\end{equation}
The identity $\overline\nabla_{E_3}E_3=-\frac{1}{\mu}\overline\nabla\mu$ implies that the only vertical fibers of $\pi$ which are geodesic correspond to the critical points of $\mu$.

\subsection{The mean curvature of a vertical cylinder}
Let $\Sigma$ be an orientable surface immersed in $\E$ and denote by $N$ a smooth unit normal vector field along $\Sigma$. This defines the function $\nu=\langle N,\xi\rangle$, known as the \emph{angle function} of the surface. 

Note that $\Sigma$ is vertical at points where $N$ is horizontal. It follows that $\Sigma$ is everywhere vertical if and only if $\nu\equiv 0$, i.e., $\Sigma\subset\pi^{-1}(\Gamma)$ for some regular curve $\Gamma\subset M$. Consider a unit-speed parametrization $\gamma\colon[a,b]\to\Gamma\subset M$ and assume that $\Sigma=\pi^{-1}(\Gamma)$ (this surface is known as the \emph{vertical} or \emph{Killing cylinder} over $\Gamma$). Consider the orthonormal frame $\{X,E_3=\frac{1}{\mu}\xi\}$ in $\Sigma$, where $X$ is a horizontal vector field on $\Sigma$ that projects to $\gamma'$. The second fundamental form in the frame $\{X,E_3\}$ is given by the matrix
\[\sigma\equiv\begin{pmatrix}
	\langle\overline\nabla_X X,N\rangle&\langle \overline\nabla_X E_3,N\rangle\\\langle \overline\nabla_{E_3} X,N\rangle&\langle \overline\nabla_{E_3} E_3,N\rangle
\end{pmatrix}=\begin{pmatrix}
\kappa_g&\tau\\\tau&\langle\tfrac{-1}{\mu}\nabla\mu,\eta\rangle
\end{pmatrix},\]
where $\kappa_g$ is the geodesic curvature of $\gamma$ in $M$ with respect to the unit normal $\eta=\pi_*N$ to $\gamma$ in $M$ and $\nabla$ denotes the gradient in $M$. This follows from~\eqref{LeviCivitaConnection} using that $X$ and $N$ are horizontal. In particular, the mean curvature of $\Sigma$ is given by
\begin{equation}\label{eqn:H-vertical-cylinder}
	2H=\kappa_g-\langle\eta,\tfrac{1}{\mu}\nabla\mu\rangle.
\end{equation}
We can get rid of the term $\langle\eta,\tfrac{1}{\mu}\nabla\mu\rangle$ by considering a conformal factor in $M$.
			
\begin{proposition}\label{prop:H-invariant}
Let $\pi:\E\to M$ be a Killing submersion and let $\Gamma\subset M$ be a regular curve. The mean curvature of the vertical cylinder $\Sigma=\pi^{-1}(\Gamma)$ with respect to a unit normal $N$ satisfies 
\[2H=\mu\,\widetilde\kappa_g,\]
where $\widetilde\kappa_g$ is the geodesic curvature of $\Gamma$ with respect to the unit normal $\eta=\frac{1}{\mu}\pi_*(N)$ in the conformal metric $\mu^2\df s^2_M$ on $M$.
\end{proposition}
					
\begin{proof}
Since the computation is local, we can assume that $M$ is a disk of $\R^2$ endowed with the metric $\df s^2_\lambda=\lambda^2(\df x^2+\df y^2)$ for some conformal factor $\lambda$ in the usual coordinates $(x,y)$. The Levi-Civita connection of $\df s^2_\lambda$ is given by
\begin{equation}\label{prop:H-invariant:eqn1}
	\begin{aligned}
		\nabla_{\partial_x}\partial_x&=\tfrac{\lambda_x}{\lambda}\partial_x-\tfrac{\lambda_y}{\lambda}\partial_y,&\nabla_{\partial_x}\partial_y&=\tfrac{\lambda_y}{\lambda}\partial_x+\tfrac{\lambda_x}{\lambda}\partial_y,\\
		\nabla_{\partial_y}\partial_x&=\tfrac{\lambda_y}{\lambda}\partial_x+\tfrac{\lambda_x}{\lambda}\partial_y,&\nabla_{\partial_y}\partial_y&=-\tfrac{\lambda_x}{\lambda}\partial_x+\tfrac{\lambda_y}{\lambda}\partial_y.
	\end{aligned}
\end{equation}
Given the curve $\gamma=(x,y)$ that parametrizes $\Gamma$, after swapping $x$ and $y$ if necessary, we can assume that the frame $\{\partial_x,\partial_y\}$ is oriented so that 
\[\gamma'=x'\partial_x+y'\partial_y,\qquad\eta=\frac{-y'\partial_x+x'\partial_y}{\lambda((x')^2+(y')^2)^{1/2}}.\] 
On the one hand, taking into account~\eqref{prop:H-invariant:eqn1}, the geodesic curvature $\kappa_g$ of $\gamma$ (with respect to $\df s^2_\lambda$ and the unit normal $\eta$) can be computed as
\begin{equation}\label{prop:H-invariant:eqn2}
	\kappa_g=\frac{\langle\nabla_{\gamma'}\gamma',\eta\rangle}{|\gamma'|^2}=\frac{x'y''-x''y'}{\lambda((x')^2+(y')^2)^{3/2}}+\frac{\lambda_x y'-\lambda_y x'}{\lambda^2((x')^2+(y')^2)^{1/2}},
\end{equation}
where we have used that 
\begin{equation*}
	\begin{aligned}
	\nabla_{\gamma'}\gamma'&=\left( x''+\frac{\lambda_x}{\lambda}\left((x')^2-(y')^2\right)+2\frac{\lambda_y}{\lambda}x'y'\right)\partial_x\\
	&\qquad+\left( y''-\frac{\lambda_y}{\lambda}\left((x')^2-(y')^2\right)+2\frac{\lambda_x}{\lambda}x'y'\right)\partial_y.
\end{aligned}
\end{equation*}
On the other hand, we can also work out $\nabla\mu=\frac{1}{\lambda^2}(\mu_x\partial_x+\mu_y\partial_y)$ and hence
\begin{equation}\label{prop:H-invariant:eqn3}
	\langle\eta,\tfrac{1}{\mu}\nabla\mu\rangle=\frac{\langle J\gamma',\tfrac{1}{\mu}\nabla\mu\rangle}{|\gamma'|}=\frac{-\mu_x y'+\mu_y x'}{\mu\lambda((x')^2+(y')^2)^{1/2}}.
\end{equation}
Plugging~\eqref{prop:H-invariant:eqn2} and~\eqref{prop:H-invariant:eqn3} into~\eqref{eqn:H-vertical-cylinder}, we finally get
\begin{equation}\label{prop:H-invariant:eqn4}
	2H=\kappa_g-\langle\eta,\tfrac{1}{\mu}\nabla\mu\rangle=\frac{x'y''-x''y'}{\lambda((x')^2+(y')^2)^{3/2}}+\frac{(\lambda\mu)_x y'-(\lambda\mu)_y x'}{\lambda^2\mu((x')^2+(y')^2)^{1/2}}.
\end{equation}
Observe that $\widetilde\kappa_g$, the curvature of $\gamma$ with respect to the metric $\mu^2\df s_\lambda^2=\df s_{\lambda\mu}^2$ can be computed by substituting $\lambda$ with $\mu\lambda$ in~\eqref{prop:H-invariant:eqn2}, so it easily follows that the right-hand side in~\eqref{prop:H-invariant:eqn4} is nothing but $\mu\,\widetilde\kappa_g$.
\end{proof}
		
In the sequel we will use the prefix `$\mu$-' to indicate that the corresponding term is computed with respect to the metric $\mu^2\df s_M^2$ in $M$. For instance, Proposition~\ref{prop:H-invariant} implies that $\Sigma=\pi^{-1}(\Gamma)$ is minimal if and only if $\Gamma$ is a $\mu$-geodesic, and $\Sigma$ is mean convex with respect to $N$ if and only if $\Gamma$ is $\mu$-convex with respect to $\eta=\frac{1}{\mu}\pi_*N$.

The classification of $H$-surfaces invariant by any $1$-parameter group of isometries in 3-dimensional Killing submersions can be reduced by this argument to a problem for curves in the orbit space, which plays the role of base of the submersion. Since the local existence and uniqueness of curves with prescribed geodesic curvature is guaranteed (in an arbitrary surface) when some initial conditions have been fixed, we can also classify invariant $H$-surfaces by means of initial conditions.
		
\begin{corollary}\label{coro:invariant-H-surfaces}
Let $E$ be a 3-manifold with a Killing vector field $\xi$, and fix $H\in\R$. Given $q\in E$ with $\xi_q\neq 0$, let $\{v,n,\xi_q/\|\xi_q\|\}$ be an orthonormal basis of $T_qE$.
\begin{enumerate}[label=\emph{(\arabic*)}]
	\item There exists an $H$-surface invariant under the action of $\xi$ passing through $q$, tangent to $v$ with unit normal $N$ such that $N_q=n$.
	\item Any two surfaces satisfying item \emph{(1)} coincide in a neighborhood of $q$.
\end{enumerate}
\end{corollary}

It is also interesting to notice that radial $\mu$-geodesics at some point $p\in M$ produce an open book decomposition of a neighborhood of $p$, so the corresponding cylinders produce an open book decomposition by minimal surfaces of a neighborhood of $\pi^{-1}(\{p\})$. Let us make it precise in the following statement, which will be useful in the construction of local Scherk surfaces (see the proof of Proposition~\ref{prop:scherk}).		

\begin{corollary}\label{coro:open-book-decomposition}
Let $\pi:\E\to M$ be a Killing submersion and let $p\in M$. Given an open neighborhood $V$ of the origin in $T_pM$ where the $\mu$-exponential map is one-to-one, there exists an open book decomposition of $\pi^{-1}(O)$, where $O$ is the $\mu$-exponential image of $V$, by minimal cylinders with binding the fiber $\pi^{-1}(\{p\})$.
\end{corollary}

\subsection{The mean curvature of a vertical graph}
Let $\pi\colon\mathbb{E}\to M$ be a Killing submersion with Killing vector field $\xi$. Let $\Omega\subset M$ be an open subset and consider the zero section $F_0\colon \Omega\to\E$. We define the \emph{vertical} or \emph{Killing graph} of $u\in\mathcal{C}^0(\Omega)$ with respect to $F_0$ as the surface parametrized by
\begin{equation}\label{eqn:Fu}
F_u:\Omega\to\mathbb{E},\quad F_u(p)=\phi_{u(p)}(F_0(p)),
\end{equation}
where $\{\phi_t\}_{t\in\mathbb{R}}$ is the $1$-parameter family of vertical translations. Let us also define $d\in\mathcal{C}^\infty(\E)$ as the signed Killing distance along fibers from a point to $F_0$, which is implicitly determined by $\phi_{d(q)}(F_0(\pi(q))=q$ for all $q\in\E$.
			
\begin{lemma}[{\cite[Lemma~3.1]{LerMan}}]\label{lemma:H}
If $u\in\mathcal C^2(\Omega)$, the mean curvature function $H_u:\Omega\to\mathbb{R}$ of $F_u$ for the upward-pointing unit normal vector field $N_u$, is given by
\begin{equation}\label{eq:H}
	H_u=\frac{1}{2\mu}\,\mathrm{div}(\mu\,\pi_*(N_u))=\frac{1}{2\mu}\,\mathrm{div}\left(\frac{\mu^2\,Gu}{\sqrt{1+\mu^{2}\|Gu\|^2}}\right),
\end{equation}
where $Gu=\nabla u-\pi_*(\overline\nabla d)$, $\mathrm{div}$ denotes the divergence on $M$, and $\overline\nabla$ and $\nabla$ stand for the gradient operators in $\E$ and $M$, respectively.
\end{lemma}
			
In the sequel, we shall write $Z=\pi_*(\overline\nabla d)$, which is a vector field in $M$ not depending on $u$. Following~\cite{LerMan}, one has $\mathrm{div}(JZ)=-\frac{2\tau}{\mu}$, where $J$ is a $\frac\pi2$-rotation in the tangent bundle of $M$ defined by assuming that $\{(\df F_0)_p(v),(\df F_0)_p(Jv),\xi_{F_0(p)}\}$ is a positively oriented basis of $T_{F_0(p)}\E$ for all non-zero $v\in T_pM$ and $p\in M$. In particular, the information about the bundle curvature is encoded in $Z$. Notice that both $\nabla u$ and $Z$ depend on the choice of $F_0$ but the so-called \emph{generalized gradient} $Gu=\nabla u-Z$ only depends on the surface parametrized by $F_u$. 
			
Next lemma was firstly proved for minimal graphs in $\R^3$ by Finn~\cite{F} and Jenkins--Serrin~\cite{JS}, and later on generalized to many other ambient spaces including unit Killing submersions~\cite{LR}. We further extend it to general Killing submersions. In order to clarify the notation we will write $W_u=\sqrt{1+\mu^{2}\|Gu\|^2}$, which is the area element induced in $\Omega$ by $F_u$ (see also~\cite[\S2]{DLM}). Finally, note that the angle function of $F_u$ is positive everywhere and can be expressed as 
\begin{equation}\label{eqn:nu}
\nu=\langle N_u,\xi\rangle=\frac{\mu}{W_u}.
\end{equation}
			
\begin{lemma}\label{lemma:factorization}
	For any $u,v\in\mathcal{C}^1(\Omega)$, let $N_u$ and $N_v$ be the upward-pointing unit normal vector fields to $F_u$ and $F_v$, respectively. Then
	\[\left\langle\frac{Gu}{W_u}-\frac{Gv}{W_v},Gu-Gv\right\rangle=\frac{1}{2\mu^2}(W_u+W_v)\|N_u-N_v\|^2\geq 0.\]
	Equality holds at some point $p\in M$ if and only if $\nabla u(p)=\nabla v(p)$.
\end{lemma}
		
\begin{proof} On the one hand, we can write
\begin{equation}\label{lemma:factorization-eqn1}
\begin{aligned}
	\left\langle\frac{Gu}{W_u}-\frac{Gv}{W_v},Gu-Gv\right\rangle&=\frac{\|Gu\|^2}{W_u}-\langle Gu,Gv\rangle\left(\frac{1}{W_u}+\frac{1}{W_v}\right)+\frac{\|Gv\|^2}{W_v}\\
	&=\frac{W_u^2-1}{\mu^2W_u}-\langle Gu,Gv\rangle\frac{W_u+W_v}{W_uW_v}+\frac{W_v^2-1}{\mu^2W_v}\\
	&=\mu^{-2}(W_u+W_v)\left(1-\mu^2\frac{\langle Gu,Gv\rangle}{W_uW_v}-\frac{1}{W_uW_v}\right).
\end{aligned}
\end{equation}
On the other hand, since $\pi_*(N_u)=\frac{\mu\, Gu}{W_u}$ and $\langle N_u,\frac{\xi}{\mu}\rangle=\frac{1}{W_u}$, we can decompose $N_u-N_v$ in horizontal and vertical components and compute
\begin{equation}\label{lemma:factorization-eqn2}
\begin{aligned}
	\|N_u-N_v\|^2&=\left\|\frac{\mu \,Gu}{W_u}-\frac{\mu\,Gv}{W_v}\right\|^2+\left(\frac{1}{W_u}-\frac{1}{W_v}\right)^2\\
	&=\frac{W_u^2-1}{W_u^2}-2\mu^2\frac{\langle Gu,Gv\rangle}{W_uW_v}+\frac{W_v^2-1}{W_v^2}+\left(\frac{1}{W_u^2}-\frac{2}{W_uW_v}+\frac{1}{W_v^2}\right)\\
	&=2\left(1-\mu^2\frac{\langle Gu,Gv\rangle}{W_uW_v}-\frac{1}{W_uW_v}\right).
\end{aligned}
\end{equation}
Plugging~\eqref{lemma:factorization-eqn2} into~\eqref{lemma:factorization-eqn1}, we get the identity in the statement. Finally observe that $W_u+W_v>0$ and $\|N_u-N_v\|^2=0$ if and only if $\nabla u=\nabla v$.
\end{proof}

\section{Convergence of minimal graphs}\label{sec:convergence}

Convergence of minimal graphs plays a key role in the construction of solutions to the Jenkins--Serrin problem. We will consider a Killing submersion $\pi\colon\E\to M$ whose fibers have infinite length and a relatively compact domain $\Omega\subset M$ with piecewise smooth boundary, so there exists $\delta>0$ such that the set $\Omega(\delta)\subset M$ consisting of the points at distance less than $\delta$ from $\Omega$ is also relatively compact. From~\cite[Prop.~2.6]{DLM}, it follows that the sectional curvature of $\pi^{-1}(\Omega(\delta))\subset\E$ is bounded by a constant $\Lambda>0$ depending only on upper bounds for the Gauss curvature of $M$, $\tau$ and $\mu$ (and their first and second derivatives) on $\Omega(\delta)$. This is a key ingredient for the existence of gradient and curvature estimates.

A minimal graph is always stable because its angle function (which lies in the kernel of its stability operator) has no zeros. Stability implies curvature estimates, as proved by Schoen~\cite[Thm.~3]{Schoen} and Rosenberg, Souam and Toubiana~\cite[Thm.~2.5]{RST}. We will rewrite the latter in terms of distance in the base.

\begin{lemma}[{\cite[Thm.~2.5]{RST}}]\label{lem:curvature-estimate}
There exists $C$ depending only on $\delta^2\Lambda$ such that the norm of the shape operator of any minimal graph $\Sigma$ over $\Omega$ satisfies
\[|A(q)|\leq\frac{C}{\min\{d_\Sigma(q,\partial\Sigma),\frac{\pi}{2\Lambda},\delta\}}\leq\frac{C}{\min\{d_M(\pi(q),\partial\Omega),\frac{\pi}{2\Lambda},\delta\}},\quad\text{for all }q\in\Sigma,\]
where $d_\Sigma$ and $d_M$ are the distance functions in $\Sigma$ and $M$, respectively.
\end{lemma}

\begin{proof}
There is no loss of generality if we assume that $G=\pi^{-1}(\Omega)$ is relatively compact in $\E$ after considering the Riemannian quotient of $\E$ by any vertical translation. Note that the graphical condition (and hence stability) is not affected by this quotient. This also implies that $\pi^{-1}(\Omega(\delta))$ is relatively compact. We will prove that $G(\delta)$, the set of points of $\E$ at distance from $G$ less than $\delta$, coincides with $\pi^{-1}(\Omega(\delta))$, so the statement follows directly from~\cite[Thm.~2.5]{RST}.

Given $q\in\pi^{-1}(\Omega(\delta))$, there is some curve $\gamma$ in $M$ joining $\pi(q)$ and some $x\in\Omega$ whose length is less than $\delta$. Denote by $\widehat\gamma$ the horizontal lift of $\gamma$ (with respect to $\pi$) starting at $q$. Since the submersion is Riemannian, $\widehat\gamma$ has the same length as $\gamma$, and joins $q$ and some $q'\in\pi^{-1}(x)\subset G$. This proves the inclusion $\pi^{-1}(\Omega(\delta))\subset G(\delta)$. To prove the other inclusion, let $q\in G(\delta)$ and $q'\in G$ such that $d(q',q)<\delta$. The minimum distance from $p$ to the fiber $\pi^{-1}(\pi(q'))$ is realized by a geodesic $\widehat\gamma$ of $\E$ which is orthogonal to the fiber $\pi^{-1}(\pi(q'))$ at its endpoint. However, this implies that $\widehat\gamma$ is everywhere horizontal since the product $\langle\widehat\gamma',\xi\rangle$ is constant along $\widehat\gamma$ ($\widehat\gamma$ is a geodesic and $\xi$ is Killing). This means that $\gamma=\pi\circ\widehat\gamma$ is a curve in $M$ joining $\pi(q)$ and $\pi(q')\in\pi(G)=\Omega$ and the length of $\gamma$ is equal to the length of $\widehat\gamma$, so it is less than $\delta$. In particular, $q\in\pi^{-1}(\Omega(\delta))$.
\end{proof}
	
Towards the convergence results, we will first recall the interior gradient estimates in~\cite{RST} (although this result is given for $\mu\equiv 1$, its proof extends almost literally to the general case and will be omitted). Given a function $u\in\mathcal{C}^1(\Omega)$, observe that $\nabla u$ is bounded if and only if $Gu$ is bounded if and only if $W_u=(1+\mu^2\|Gu\|^2)^{1/2}$ is bounded if and only if $\nu=\mu W_u^{-1}$ is bounded away from zero.

\begin{lemma}[{\cite[Thm.~3.6]{RST}}]\label{lem:interior-gradient-estimate}
Given $C_1, C_2 > 0$, there exists a constant $C>0$, depending on $C_1$, $C_2$ and $\Lambda$, such that for any $p \in\Omega$ with $d(p, \partial \Omega) > C_2$ and for any solution $u\in\mathcal C^2(\Omega)$ of the minimal surface equation such that $|u|< C_1$ on $\overline\Omega$, we have that $W_u(p)\leq C$.
\end{lemma}

The following is a Harnack type estimate whose proof is inspired by~\cite[Thm.~4.4]{Younes}. We will denote by $B_M(p,R)$ the metric ball in $M$ centered at $p$ of radius $R$.

\begin{proposition}\label{prop:harnack}
Let $u$ be a $\mathcal C^2$-solution of the minimal surface equation on $\Omega$ such that $W_u(p)\leq C$ for some $p\in\Omega$. There exists a radius $R>0$, which only depends on $C$, $\delta$, $\Lambda$, $d(p,\partial\Omega)$ and the geometry of $M$, such that $W_u\leq 2C$ on  $B_M(p,R)$. 
\end{proposition}

\begin{proof}
Define $R$ as the minimum of $\frac{1}{2}d(p,\partial\Omega)$ and the injectivity radius of $\Omega$ at $p$. Since $\Omega$ is relatively compact in $M$, the injectivity radius is bounded below by a positive constant, so $R$ depends only on $d(p,\partial\Omega)$ and the geometry of $M$. We will work on the closed metric ball $\overline B_M(p,R)$, which is contained in $\Omega$ and obtain a suitable upper bound for $\|\nabla W\|$, where $W=W_u$. To this end, we will employ the orthonormal frame~\eqref{LocalFrame} in a local patch covering $\pi^{-1}(\overline B_M(p,R))$. Observe that $\overline B_M(p,R)$ is topologically a closed disk by assumption, whence there is a conformal parametrization of a neighborhood of $\overline B_M(p,R)$ in $M$. We will also assume the graph is parametrized by $(x,y)\mapsto(x,y,u(x,y))$, and define the (not necessarily orthonormal) frame 
	\begin{align*}
			e_1&=\partial_x+u_x\partial_z=\lambda E_1+\lambda \alpha E_3,\\
			e_2&=\partial_y+u_y\partial_z=\lambda E_2+\lambda \beta E_3,\\
			N&=\tfrac{-\alpha}{W}E_1+\tfrac{-\beta}{W}E_2+\tfrac{1}{W}E_3,
	\end{align*}
where $\alpha=\mu\left(\frac{u_x}{\lambda}-a\right)$ and $\beta=\mu\left(\frac{u_y}{\lambda}-b\right)$. Notice that $e_1$ and $e_2$ are tangent to the graph and $N$ is the upward-pointing unit normal. Differentiating $N$ with respect to $e_1$, it follows that $\overline\nabla_{e_1}N=U+V_1+V_2+V_3$, where
\begin{align*}
	U&=\left(-\bigl(\tfrac{\alpha}{W}\bigr)_x-\tfrac{\lambda_y}{\lambda}\tfrac{\beta}{W}\right)E_1+\left(-\bigl(\tfrac{\beta}{W}\bigl)_x+\tfrac{\lambda_y}{\lambda}\tfrac{\alpha}{W}\right)E_2+\bigl(\tfrac{1}{W}\bigr)_xE_3,\\
	V_1&=\tfrac{1}{W}\left(-\lambda\tau\alpha\beta E_1+\lambda\tau(\alpha^2-1)E_2+\lambda\tau\beta E_3\right),\\
	V_2&=\tfrac{-\alpha}{W}\tfrac{\mu_x}{\mu}\left(E_1+\alpha E_3\right),\\
	V_3&=\tfrac{-\alpha}{W}\tfrac{\mu_y}{\mu}\left(E_2+\beta E_3\right).
\end{align*}
Observe that $U$ contains the second derivatives of $u$, whereas $V_1$, $V_2$ and $V_3$ are lower-order terms. Cauchy--Schwarz inequality easily yields
\begin{equation}\label{thm:harnack:eqn1}
\|U\|^2\leq4\left(\|\overline{\nabla}_{e_1}N\|^2+\|V_1\|^2+\|V_2\|^2+\|V_3\|^2\right).
\end{equation}
Furthermore, we can estimate
\begin{equation}\label{thm:harnack:eqn2}
\begin{aligned}&\|U\|^2=\frac{\alpha_x^2+\beta_x^2}{W^4}+\left(\frac{\alpha_x\beta-\alpha\beta_x}{W^2}+\frac{\lambda_y}{\lambda}\right)^2-\frac{\lambda_y^2}{\lambda^2W^2}\geq \frac{\alpha_x^2+\beta_x^2}{W^4}-\frac{\lambda_y^2}{\lambda^2W^2},\\
&\|V_1\|^2\leq\lambda^2\tau^2W^2,\qquad\qquad\|V_2\|^2\leq\frac{\mu_x^2}{\mu^2}W^2,\qquad\qquad\|V_3\|^2\leq \frac{\mu_y^2}{\mu^2}W^2.
\end{aligned}\end{equation}
Plugging~\eqref{thm:harnack:eqn2} into~\eqref{thm:harnack:eqn1}, it follows that
\begin{equation}\label{thm:harnack:eqn3}
\frac{\alpha_x^2+\beta_x^2}{4W^4}\leq\|\overline\nabla_{e_1}N\|^2+\lambda^2W^2\left(\tau^2+\frac{\|\nabla\mu\|^2}{\mu^2}\right)+\frac{\lambda_y^2}{4W^2\lambda^2}.\end{equation}

Since the shape operator $A$ of the graph satisfies $Ae_1=-\overline\nabla_{e_1}N$, the curvature estimate given by Lemma~\ref{lem:curvature-estimate} yields the existence of $C_1>0$ (only depending on $d(p,\partial\Omega)$ and the geometry of $\pi^{-1}(\Omega(\delta))$) such that
\begin{equation}\label{thm:harnack:eqn4}\|\overline\nabla_{e_1}N\|^2=\|Ae_1\|^2\leq C_1\|e_1\|^2=C_1\lambda^2(1+\alpha^2)\leq C_1\lambda^2W^2
\end{equation}
on $B_M(p,R)$. By combining~\eqref{thm:harnack:eqn3} and~\eqref{thm:harnack:eqn4}, we find a constant $C_2>0$ (only depending on $d(p,\partial\Omega)$ and derivatives of $\tau$ and $\mu$) such that
\begin{equation}\label{thm:harnack:eqn5}
\frac{\alpha_x^2+\beta_x^2}{\lambda^2}\leq C_2W^6+\frac{\lambda_y^2W^2}{\lambda^4}.
\end{equation}
Similarly, we work out $\overline\nabla_{e_2}N$ to obtain the estimate
\begin{equation}\label{thm:harnack:eqn6}
\frac{\alpha_y^2+\beta_y^2}{\lambda^2}\leq C_2W^6+\frac{\lambda_x^2W^2}{\lambda^4}.
\end{equation}
Note that $\nabla W=\frac{\alpha\alpha_x+\beta\beta_x}{\lambda W}\frac{\partial_x}{\lambda}+\frac{\alpha\alpha_y+\beta\beta_y}{\lambda W}\frac{\partial_y}{\lambda}$, so Cauchy--Schwarz inequality along with~\eqref{thm:harnack:eqn5} and~\eqref{thm:harnack:eqn6} gives
\[\|\nabla W\|^2\leq\frac{\alpha_x^2+\beta_x^2+\alpha_y^2+\beta_y^2}{\lambda^2}\leq 2C_2W^6+\frac{\|\nabla\lambda\|^2W^2}{\lambda^2}\leq\left(2C_2+\frac{\|\nabla\lambda\|^2}{\lambda^2}\right) W^6,\]
(observe that $W^2\leq W^6$ because $W>1$). Since $\lambda^{-2}\|\nabla\lambda\|^2$ is bounded in $\overline B_M(p,R)$, it has a maximum. Define $I(p)$ as the infimum of these maxima for all conformal patches of $M$ covering $\overline B_M(p,R)$, which is continuous and only depends on $d(p,\partial\Omega)$ and the geometry of $M$. All in all, we conclude from~\eqref{thm:harnack:eqn6} that there exists $C_3$ depending on $C$, $\delta$, $\Lambda$, $d(p,\partial\Omega)$ and the geometry of $M$ such that $\|\nabla W\|\leq C_3W^3$.

Finally, let $\gamma:[0,R]\to M$ be a unit-speed geodesic of $M$ with $\gamma(0)=p$, and define the function $f(s)=W(\gamma(s))$. Therefore $f(0)=W(p)\leq C$ and 
\[f'(s)=\langle\nabla W,\gamma'(s)\rangle\leq\|\nabla W\|\leq C_3f(s)^3.\]
Integrating this differential inequality, we obtain $f(s)\leq 2C$ for $s\in[0,\frac{3}{8C^2C_3})$. This means that $W\leq 2C$ on $B_M(p,R')$, where $R'=\min(R,\frac{3}{8C^2C_3})$. 
\end{proof}
	
\emph{A-priori} $\mathcal{C}^0$-estimates yield \emph{a-priori} $\mathcal{C}^1$-estimates by Lemma~\ref{lem:interior-gradient-estimate} and Proposition~\ref{prop:harnack}, and this is enough to get $\mathcal{C}^\infty$-convergence: Ladyzhenskaya-Ural'tseva theory implies \emph{a-priori} $\mathcal{C}^{1,\alpha}$-estimates~\cite{LU}, and Schauder theory in turn implies \emph{a-priori} $\mathcal{C}^{2,\alpha}$-estimates~\cite{GT}; since the minimal surface equation is of second order, standard elliptic \textsc{pde} arguments also provide $\mathcal{C}^k$-estimates for all $k\in\N$. The classical Arzelà-Ascoli theorem yields the following convergence result.
	
\begin{theorem}[Compactness]\label{thm:compactness}
Let $\{u_n\}$ be a $\mathcal{C}^0$-uniformly bounded sequence of smooth solutions of the minimal surface equation in $\Omega$. There exists a subsequence of $\{u_n\}$ converging (in the $\mathcal{C}^k$-topology on compact subsets for all $k\in\N$) to a solution of the minimal surface equation in $\Omega$.
\end{theorem}
	
For completeness, we give a proof of the Monotone Convergence Theorem. The sets $\mathcal U$ and $\mathcal V$ in the statement are typically called the \emph{convergence} and \emph{divergence} sets, respectively (not to be confused with the divergence lines we introduce in Section~\ref{sec:divergence-lines}). Notice that considering a subsequence is not necessary by monotonicity.

\begin{theorem}[Monotone Convergence]\label{thm:MonConT}
Let $\{u_n\}$ be a nondecreasing sequence of smooth solutions to the minimal surface equation in $\Omega$. There exists an open subset $\mathcal U\subset\Omega$ such that $\{u_n\}$ converges in $\mathcal U$ (in the ${\mathcal{C}}^k$-topology on compact subsets of $\mathcal U$ for all $k\in\N$) and diverges uniformly to $+\infty$ on compact subsets of $\mathcal{V}=\Omega\sm\mathcal{U}$.  	
\end{theorem}

\begin{proof} 
We can assume that there exists $p\in\Omega$ such that $\left\{u_n\right\}$ is uniformly bounded, for otherwise the sequence clearly diverges to $+\infty$ on all compact subsets of $\Omega$, so the statement follows. In fact, it suffices to prove that $\{u_n\}$ converges on $B_M(p,R)$ for some $R$, where we can assume that the functions are positive after adding suitable constants. By Lemma~\ref{lem:interior-gradient-estimate}, there are constants $C_n$ such that $W_{u_n}(p)\leq C_n$, and Proposition~\ref{prop:harnack} implies that $W_{u_n}\leq 2C_n$ in $B_M(p,R_n)$, with $R_n$ depending on $u_n(p)$, $C_n$, $d(p,\partial \Omega)$ and the geometry of $\pi^{-1}(\Omega(\delta))$. Since $\{u_n(p)\}$ is bounded, $C_n$ and $R_n$ can be chosen not depending on $u_n(p)$, so we can assume that there are $C,R'>0$ depending only on $d(p,\partial \Omega)$ and the geometry of $\pi^{-1}(\Omega(\delta))$, such that $W_{u_n}\leq 2C$ on $B_M(p,R')$ for all $n$. This means that $\|Gu_n\|$ (and hence $\|\nabla u_n\|$ and $u_n$) is uniformly bounded on $B_M(p,R')$. Theorem~\ref{thm:compactness} therefore implies that $\{u_n\}$ admits a subsequence converging to a solution of the minimal surface equation on $B_M(p,R')$. Since the sequence is monotone, we deduce that $\{u_n\}$ itself converges to such solution. The fact that the divergence is uniform on compacts of $\mathcal V$ is also a consequence of the monotonicity.
\end{proof}
	
We will also need a couple of technical results in our arguments. The first one is a lemma that ensures that infinite values can only be prescribed along a $\mu$-geodesic. The proof is standard and will be ommited (it is similar to that of~\cite[Thm.~3.3]{RST} substituting geodesics with $\mu$-geodesics just because $\mu$-geodesics are the projections of vertical minimal cylinders in the general case).

\begin{lemma}\label{lemma:boundary-structure}
Let $\Sigma$ be a minimal graph over $\Omega$ given by $u\in\mathcal C^\infty(\Omega)$ with respect to $F_0$. Assume that $\gamma\subset\partial\Omega$ is a regular open arc such that $\lim\{u(p_n)\}=\pm\infty$ for all sequences $\{p_n\}$ of points in $\Omega$ converging to any $p\in\gamma$. Then $\gamma$ is $\mu$-geodesic and the angle function of $\Sigma$ goes to $0$ along any sequence approaching a point of $\gamma$.
\end{lemma}

The second one is a consequence of the curvature estimates also known as the Uniform Graph Lemma, which we will state for graphs in Killing submersions. Since graphs admit curvature estimates only depending on $\Lambda$, $\delta$ and the distance to the boundary (see Lemma~\ref{lem:curvature-estimate}), we can rewrite~\cite[Prop.~4.3]{RST} in a more convenient way. Indeed, in the proof given in~\cite{RST} it is shown that such a graph has uniformly bounded Euclidean second fundamental form in harmonic coordinates, so we can also use~\cite[Lemma~4.1.1]{PerezRos} to ensure that the growth of the graph is under control as stated in item (2) of the following lemma:

\begin{lemma}[{\cite[Prop.~4.3]{RST}}]\label{lem:uniform-graph-lemma}
Let $\Sigma$ be a minimal graph over $\Omega$ and let $q\in\Sigma$. There exist constants $a,\rho,\rho_0>0$ (depending on $\delta$, $\Lambda$, $d(q,\partial\Sigma)$ and on a positive lower bound for the injectivity radius of $\Omega$) and an open neighborhood $U_q\subset\E$ of $q$ that can be parametrized by harmonic coordinates, such that:
\begin{enumerate}[label=\emph{(\arabic*)}]
	\item A subset $\Sigma_q\subset\Sigma\cap U_q$ containing $q$ is an Euclidean graph (in the harmonic coordinates) over the disk of $D(0,\rho)\subset T_q\Sigma$ of Euclidean radius $\rho$.
	\item If $f\in\mathcal C^\infty(D(0,\rho))$ is the function that defines the Euclidean graph, then $f(v)\leq a|v|^2$ for all $v\in D(0,\rho)$, where $|\cdot|$ is the Euclidean norm in $T_q\Sigma$.
	\item The subset $\Sigma_q$ contains the geodesic disk $B_\Sigma(q,\rho_0)$.
\end{enumerate}
\end{lemma}

\section{The Dirichlet problem with finite boundary values}\label{sec:dirichlet}

We will now show the existence and uniqueness of minimal graphs in $\E$ with finite boundary values over a relatively compact domain $\Omega\subset M$ with piecewise $\mu$-convex boundary. If $\Omega$ is simply connected with convex corners, this reduces to a Plateau problem where Meeks--Yau solution applies~\cite{MY1}, while for a general domain we will apply Perron's method. Some related results on existence of minimal graphs in Killing submersions can be found in~\cite{ADR,DD,DDR,KM}.

\begin{definition}\label{def:nitsche}
A {\em Nitsche contour} in $\mathbb{E}$ is a pair $(\Omega,\Gamma)$, where $\Omega\subset M$ is a relatively compact domain and $\Gamma$ is the union of finitely many piece regular curves, each of which admits a one-to-one continuous parametrization $\gamma:[a,b]\to\Gamma$ satisfying the following two conditions:
\begin{enumerate}[label=(\alph*)]
	\item There exists a partition $a=t_1<s_1\leq t_2<\ldots\leq t_r<s_r\leq t_{r+1}=b$ such that $\gamma(a)=\gamma(b)$ and, for any $j\in\{1,\ldots,r\}$, the component $\gamma_{|[t_j,s_j]}$ projects one-to-one by $\pi:{\E}\to M$, and $\gamma_{|[s_j,t_{j+1}]}$ is a vertical segment.

	\item The curve $\Gamma$ is a graph over $\partial\Omega$ except possibly a at the points of the form $\pi(t_i)$, which will be called the \emph{vertices} of $\Omega$.
\end{enumerate}
\end{definition}
	
It is important to recall that the fibers of $\pi$ are assumed to have infinite length, and graphs over $\Omega$ will be parametrized with respect to the fixed smooth section $F_0$. We start by proving a maximum principle.
	
\begin{proposition}\label{prop:uniqueness}
Let $(\Omega,\Gamma)$ and $(\Omega,\Gamma')$  be two Nitsche contours in $\E$ with the same domain $\Omega$ and the same vertex set $V\subset\partial\Omega$. Let $u,v\in\mathcal{C}^\infty(\Omega)$ such that
\begin{itemize}
	\item[i)] the mean curvatures of their graphs satisfy the inequality $H_u\geq H_v$,
	\item[ii)] both $u$ and $v$ extend continuously to $\overline\Omega\sm V$ giving rise to surfaces with boundaries $\Gamma$ and $\Gamma'$, respectively.
\end{itemize}
If $u\leq v$ in $\partial\Omega\sm V$, then $u\leq v$ in $\Omega$.
\end{proposition}
	
\begin{proof}
Reasoning by contradiction, consider $w=u-v$ and assume that $ U=\{p\in\Omega:w(p)>0\}$ is not empty. By adding a small enough positive constant to $v$ so the condition $ U\neq\emptyset$ is preserved, we can assume, without loss of generality, that $\nabla w$ does not vanish along $\partial U$ and $u<v$ in $\partial\Omega\sm V$. Therefore, $\partial U$ is a family $\{C_\alpha\}$ of regular curves without intersection points. The maximum principle prevents the existence of any connected component of $U$ whose boundary is contained in the interior of $\Omega$. Moreover, the conditions $u<v$ in $\partial\Omega\smallsetminus V$ and $\nabla w\neq 0$ in $\partial U$ ensure that each $C_\alpha$ starts and ends in the vertex set $V\subset \partial\Omega$.

We will consider a connected component of $\widetilde{U}\subset U$ and, given $\varepsilon>0$, we will denote by $\widetilde{U}_\varepsilon$ the set of points of $\widetilde{U}$ which are not in the geodesic balls of radius $\varepsilon$ with centers in $V$. For $\varepsilon>0$ small enough, the discussion in the previous paragraph allows us to write $\partial\widetilde{U}_\varepsilon=\Gamma^1_\varepsilon\cup \Gamma^2_\varepsilon$, where $\Gamma^1_\varepsilon\subset \partial\widetilde{U}$ consists of finitely many curves and $\Gamma^2_\varepsilon$ is constituted by arcs of geodesic circles centered at the points of $V$.

Since the functions $u$ and $v$ satisfy $H_u\geq H_v$ in $\Omega$, we get from Lemma \ref{lemma:H} that $\mathrm{div}\frac{\mu^2\, Gu}{W_u}\geq \mathrm{div}\frac{\mu^2\, Gv}{W_v}$ in $\Omega$. The divergence theorem yields
\begin{equation}\label{thm:eq1}
	0\leq\int_{\widetilde{U}_\varepsilon}\mathrm{div}\left(\frac{\mu^2\, Gu}{W_u}-\frac{\mu^2\, Gv}{W_v}\right)=\int_{\partial\widetilde{U}_\varepsilon}\mu^2\left\langle\frac{Gu}{W_u}-\frac{Gv}{W_v},\eta\right\rangle,
\end{equation}
where $\eta$ is the outer unit conormal vector field to $\widetilde{U}_\varepsilon$ along its boundary. On the other hand, Lemma \ref{lemma:factorization} guarantees that
\begin{equation}\label{thm:eq2}
	\left\langle\frac{Gu}{W_u}-\frac{Gv}{W_v},\nabla w\right\rangle=\frac{1}{2\mu^2}(W_u+W_v)\|N_u-N_v\|>0\quad\text{on }\Gamma^1_\varepsilon,
\end{equation}
where $N_u$ and $N_v$ stand for the downward unit vector fields, normal to $F_u$ and $F_v$, respectively. The last strict inequality holds because $\nabla w\neq 0$ along $\Gamma^1_\varepsilon$. Nevertheless, since $w=0$ in $\Gamma^1_\varepsilon$ and $w>0$ in $U_\varepsilon$, the vector $\nabla w$ is a negative multiple of $\eta$ along $\Gamma^1_\varepsilon$. Hence the functions
\begin{equation}\label{thm:eq3}
\alpha_i(\varepsilon)=\int_{\Gamma^i_\varepsilon}\mu^2\left\langle\frac{Gu}{W_u}-\frac{Gv}{W_v},\eta\right\rangle,\qquad i\in\{1,2\},
\end{equation}
satisfy $\lim_{\varepsilon\to 0}\alpha_1(\varepsilon)<0$ by Equation~\eqref{thm:eq2}, whereas $\lim_{\varepsilon\to0}\alpha_2(\epsilon)=0$ since the integrand in $\alpha_2(\epsilon)$ is bounded by Cauchy--Schwarz inequality and the length of $\Gamma^2_\varepsilon$ tends to zero as $\varepsilon\to 0$. Consequently, $\alpha_1(\varepsilon)+\alpha_2(\varepsilon)<0$ for some small $\varepsilon$, contradicting the fact that $\alpha_1(\varepsilon)+\alpha_2(\varepsilon)\geq 0$ by Equation~\eqref{thm:eq1}.
\end{proof}

\subsection{The simply connected case}
In order to prove existence and uniqueness of the Plateau problem over a simply connected domain, we have to understand first which domains are admissible. Motivated by the mean-convex property in the general existence result of Meeks and Yau~\cite{MY1}, we give the following definition.

\begin{definition}
We say that a Nitsche contour $(\Omega,\Gamma)$ is  \emph{admissible} if $\Omega$ is a simply-connected relatively compact set satisfying the following properties.
\begin{enumerate}[label=(\alph*)]
	\item Each component of $\partial\Omega$ is a regular curve with non-negative $\mu$-geodesic curvature with respect to the inner conormal to $\Omega$.
	\item The interior angle at each vertex of $\partial\Omega$  is at most $\pi$.
\end{enumerate}
\end{definition}

Indeed, it follows from Proposition~\ref{prop:H-invariant} that an admissible Nitsche contour satisfies the assumption of~\cite[Thm.~1]{MY1}. Notice that in~\cite{MY1} the sign of the mean curvature is opposite with respect to ours.

\begin{theorem}\label{thm:existence}
Let $(\Omega,\Gamma)$ be an admissible Nitsche contour in $\E$. There exists a unique minimal surface $\Sigma\subset\pi^{-1}(\overline\Omega)$ with boundary $\Gamma$. Moreover, $\Sigma$ is topologically a disk whose interior is a graph over $\Omega$.
\end{theorem}
	
\begin{proof}
The existence a minimal disk in $\pi^{-1}(\Omega)$ with boundary $\Gamma$ is guaranteed by~\cite[Thm.~1]{MY1}. To prove uniqueness, it suffices to check that every minimal surface $\Sigma\subset\pi^{-1}(\Omega)$ with boundary $\Gamma$ is a graph, and then apply Proposition~\ref{prop:uniqueness}. 

If the Nitsche contour does not contain vertical segments, this assertion follows easily from the maximum principle, so we will assume that there are vertical segments. A slight deformation of these segments gives a new boundary $\Gamma'$ above $\Gamma$ that projects one-to-one onto $\partial\Omega$ (as in Figure~\ref{img:deformation}, center). The unique solution $\Sigma'$ to the Plateau problem in $\pi^{-1}(\Omega)$ with boundary $\Gamma'$ is a graph and must lie above $\Sigma$ due to the maximum principle again. If we consider a decreasing sequence $\{\Gamma_n\}$ of such deformed curves (i.e., $\Gamma_n$ lies above $\Gamma_{n+1}$ for every $n\in\N$) tending to the original Nitsche contour $\Gamma$, the corresponding solutions $\Sigma_n$ tend to a solution $\Sigma_+$ with boundary $\Gamma$ (note that this sequence is decreasing and uniformly bounded on $\Omega$ so Theorem~\ref{thm:MonConT} applies). As all the surfaces $\Sigma_n$ lie above $\Sigma$, so does $\Sigma_+$. 
\begin{figure}
	\centering\includegraphics[width=0.95\textwidth]{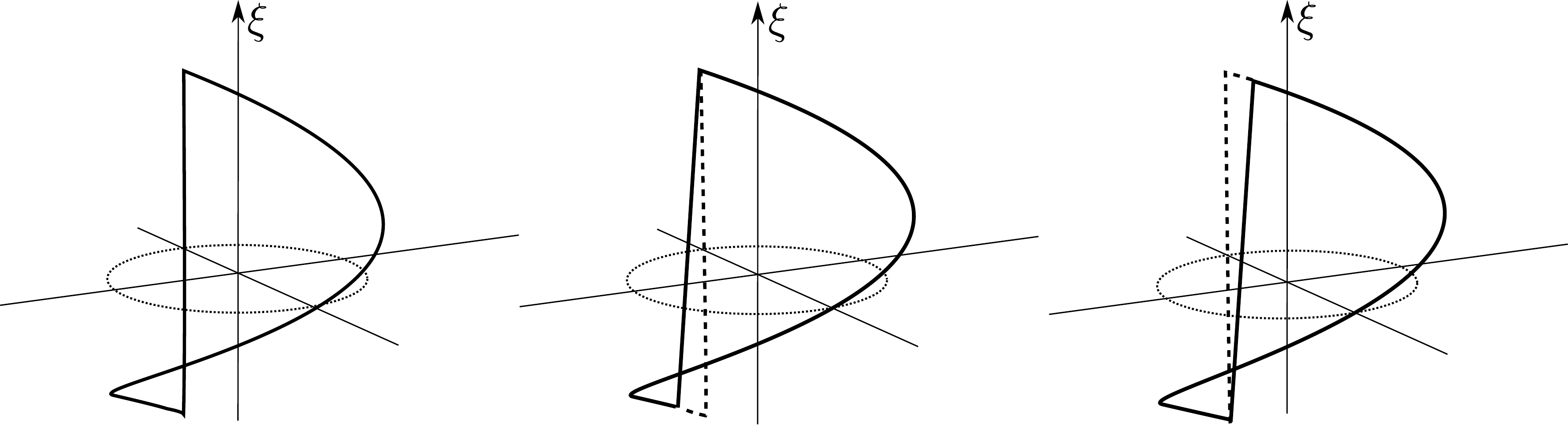}
	\caption{Deformation of a Nitsche graph (left) from above (center) and from below (right).}\label{img:deformation}
\end{figure}
Likewise, it is possible to deform the boundary $\Gamma$ so the new boundary lies below the original one and construct a graph $\Sigma_-$ with boundary $\Gamma$ and lying under $\Sigma$ (see Figure~\ref{img:deformation}, right). Proposition~\ref{prop:uniqueness} implies that $\Sigma_-=\Sigma_+$. Since $\Sigma$ lies between $\Sigma_+$ and $\Sigma_-$, we conclude that $\Sigma=\Sigma_+=\Sigma_-$ is a graph.
\end{proof}
	
\subsection{The general case}\label{subsec:perron}
We will use the Perron process to extend Theorem~\ref{thm:existence} to domains that are not necessarily simply connected or whose boundaries have reentrant corners. This method is rather well known (e.g. it was applied originally by Jenkins and Serrin~\cite{JS}) so we will just sketch it here in the Killing-submersion setting, for the sake of completeness. We will essentially follow Sa Earp and Toubiana's approach~\cite{ST1,ST3}, see also~\cite{Ng,NST}. Our goal is to solve the Dirichlet problem
	\begin{equation}
		\label{equation:dirichlet}
		P(\Omega,f):\left\{\begin{array}{l}
			H_u=0 \text{ in }\Omega,\\ 
			u|_{\partial\Omega}=f, 
		\end{array}\right.
	\end{equation} 
where $f$ is a continuous function.

Given $u\in\mathcal C^0(\Omega)$ and $U\subset\Omega$ a small closed convex disk, we will denote by $\tilde{u}_U$ the unique minimal graph over $U$ that solves the Dirichlet problem on $U$ with the same values as $u$ on $\partial U$, which exists by Theorem~\ref{thm:existence} via approximating $u|_{\partial U}$ by functions in $\mathcal C^1(\partial U)$. We also define $M_{U,u}\in\mathcal C^0(\overline\Omega)$ as
\[M_{U,u}(p)=
	\begin{cases}
		u(p), & \text{if } p\in \overline{\Omega}\sm U,\\
		\tilde{u}_U(p), & \text{if } p\in U.
	\end{cases}\] 
We say that $u\in\mathcal{C}^0(\overline{\Omega})$ is a \emph{subsolution} (resp. \emph{supersolution}) for $P(\Omega,f)$ if for any  small closed  disk $U\subset \Omega$, we have $u\leq M_{U,u}$ (resp.\ $u\geq M_{U,u}$),
and $u|_{\partial \Omega}\leq f$ (resp.\ $u|_{\partial \Omega}\geq f$). Due to the ellipticity of the mean curvature equation, it easily follows that $u\in\mathcal{C}^2(\Omega)$ is a subsolution (resp.\ supersolution) if and only if $H_u\geq 0$ (resp.\ $H_u\leq 0$). Consequently, a solution $u\in\mathcal{C}^2(\Omega)$ is both a subsolution and supersolution. This fact will be used later to obtain subsolutions.

We also need to bring up the notion of barrier. We say that $p_0\in\partial\Omega$ admits an \emph{upper barrier} (resp.\ \emph{lower barrier}) for $P(\Omega,f)$ if for any constant $M_0>0$ and any $k\in\mathbb N$, there exist an open neighborhood $V_k$ of $p_0$ in $M$ and a function $\omega^+_k$ (resp.\ $\omega^-_k$) of class $\mathcal{C}^2(V_k\cap\Omega)\cap\mathcal{C}^0(\overline{V_k\cap \Omega})$ such that
\begin{enumerate}[label=(\roman*)]
	\item $\omega^+_k\geq f$ (resp.\ $\omega^-_k\leq f$) on $\partial \Omega\cap V_k$,
	
	\item $\omega^+_k\geq M_0$ (resp.\ $\omega^-_k\leq -M_0$) on $\Omega\cap \partial V_k$,
	
	\item $H_{\omega^+_k}\leq 0$ (resp.\ $H_{\omega^-_k}\geq 0$) on $\Omega\cap V_k$,
		
	\item $\lim_{k\to\infty}\omega^+_k(p_0)=f(p_0)$
		(resp.  $\lim_{k\to\infty}\omega^-_k(p_0)=f(p_0)$).
\end{enumerate}
This is motivated by the following well known result.
	
\begin{lemma}[Perron Process]\label{lemma:perron}
Let $\Omega\subset M$ be an open domain with piecewise regular boundary and assume that $f:\partial\Omega\to\R$ is continuous on each component of $\partial\Omega$ and has left and right limits at each vertex of $\Omega$. Assume that $P(\Omega,f)$ has a supersolution $\phi$ and let $\mathcal{S}_{\phi}$ the set of subsolutions $\varphi$ of $P(\Omega,f)$ such that $\varphi\leq\phi$.
\begin{enumerate}[label=\emph{(\arabic*)}]
	\item  If ${\mathcal S}_\phi\neq\emptyset$, then the function $u(p)=\sup\{v(p):v\in {\mathcal S}_\phi\}$ is of class $\mathcal C^2(\Omega)$ and satisfies the equation $H_u=0$ in $\Omega$.
	\item If $\Omega$ is bounded and $\partial\Omega$ admits upper and lower barriers at some regular point $p_0\in\partial\Omega$ for the problem $P(\Omega,f)$, then the above solution $u$ extends continuously at $p_0$ by setting $u(p_0)=f(p_0)$.
\end{enumerate}
\end{lemma}

We will now prove the existence result in $\mathbb{E}$.

\begin{theorem}\label{thm:generalized-existence}
Let $\Omega\subset M$ be a relatively compact domain with piecewise regular boundary $\partial\Omega=\cup_{i\in I}C_i$, where each $C_i$ is either a $\mu$-geodesic arc (possibly closed) or a $\mu$-convex curve. Let $f_i\in\mathcal{C}^0(C_i)$ be bounded functions. There exists a unique minimal solution $u$ on $\Omega$ such that $u=f_i$ on the interior of each $C_i$.
\end{theorem}	

\begin{proof}
Observe that this Dirichlet problem has subsolutions and supersolutions because the functions $f_i$ are bounded and there are minimal sections over $\Omega$ by Lemma~\ref{lemma:minimal-sections}. Therefore, it suffices to show the existence of barriers at any point $p_0$ in the interior of an arc $C_i\subset\partial\Omega$ where $f_i\in\mathcal C^0(C_i)$ has been assigned. 

Let $T$ be a small $\mu$-geodesic triangle contained in a convex neighborhood of $p_0$ with sides $\ell_1$, $\ell_2$ and $\ell_3$ such that $\ell_1$ is tangent to $C_i$ at $p_0$ and $T\cap \Omega$ is a neighborhood of $p_0$ in $\Omega$ (if $\partial\Omega$ is $\mu$-geodesic near $p_0$, then $\ell_1$ can contain an open subarc of $\partial\Omega$). We can choose $T$ sufficiently small such that there is a neighborhood $V_k$, $k\in\mathbb{N}$, of $p_0$ with boundary $\ell_2\cup\ell_3$ and a $\mu$-convex arc in $M$ joining the endpoints of $\ell_1$. We can also assume that $\partial\Omega\cap V_k$ a connected subarc of $C$ and $\Omega\cap\partial V_k$ is a connected subset of $\ell_2\cup\ell_3$. Theorem~\ref{thm:existence} guarantees the existence of a minimal graph $\omega^+_k$ over $T$ with boundary values $f_i(p_0)+\frac{1}{k}$ on $\ell_1$ and $\max\{M_0,f_i(p_0)+\tfrac{1}{k}\}$ on $\ell_2\cup\ell_3$. As $F_0$ can be chosen as a minimal section, the maximum principle implies that $\omega_k^+$ induces a function in $\mathcal C^2(\Omega\cap V_k)\cap\mathcal C^0(\overline{\Omega\cap V_k})$ by restriction that becomes an upper barrier at $p_0$. Likewise, we can define a lower barrier $\omega_k^-$ using the minimal graph over $T$ with boundary values $f_i(p_0)-\frac{1}{k}$ on $\ell_1$ and $\min\{-M_0,f_i(p_0)-\frac{1}{k}\}$ on $\ell_2\cup\ell_3$.
\end{proof}

\section{Local barriers}\label{sec:local-barriers}
	
We would like to obtain Scherk-type minimal surfaces on \emph{small} $\mu$-geodesic triangles $T\subset M$ that will serve as local barriers in our Jenkins--Serrin constructions. We will denote by $p_1,p_2,p_3$ the vertices of $T$ and by $\ell_1$, $\ell_2$ and $\ell_3$ the corresponding opposite geodesic sides. Corollary~\ref{coro:open-book-decomposition} yields the existence of a relatively compact open neighborhood $U_i$ of $p_i$, where there is an open book decomposition by $\mu$-geodesics with binding at $p_i$. We will say that $T$ is \emph{small} whenever $T\subset U$, where $U=U_1\cap U_2\cap U_3$, and all interior angles of $T$ are at most $\pi$ (notice that such a triangle $T$ exists around any point $p\in M$ as long as we choose $p_1,p_2,p_3$ in a totally $\mu$-convex neighborhood of $p$).
	
\begin{proposition}\label{prop:scherk} 
There exists a minimal graph over $T$ with zero value (with respect to $F_0$) on $\ell_2\cup\ell_3$ and asymptotic value $\pm\infty$ on $\ell_1$. Moreover, the tangent planes of $\Sigma$ become vertical when approaching any interior point of $\ell_1$.
\end{proposition}

\begin{proof}
Assume that the boundary value on $\ell_1$ is $+\infty$, since the case of $-\infty$ is analogous. For any $n$, the existence of a minimal solution $u_n$ on $T$ with value $0$ on $\ell_2\cup\ell_3$, and value $n$ on $\ell_1$ is guaranteed by Theorem~\ref{thm:existence}. By Proposition~\ref{prop:uniqueness}, the sequence $\left\{u_n\right\}$ is nondecreasing and positive. Hence, to show that the limit $u=\lim_{n\to\infty}u_n$ exists, it is sufficient to prove that $\left\{u_n\right\}$ is uniformly bounded on any compact subset $K\subset\overline{T}\sm\ell_1$ and then apply Theorem~\ref{thm:MonConT}. Lemma~\ref{lemma:boundary-structure} implies the last assertion of the statement.

Denote by $\Sigma_n$ the graph of $u_n$. We will avoid the customary use of Douglas criterion by building a sequence of minimal disks $\left\{D_k\right\}$ such that
\begin{enumerate}
	\item $D_k$ is above $\Sigma_n$ for all $n$ and $k$, and
	\item the family of the horizontal projections $\left\{\pi(D_k)\right\}$ exhausts $T$ as $k\to\infty$.
\end{enumerate} 
The existence of $D_k$ guarantees that $\{u_n\}$ is uniformly bounded on each compact subset $K\subset T$ since property (2) implies that $K\subset\pi(D_k)$ for some $k$.

The sequence $D_k$ will be obtained inductively, but we need to set some notation first. For each $\varepsilon>0$, let $\widetilde T$ be the $\mu$-geodesic triangle with vertices $p_1$, $\tilde{p}_2$ and $\tilde{p}_3$, such that $\tilde{p}_2$ and $\tilde{p}_3$ belong to the $\mu$-geodesic containing $\ell_1$ at distance $\varepsilon$ from $p_2$ and $p_3$, respectively. We will denote by $\tilde{\ell}_1$, $\tilde{\ell}_2$ and $\tilde{\ell}_3$ the sides of $\widetilde{T}$ opposite to $p_1$, $\tilde{p}_2$ and $\tilde{p}_3$, respectively (see Figure~\ref{fig:scherk}, left). We will assume that $\varepsilon$ is small enough such that $\widetilde T\subset U$. To avoid a cumbersome notation, and only throughout this proof, we will consider the usual trivialization $F:U\times\R\to\pi^{-1}(U)$ given by $F(p,t)=\phi_t(F_0(p))$, where $\{\phi_t\}_{t\in\R}$ is the $1$-parameter group of isometries associated to $\xi$. We will work on $U\times\R$ with the pullback metric by $F$ in the sequel, so the minimality of $F_0|_U$ means that $U\times\{t_0\}$ is minimal for all $t_0\in\R$.

\begin{figure}
	\centering\includegraphics[width=\textwidth]{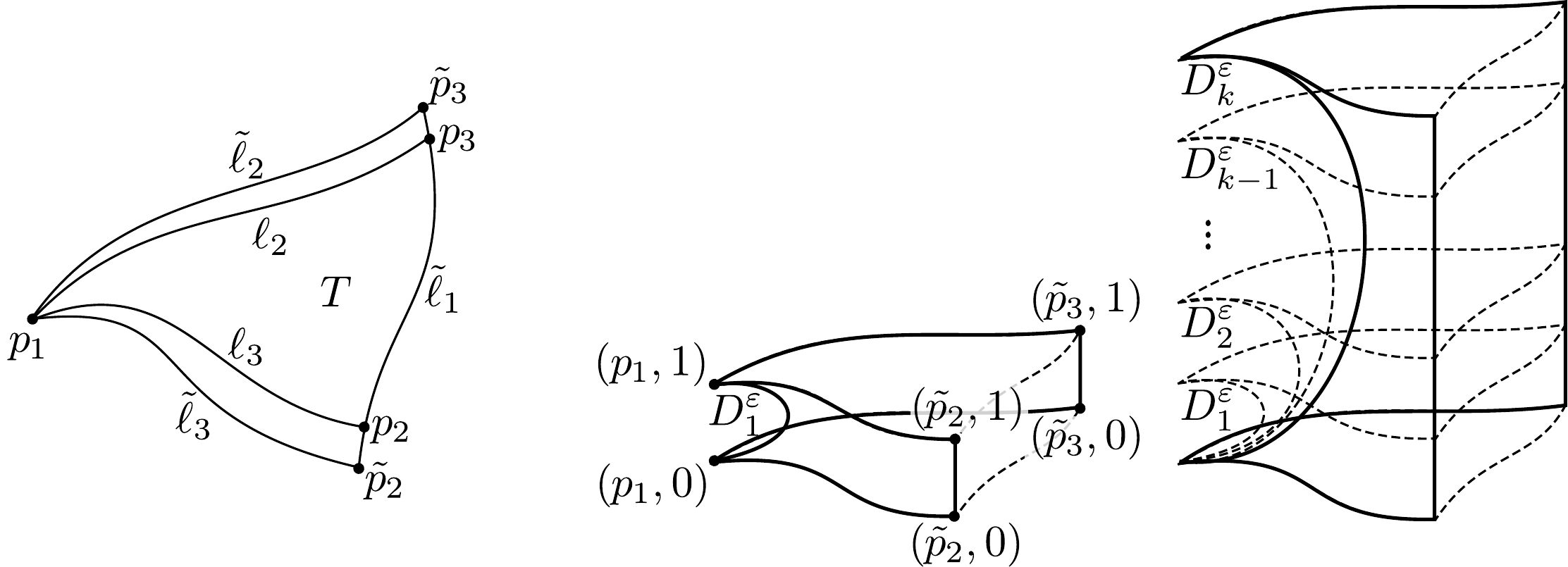}
	\caption{The $\mu$-geodesic triangles $T$ and $\widetilde T$ in the proof of Proposition~\ref{prop:scherk}, the initial minimal disk $D_1^\varepsilon\subset\widetilde T\times\R$ and the sequence of disjoint minimal disks $D_k^\varepsilon$ constructed by recurrence.}
	\label{fig:scherk}
\end{figure}

Let $M_1=\widetilde{T}\times[0,1]$ be the smooth compact $3$-manifold with boundary $\widetilde{T}\times\{0,1\}\cup (\ell_1\cup\tilde\ell_2\cup\tilde\ell_3)\times [0,1]$. Since $\partial M_1$ consists of five minimal smooth pieces meeting at angles less than $\pi$,~\cite[Thm.~1]{MY1} gives an (area-minimizing) smooth minimal disk $D^\varepsilon_1$ with boundary $(\tilde{\ell}_2\cup \tilde{\ell}_3)\times\{0,1\}\cup(\{\tilde{p}_2,\tilde{p}_3\}\times[0,1])$ that divides $T\times\R$ in two simply connected components (see Figure~\ref{fig:scherk}, center). The closure of the component whose boundary does not contain $\{p_1\}\times[0,1]$ will be denoted by $M_1^+$.

Given $k\geq 2$, we define by recurrence $M_k=M_{k-1}^+\cap(\widetilde{T}\times[0,k])$ so that $\partial M_k=(\widetilde{T}\times\{0,k\})\cup D_{k-1}^\varepsilon\cup(\tilde\ell_1\times[0,k])\cup((\tilde\ell_2\cup\tilde\ell_3)\times[k-1,k])$. Again, by~\cite[Thm.~1]{MY1}, we find a minimal surface $D_k^\varepsilon\subset M_k$ with boundary $(\tilde{\ell}_2\cup \tilde{\ell}_3)\times\{0,k\}\cup(\{\tilde{p}_2,\tilde{p}_3\}\times[0,k])$. We also define $M_k^+$ as the closure of the component of $(\widetilde T\times\R)\sm D_k^\varepsilon$ whose boundary does not contain $\{p_1\}\times[0,k]$. Notice also that $D^\varepsilon_k$ and $D^\varepsilon_{k-1}$ do not have interior contact points (and they are not tangent at any point of their common boundary) by the maximum principle since $D^\varepsilon_{k-1}$ acts as a barrier in the construction of $D^\varepsilon_{k}$ (see Figure~\ref{fig:scherk}, right). Since $D^{\varepsilon}_k$ is above $\Sigma_n$ for all $n$ and $k$ by the maximum principle and $\partial D^{\varepsilon}_k\cap\partial\Sigma_n=\emptyset$, we define $D_k$ as $D_{k}^0$ and conclude that it lies above $\Sigma_n$. In particular, property (1) holds true and $\Sigma_n$ is contained in $\cap_{k\in\N}M_k^+$ for all $n$. 

As for property (2), observe that $\pi(D_k)\subset\pi(D_{k+1})$ for all $k$, so we we will reason by contradiction assuming that $\cup_{k\in\N}\pi(D_k)$ is not all $T$. Translate vertically each $D_k$ so that it now lies in $T\times[\frac{-k}{2},\frac{k}{2}]$. These translated disks are area-minimizing (in particular, stable) in $U\times\R$, which has bounded geometry. Since the sequence $D_k$ is ordered (in the sense described in the previous paragraph), and $\cup_{k\in\N}\pi(D_k)$ is not all $T$, we can find an accumulation point $q_0\in T\times\R$ projecting outside $\cup_{k\in\N}\pi(D_k)$. All in all, standard convergence arguments yield the existence of a stable minimal surface $D_\infty\subset T\times\R$ with boundary $\{p_1,p_2\}\times\R$. Since $q_0\in D_{\infty}$, we conclude that $D_\infty$ cannot be the vertical cylinder $\ell_1\times\R$.

Consider the open-book decomposition of $T$ with binding $\pi^{-1}(\tilde{p}_3)$ given by Corollary~\ref{coro:open-book-decomposition}. Since $D_\infty$ lies in $T\times\R$ and $\tilde{p}_3$ is outside $T$, we can find a leaf $P$ of this open-book decomposition such that $P$ and $D_\infty$ are asymptotic and $D_\infty$ lies in one of the components of $(T\times\R)\sm P$ (note that there cannot be interior tangency points of $D_\infty$ and $P$ by maximum principle). Let $\{q_n\}$ be a sequence of points in $D_\infty$ whose distance to $P$ converges to zero, and let $D_\infty^n$ be the vertical translation of $D_\infty$ that sends $q_n$ to a point at height zero. Again, we can take the limit of $D_\infty^n$ as $n\to\infty$ and produce a minimal surface $D_\infty^\infty$ containing $(\pi(q_0),0)\in P$ and lying in the closure of the same component of $(T\times\R)\sm P$, so in this case $D_\infty^\infty$ does coincide with $P$ by the maximum principle. Let $K\subset P$ be a compact domain such that $\pi(K)\sm T\neq\emptyset$, so the convergence ensures that there exist domains $K_n$ in $\Sigma_n$ such that $K_n$ converges uniformly to $K$. This says that there are points in $\Sigma_n$ that project outside $T$, which is a contradiction. 
\end{proof}
	
\begin{remark}
Under the same assumptions, if $p$ and $p'$ are two points in $\ell_1$ and $\gamma$ is a $\mu$-convex curve in $T$ joining $p$ and $p'$, then the same argument in the proof of Proposition~\ref{prop:scherk} yields the existence of a minimal graph $u$ over $\Omega'$ such that $u|_\gamma=g$ and $u|_{\ell_1}=\pm\infty$ for any bounded function $g\in\mathcal{C}^0(\gamma)$, where $\Omega'$ is the relatively compact subdomain demarcated by $\ell_1$ and $\gamma$.
\end{remark}

We can use these Scherk type surfaces to analyze the boundary behavior of a sequence of minimal graphs that converges in the interior of the domain. This is a key step in the proof of Theorem~\ref{thm:JS} (see also~\cite[Lemma~7]{JS} and the Boundary Value Lemma in~\cite{CR}, whose proofs use different barriers). 
	
\begin{proposition}\label{prop:boundary-value}
Let $\{u_n\}$ be a sequence of minimal graphs in a domain $\Omega\subset M$. Assume that there is a $\mu$-convex arc $C\subset\partial\Omega$ such that each $u_n$ can be extended continuously to $\Omega\cup C$. If $u_n$ converges uniformly on compact subsets of $\Omega$ to a minimal graph $u$ and $\{u_n|_C\}$ converges uniformly to a function $f$ on $C$, then 
\begin{enumerate}[label=\emph{(\alph*)}]
	\item $\{u_n\}$ is uniformly bounded on a neighborhood of each $p_0\in C$,
	\item $u$ extends continuously to $\Omega\cup C$ by setting $u|_{C}=f$.
\end{enumerate}

\end{proposition}

\begin{proof}
In order to prove item (a), let us distinguish two cases:
\begin{enumerate}[label=(\arabic*)]
	\item If $C$ is strictly $\mu$-convex at $p_0$, then take two Scherk graphs over a small triangle $T$ with values $\pm\infty$ along a side tangent to $\partial\Omega$ at $p_0$ and values $f(p_0)\pm1$ on the other two sides (as in~\cite[Fig.~2]{CR}, see also~\cite[Lemma~7]{JS}). It is clear that if $T$ is small enough, all the $u_n$ lie in between these two Scherk barriers.
	\item Assume that $C$ is a $\mu$-geodesic arc with $p_0$ in the interior (we restrict $C$ if necessary). Take $p_1\in C$ close enough to $p_0$ such that $B_M(p_1,r)$ lies in a totally $\mu$-convex neighborhood of $p_1$ that contains $p_0$. Let $C_\theta$ be the radial geodesics through $p_1$ parametrized by the angle they make with $C_0=C$. Let $d=d_M(p_0,p_1)$ and, for a small $0<\rho<d$, consider the $\mu$-geodesic segment $C_{\theta,\rho}=\{p\in C_\theta:|d_M(p_1,p)-d|<\rho\}$ and the vertical region $Q_{\theta,\rho}=\cup_{t\in(0,2\rho)}\phi_t(C_{\theta,\rho})\subset\pi^{-1}(C_\theta)$, see Figure~\ref{fig:douglas}. For a fixed $\rho$, $\partial Q_{\theta,\rho}$ converges to $\partial Q_{0,\rho}$ as $\theta\to 0$ so we can find (not necessarily minimal) annuli $\Sigma_{\theta,\rho}$ of arbitrarily small area with boundary $\partial Q_{\theta,\rho}\cup\partial Q_{0,\rho}$ by making $\theta$ small enough. Since $\partial Q_{0,\rho}$ remains fixed, Douglas criterion ensures the existence of a minimal annulus $S$ with boundary the two quadrilaterals $\partial Q_{0,\rho}\cup \partial Q_{\theta,\rho}$ for a small enough $\theta>0$. Since $u_n$ converges uniformly to $f$ along $C_{0,\rho}\subset C$ and converges uniformly to $u$ on the compact subset $C_{\rho,\theta}\subset\Omega$, a vertical translation of $S$ provides the desired uniform estimate for $u_n$ (from above and from below) on $\pi(S)\cup C_{0,\rho}$, which is a neighborhood of $p_0$ in $\Omega\cup C$.
\end{enumerate}

\begin{figure}
\centering\includegraphics[width=0.45\textwidth]{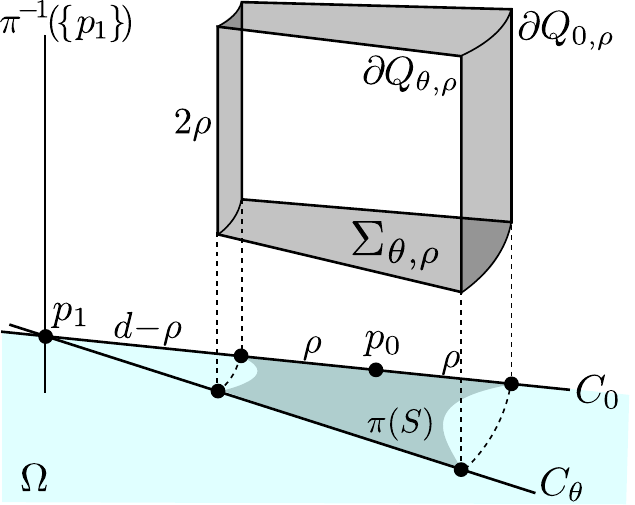}
\caption{The minimal annulus $\Sigma_{\theta,\rho}$ has less area than any two disks with boundary $\partial Q_{\theta,\rho}\cup\partial Q_{0,\rho}$ for small enough $\theta$. The shaded region in $\pi(S)\subset\Omega$ is the domain where the local barriers apply.}\label{fig:douglas}
\end{figure}

As for item (b), consider the barriers $\omega_k^{\pm}$ at $p_0$ given in the proof of Theorem~\ref{thm:generalized-existence}. Since $f$ is continuous and $\{u_n\}$ converges uniformly to $f$, there exists $r>0$ such that $|f(p)-f(p_0)|\leq\frac{1}{k}$ and $|u_n(p)-f(p_0)|\leq\frac{1}{k}$ for all $n\in\N$ and $p\in\partial\Omega$ with $d_M(p,p_0)<r$. We can choose the triangle $T$ that defines $\omega_k^{\pm}$ sufficiently small such that $d_M(p,p_0)<r$ for all $p\in T\cap\Omega$. Item (a) allows us to assume that $T$ is again small enough so $\{u_n\}$ is uniformly bounded on $\overline{T\cap\Omega}$. This means that we can choose the constant $M_0$ (see the definition of barrier in Section~\ref{subsec:perron}) large enough such that for all $n,k\in\N$ and $p\in T\cap\Omega$ we have
\begin{equation}\label{prop:boundary-value:eqn1}
\omega_k^-(p)\leq u_n(p)\leq\omega_k^+(p).
\end{equation}
This inequality holds in the boundary of $T\cap\Omega$ and extends to the interior by the maximum principle. Letting $n\to\infty$, the same inequality~\eqref{prop:boundary-value:eqn1} holds for $u$ at any interior point $p\in T\cap\Omega$. Finally, we get that $\lim_{p\to p_0}u(p)=f(p_0)$ by letting $k\to\infty$ so we get item (b). Here, we are using that $M_0$ is fixed in the process, whence $\omega_k^+$ (resp.\ $\omega_k^-$) is a decreasing (resp.\ increasing) sequence of functions.
\end{proof}

\begin{proposition}\label{prop:boundary-infinity}
Let $\{u_n\}$ be a sequence of minimal graphs in a domain $\Omega\subset M$. Assume that there is a $\mu$-geodesic arc $A\subset\partial\Omega$ such that each $u_n$ can be extended continuously to $\Omega\cup A$. If $u_n$ converges uniformly on compact subsets of $\Omega$ to a minimal graph $u$ and $\{u_n|_A\}$ diverges uniformly to $\pm\infty$, then $u$ also diverges to $\pm\infty$ as we approach $A$.
\end{proposition}

\begin{proof}
Assume that $\{u_n|_A\}$ diverges to $+\infty$ (the case of $-\infty$ is similar) and let $p_0\in A$. The same argument as in case (2) of the proof of Proposition~\ref{prop:boundary-value} implies that there is a neighborhood $V$ of $p_0$ in $\Omega\cup A$ where $\{u_n\}$ is uniformly bounded from below, say there is some $a\in\R$ such that $u_n(p)\geq a$ for all $p\in V$ and $n\in\N$. Let $A'\subset A\cap V$ be a subarc centered at $p_0$ small enough so that there is a $\mu$-convex curve $\Gamma\subset\Omega\cap V$ joining the endpoints of $A'$ and $A'\cup\Gamma$ is the boundary of a topological disk $D$. Let $v_m$ be the minimal graph over $D$ with boundary values $m$ on $A'$ and $a$ on $\Gamma$, which exists by Theorem~\ref{thm:existence}. By the maximum principle, it follows that $u_n\geq v_m$ on $D$ for $n$ large enough, whence $u\geq v_m$ on $D$ for all $m$. Since $v_m$ is an increasing sequence of functions that take arbitrarily large values, we deduce that $\lim_{p\to p_0}u(p)=+\infty$.
\end{proof}

\section{The Jenkins--Serrin Problem}\label{sec:JS}
	
We give necessary and sufficient conditions to solve the Dirichlet problem for the minimal surface equation in $\E$ over a relatively compact domain $\Omega\subset M$, with possible infinite boundary values on some arcs of $\partial\Omega$, which must be $\mu$-geodesic by Lemma~\ref{lemma:boundary-structure}. The domain $\Omega$ is allowed to have simple closed $\mu$-geodesics as boundary components with no vertices. This makes us consider the following problem:

\begin{definition}\label{def:JS-problem}
A relatively compact open connected domain $\Omega\subset M$ will be called a \emph{Jenkins--Serrin domain} if $\partial\Omega$ is piecewise regular and consists of $\mu$-geodesic open arcs or simple closed $\mu$-geodesics $A_1,\ldots,A_r,B_1,\ldots,B_s$ and $\mu$-convex curves $C_1,\ldots,C_m$ with respect to the inner conormal to $\Omega$. The finite set $V\subset\partial\Omega$ of intersections of all these curves will be called \emph{vertex set} of $\Omega$. 

The \emph{Jenkins--Serrin problem} consists in finding a minimal graph over $\Omega$, with limit values $+\infty$ on each $A_i$ and $-\infty$ on each $B_i$, and such that it extends continuously to $\Omega\cup(\cup_{i=1}^mC_i)$ with prescribed continuous values on each $C_i$ with respect to a prescribed initial section $F_0$ defined on a neighborhood of $\Omega$.
\end{definition}

Note that all arcs are assumed to not contain their endpoints because a possible solution to the Jenkins--Serrin problem is not actually defined (as a function) at the vertices of the domain $\Omega$ in general, where discontinuities may occur. 

An extra admissibility condition for Jenkins--Serrin domains is needed.  

\begin{definition}
A Jenkins--Serrin domain $\Omega\subset M$ is said \emph{admissible} if neither two of the $A_i$'s nor two of the $B_i$'s meet at a convex corner.
\end{definition}

Proposition~\ref{prop:admissibility} will show that the admissibility is necessary and Theorem~\ref{thm:JS} will show that it is sufficient. If there are no $C_i$ components, Jenkins and Serrin~\cite{JS} also assumed that neither $\cup A_i$ nor $\cup B_i$ is connected, a condition that has been also required in the case of $M\times\R$, see~\cite{Pin} or~\cite{MRR}. Our approach allows us to drop this extra hypothesis for the admissibility of the domain, see also Section~\ref{sec:examples}.

\begin{definition}\label{def:mu-polygon}
Let $\Omega$ be a Jenkins--Serrin domain.  We will say that $\mathcal{P}$ is a $\mu$-polygon inscribed in $\Omega$ if $\mathcal{P}$ is the union of disjoint curves $\Gamma_1\cup\dots\cup\Gamma_k$ satisfying the following conditions (see Figure~\ref{fig:JS}):
	\begin{itemize}
		\item $\mathcal P$ is the boundary of an open and connected subset of $\Omega$;
		\item each $\Gamma_j$ is either a closed $\mu$-geodesic or a closed piecewise-regular curve with $\mu$-geodesic components whose vertices are among the vertices of $\Omega$.
	\end{itemize}
For such an inscribed $\mu$-polygon $\mathcal P$, define
\begin{align*}
\alpha(\mathcal P)&=\Length_\mu((\cup A_i)\cap\mathcal{P}),&\gamma(\mathcal P)&=\Length_\mu(\mathcal{P})
,\\
\beta(\mathcal P)&=\Length_\mu((\cup B_i)\cap\mathcal{P}).&&
\end{align*}
\end{definition}

\begin{figure}
	\centering\includegraphics[width=0.5\textwidth]{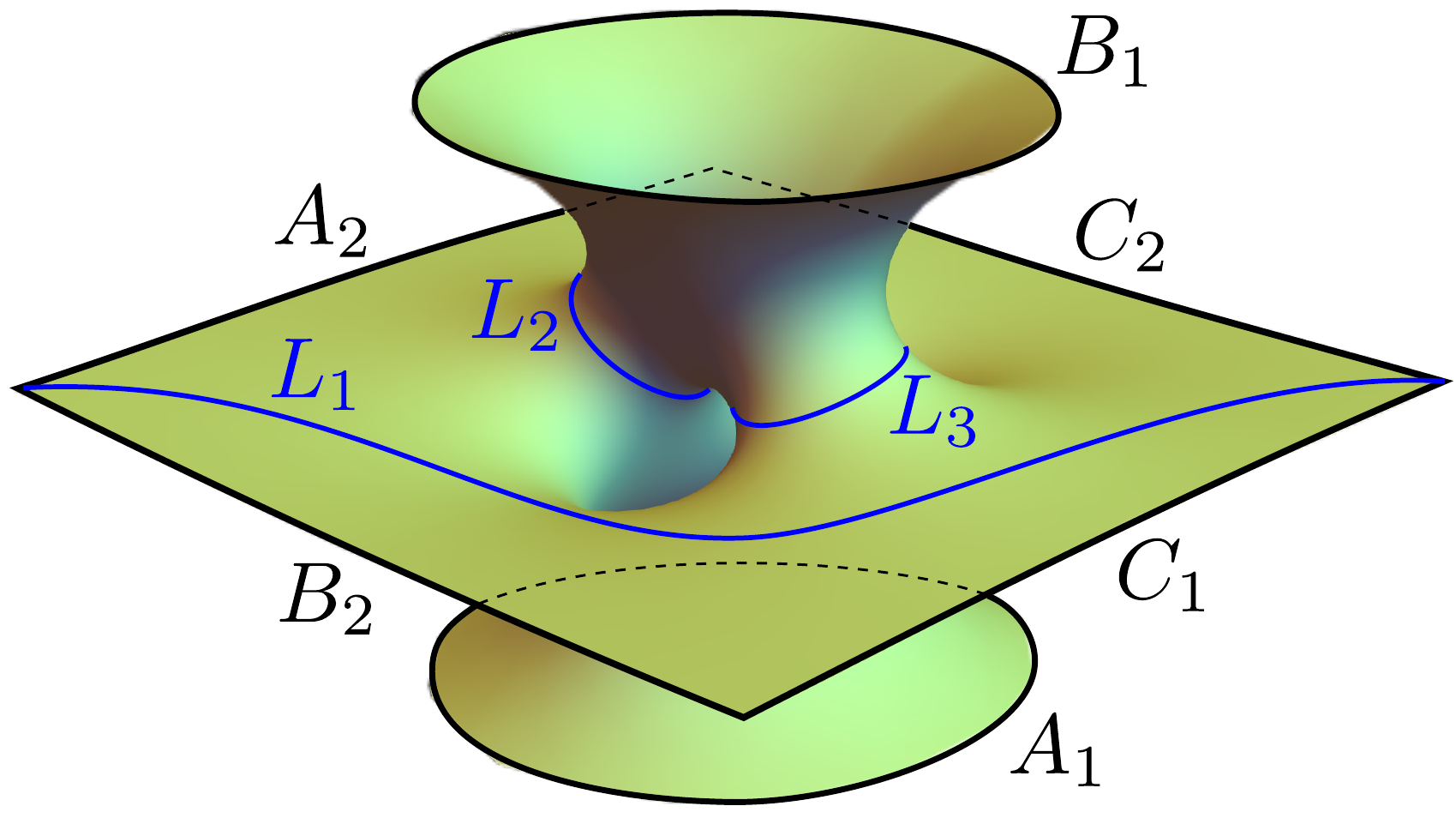}
	\caption{A Jenkins--Serrin problem with six $\mu$-geodesic boundary components over a domain with the topology of a Costa surface. A possible inscribed $\mu$-polygon is $\mathcal P=\cup_{i=1}^4\Gamma_i$ with $\Gamma_1=L_1\cup A_2\cup C_2$, $\Gamma_2=L_2$, $\Gamma_3=L_3$ and $\Gamma_4=A_1$.}\label{fig:JS}
\end{figure}

Now we have all the ingredients to state our main theorem.

\begin{theorem}\label{thm:JS}
Let $\Omega$ be an admissible Jenkins--Serrin domain.
\begin{enumerate}[label=\emph{(\alph*)}]				
\item If the family $\{C_i\}$ is non-empty, then the Jenkins--Serrin problem in $\Omega$ has a solution if and only if 
\begin{equation}\label{JS-condition}
	2\alpha(\mathcal P)<\gamma(\mathcal P)\quad\text{and}\quad 2\beta(\mathcal P)<\gamma(\mathcal P)
\end{equation} 
for all inscribed $\mu$-polygons $\mathcal P\subset\Omega$, in which case the solution is unique. 

\item If the family $\{C_i\}$ is empty, then the Jenkins--Serrin problem in $\Omega$ has a solution if and only if~\eqref{JS-condition} holds true for all inscribed $\mu$-polygons $\mathcal{P}\neq\partial\Omega$ and $\alpha(\partial\Omega)=\beta(\partial\Omega)$. The solution is unique up to vertical translations.
\end{enumerate}
\end{theorem}

The conditions in the statement about inscribed polygons will be called the \emph{JS-conditions} for short. In the rest of this section, we will introduce the flux to prove that these JS-conditions are necessary (Proposition~\ref{prop:admissibility}) as well as the uniqueness (Theorem~\ref{thm:infinity}). Finally, the existence of solutions will be proved in Section~\ref{sec:divergence-lines} by the method of divergence lines.

\subsection{The flux across a curve}\label{subsec:flux}
Let $\Omega\subset M$ be any domain. As shown by Lemma~\ref{lemma:H}, the fact that $u\in\mathcal{C}^\infty(\Omega)$ satisfies the minimal surface equation can be written as $\mathrm{div}(X_u)=0$, where $X_u=\mu^2Gu/W_u$. This zero-divergence equation leads naturally to the definition of a flux for minimal graphs across curves of $\Omega$.

\begin{definition}
Let $\Gamma\subset\Omega$ be a piecewise regular curve. The \emph{flux} of $u\in\mathcal{C}^\infty(\Omega)$ across $\Gamma$ with respect to a unit normal vector field $\eta$ to $\Gamma$ in $M$ is defined as
	\[\Flux(u,\Gamma)=\int_\Gamma\langle X_u,\eta\rangle.\]
\end{definition}

Since $\|X_u\|\leq\mu$ is bounded in $\Omega$, the flux of $u$ is well defined. This definition depends on the choice of the unit normal vector field, but the absolute value $|\Flux(u,\Gamma)|$ does not. The divergence theorem ensures that the flux across a curve enclosing a certain domain vanishes, so $|\Flux(u,\Gamma)|=|\Flux(u,\Gamma')|$ for two piecewise regular curves $\Gamma$ and $\Gamma'$ which are homotopic with respect to their common endpoints. Note also that Cauchy--Schwarz inequality yields the upper bound
\[|\Flux(u,\Gamma)|\leq \int_\Gamma\mu=\Length_\mu(\Gamma).\]
This last term denotes the $\mu$-length of $\Gamma$, i.e., the length of $\Gamma$ with respect to the conformal metric $\mu^2\df s_M^2$. 

If $X_u$ extends continuously to a regular curve $\Gamma\subset\partial\Omega$, then the flux across $\Gamma$ can be defined similarly. Next lemma discusses the two different scenarios in which this idea has been typically applied. 

\begin{lemma}\label{lemma:flux-boundary} 
Let $u$ be a solution to the minimal surface equation over $\Omega$.
\begin{enumerate}[label=\emph{(\arabic*)}]
	\item If $u$ has limit value $\pm\infty$ along a $\mu$-geodesic arc $A\subset\partial\Omega$, then $\Flux(u,A)=\pm\Length_\mu(A)$ with respect to the outer conormal to $\Omega$ along $A$.
	\item If $u$ extends continuously to $\Omega\cup C$, where $C\subset\partial\Omega$ is a $\mu$-convex curve (with respect to the inner conormal), then $|\Flux(u,C)|<\Length_\mu(C)$.
\end{enumerate}
\end{lemma}

\begin{proof}
The equality $|\Flux(u,A)|=\Length_\mu(A)$ easily follows from the fact that if $u\to\pm\infty$ along $A$, then the tangent planes converge uniformly to vertical planes by Lemma~\ref{lemma:boundary-structure}. This means that $\nabla u$ is not bounded when approaching $A$, whence $X_u$ is asymptotically equivalent to $\mu\, Gu/\|Gu\|$ or $\mu\nabla u/\|\nabla u\|$, where the norm is computed with respect to $\df s_M^2$. Consequently, $X_u$ can be extended continuously to $\Omega\cup A$ as $X_u=\pm\mu\cdot\eta$ on $A$, where the sign is positive if $u\to+\infty$ or negative if $u\to-\infty$ and $\eta$ is the outer conormal to $\Omega$ along $A$.
	
In order to prove item (2), we will suppose without loss of generality that $\Omega$ is itself $\mu$-convex and $u$ has continuous values in $\partial\Omega$, because the argument is local. Let $C'$ be a proper open subset of $\partial\Omega$. Theorem~\ref{thm:generalized-existence} guarantees the existence of $v\in\mathcal{C}^\infty(\Omega)$ satisfying the minimal graph equation such that $v=u-a$ in $C'$ and $v=u$ in $\partial\Omega\smallsetminus C'$. Since $u-v$ is not constant, Lemma~\ref{lemma:factorization} gives
	\[\int_\Omega\langle\nabla u-\nabla v,X_u-X_v\rangle>0.\]
	Since $\mathrm{div}\left((u-v)(X_u-X_v)\right)=\langle\nabla u-\nabla v,X_u-X_v\rangle$, divergence theorem yields
	\[0<\int_{\partial\Omega}(u-v)\langle X_u-X_v,\eta\rangle=a\int_{C'}\langle X_u-X_v,\eta\rangle.\]
	Letting $a=\pm 1$ and using the inequality $|\Flux(v,C')|\leq\Length_\mu(C')$, we obtain that $|\Flux(u,C')|<\Length_\mu(C')$, whence $|\Flux(u,C)|<\Length_\mu(C)$.
\end{proof}

Next proposition shows that the admissibility of the domain and the JS-conditions given by Theorem~\ref{thm:JS} are necessary.

\begin{proposition}\label{prop:admissibility}
Consider a Jenkins--Serrin problem over some domain $\Omega\subset M$.
\begin{enumerate}[label=\emph{(\arabic*)}]
 	\item If two $\mu$-geodesic components of $\partial\Omega$ meet at a convex corner and are both assigned the same value $+\infty$ or $-\infty$, then the problem has no solutions.
 	\item If the problem has a solution, then the JS-conditions are satisfied.
 \end{enumerate}
\end{proposition}

\begin{proof}
Assume by contradiction that the problem with two adjacent sides $A_1$ and $A_2$ meeting at a convex corner $p$ has a solution $u$. Let $p_1\in A_1$ and $p_2\in A_2$ be sufficiently close to $p$ so that the $\mu$-geodesic arcs joining $p,p_1,p_2$ are contained in $\Omega$ and realize the $\mu$-distance between these three points. The flux of $u$ across the boundary of the triangle of vertices $p,p_1,p_2$ is zero, which implies that $\Length_\mu(p_1p_2)>\Length_\mu(pp_1)+\Length_\mu(pp_2)$ by Lemma~\ref{lemma:flux-boundary}. This is in contradiction with the triangle inequality for the $\mu$-metric.

As for item (2), let $\mathcal P$ be an inscribed $\mu$-polygon which is the boundary of an open connected subset $\Omega_0\subset\Omega$. The flux of a solution $u$ across $\mathcal P$ gives
\begin{equation}\label{prop:admissibility:eqn1}
\Flux(u,(\cup A_i)\cap \mathcal P)+\Flux(u,(\cup B_i)\cap \mathcal P)+\Flux(u,\mathcal P\sm[(\cup A_i)\cup(\cup B_i)])=0,
\end{equation}
with respect to the outer conormal to $\Omega_0$ along $\mathcal P$. The first two summands in~\eqref{prop:admissibility:eqn1} add up to $\alpha(\mathcal P)-\beta(\mathcal P)$, whereas the third one is, in absolute value, less than $\gamma(\mathcal P)-\alpha(P)-\beta(\mathcal P)$ by Lemma~\ref{lemma:flux-boundary}. This gives the inequality
\[\gamma(\mathcal P)-\alpha(P)-\beta(\mathcal P)>|\Flux(u,\mathcal P\sm[(\cup A_i)\cup(\cup B_i)])|=|\alpha(\mathcal P)-\beta(\mathcal P)|,\]
and it easily follows that $2\alpha(\mathcal P)<\gamma(\mathcal P)$ and $2\beta(\mathcal P)<\gamma(\mathcal P)$. However, this is true unless $\mathcal P=\partial\Omega$ and there are no $C_i$ components, in which case the third summand in~\eqref{prop:admissibility:eqn1} is identically zero, whence $\alpha(\mathcal P)=\beta(\mathcal P)$.
\end{proof}

\subsection{A maximum principle for Jenkins--Serrin graphs}\label{subsec:maximum-principle}
Next result extends Proposition~\ref{prop:uniqueness} to allow infinite values. It is stated as needed later, but it is worth observing that item (a) still holds true under milder assumptions with the same proof (e.g., if $u$ extends continuously or has asymptotic value $-\infty$ on a $\mu$-geodesic arc on which $v$ has asymptotic value $+\infty$).

\begin{theorem}[Uniqueness]\label{thm:infinity}
Let $\Omega\subset M$ be a Jenkins--Serrin domain. Suppose that $u,v\in\mathcal{C}^\infty(\Omega)$ define minimal graphs (with respect to a given initial section $F_0$) that extend continuously to $\Omega\cup(\cup C_i)$ and they both tend to $+\infty$ on each $A_i$ and to $-\infty$ on each $B_j$. 
\begin{enumerate}[label=(\alph*)]
	\item If $\cup C_i\neq\emptyset$ and $u\leq v$ in $\cup C_i$, then $u\leq v$ in $\Omega$. 
	\item If $\cup C_i=\emptyset$, then $u=v+c$ for some $c\in\R$.
\end{enumerate}
\end{theorem}
	
\begin{proof}
Let $w=u-v$ and assume that $U=\{p\in\Omega:w(p)>0\}\neq\emptyset$ to reach a contradiction, which proves item (a), but also item (b) if we previously add a negative constant to $v$ so that $\{p\in\Omega:w(p)\leq 0\}$ is not empty either. We will use and extend the notation and arguments of Proposition~\ref{prop:uniqueness}. By adding a small positive constant to $v$ to assume that $\nabla w\neq 0$ along $\partial U$ and $w<0$ on $\cup C_i$ and restrict to a connected component $\widetilde U_\varepsilon$. Since $u$ and $v$ define minimal graphs, the divergence theorem guarantees that
\begin{equation}\label{eq:thm-cero}
	\Flux(u,\partial \widetilde U_\varepsilon)-\Flux(v,\partial \widetilde U_\varepsilon)=\int_{\partial \widetilde U_\varepsilon}\langle X_u-X_v,\eta\rangle=0,
\end{equation}
where $\eta$ denotes the outer conormal to $\widetilde U_\epsilon$ along its boundary. We can decompose $\partial\widetilde U_\varepsilon=\Gamma^1_\varepsilon\cup \Gamma^2_\varepsilon\cup\Gamma^3_\varepsilon$, where $\Gamma^1_\varepsilon\subset\widetilde U\subset\mathrm{int}(\Omega)$, $\Gamma^2_\varepsilon$ lies in the boundary of the geodesic balls of radius $\varepsilon$, and $\Gamma^3_\varepsilon\subset(\cup A_i)\cup(\cup B_i)$. Note that the component $\Gamma^3_\varepsilon$ did not appear in Proposition~\ref{prop:uniqueness} because there were no infinite values in there. The assumption that $\nabla w\neq 0$ along $\partial U$ implies that $\Gamma^1_\varepsilon$, $\Gamma^2_\varepsilon$ and $\Gamma^3_\varepsilon$ are away from $(\cup C_i)\cup V$ and consist of finitely many regular curves.

Consider the functions $\alpha_i(\varepsilon)$ defined as in~\eqref{thm:eq3} now for $i\in\{1,2,3\}$. By the same reasons as in Proposition~\ref{prop:uniqueness}, we have that $\lim_{\varepsilon\to 0}\alpha_1(\varepsilon)<0$ and $\lim_{\varepsilon\to 0}\alpha_2(\varepsilon)=0$. However, since $\Flux(u,\Gamma_\varepsilon^3)=\Flux(v,\Gamma_\varepsilon^3)$ by Lemma~\ref{lemma:flux-boundary}, we find that $\alpha_3(\varepsilon)=0$ for all $\varepsilon>0$, i.e., $\Gamma_\varepsilon^3$ does not contribute to the flux. This contradicts the fact that $\alpha_1(\varepsilon)+\alpha_2(\varepsilon)+\alpha_3(\varepsilon)=0$ that follows from Equation~\eqref{eq:thm-cero}.
\end{proof}

\section{Divergence lines}\label{sec:divergence-lines}

In order to prove the existence of solution to the Jenkins--Serrin problem, we will consider the possible limits of a sequence of graphs (not necessarily monotone), a context in which the theory of divergence lines plays an important role~\cite{Mazet,MRR}. Recall that $\pi:\E\to M$ is a Killing submersion whose fibers have infinite length, $\Omega\subset M$ is a relatively compact domain, and we are considering Killing graphs with respect to a fixed zero minimal section $F_0$ defined on a neighborhood of $\overline\Omega$.

Let $\{u_n\}$ be a sequence of minimal graphs in $\Omega$. For each $p\in\Omega$, define the translated minimal graph $\Sigma_n(p)\subset\E$ as the graph of $u_n-u_n(p)$. Observe that $\Sigma_n(p)$ contains the point $q=F_0(p)$ for all $n\in\N$ and has uniformly bounded curvature in a solid vertical cylinder of axis $\pi^{-1}(p)$ whose radius does not depend on $n$ by Lemma~\ref{lem:curvature-estimate}. Since $\pi^{-1}(\Omega(\delta))$ has bounded geometry, standard convergence arguments show that a subsequence of $\Sigma_n(p)$ converges (locally nearby $q$) in the $\mathcal C^k$-topology on compact subsets for all $k\geq 0$ to a minimal surface $\Sigma_\infty$ that contains $q$. In particular, the angle functions $\nu_n$ of $\Sigma_n(p)$ converge to the angle function $\nu_\infty$ of $\Sigma_\infty$, whence $\nu_\infty\geq 0$. Since $\nu_\infty$ lies in the kernel of the stability operator of $\Sigma_\infty$, it satisfies a maximum principle (see~\cite[Ass.~2.2]{MeeksPerezRos}) so that either $\nu_\infty$ is identically zero or $\nu_\infty$ never vanishes.
\begin{itemize}
	\item If the generalized gradients $Gu_n$ are bounded at $p$, then any convergent subsequence of $\Sigma_n(p)$ actually converges to a minimal graph over a metric ball $D_M(p,R)$. By Proposition~\ref{prop:harnack} and Theorem~\ref{thm:compactness}, the radius $R$ can be chosen depending only on $d(p,\partial\Omega)$ and on the value of $\|Gu_n\|$ at $p$.

	\item If the generalized gradients $Gu_n$ (and hence the usual gradients $\nabla u_n$) are not bounded at $p$, up to a subsequence, we can assume that $\nu_n(p)=(\mu^{-2}+\|Gu_n(p)\|^2)^{-1/2}\to 0=\nu_\infty(p)$. This yields $\nu_\infty\equiv 0$ so we can produce a limit surface $\Sigma_\infty$ which is part of a vertical cylinder over a $\mu$-geodesic arc through $p$. Let $L$ be maximal extension of this $\mu$-geodesic arc inside $\Omega$. A standard diagonal argument says that there is a further subsequence $\Sigma_{\sigma(n)}(p)$ which converges uniformly to $\pi^{-1}(L)$ in the $\mathcal C^k$-topology on compact subsets for all $k\geq 0$ (see~\cite[Lemma~4.3]{MRR}) and the unit normals of the sequence become horizontal along $L$. In the Killing-submersion setting, this means that
	\begin{equation}\label{eqn:divergence-line-normal}
	\nu_{\sigma(n)}\longrightarrow 0,\qquad\text{and}\qquad \eta_{\sigma(n)}\longrightarrow\pm\eta_L,
	\end{equation}
	where $\eta_n={\nabla u_n}/{\|\nabla u_n\|}$ and $\eta_L$ is a unit normal to $L$ in the metric $\df s^2_M$ (not in the $\mu$-metric). Actually, to this end and for the arguments hereafter, we could have defined $\eta_n=Gu_n/\|Gu_n\|$ equivalently. 
\end{itemize}

\begin{definition}\label{defi:divergence-line}
A $\mu$-geodesic $L\subset\Omega$ is called a \emph{divergence line} of a sequence of minimal graphs $u_n$ over $\Omega$ if $L$ is maximal (i.e., it is not a proper subset of another $\mu$-geodesic $L'\subset\Omega$) and the graphs of $u_n-u_n(p)$ converge uniformly to $\pi^{-1}(L)$ on compact subsets for some (and hence for all) $p\in L$.
\end{definition}

\subsection{Structure of the set of divergence lines}\label{subsec:divlines}
Throughout this section, we will assume that $\Omega$ is a Jenkins--Serrin domain as in Definition~\ref{def:JS-problem}, though most properties can be easily adapted to general bounded or unbounded domains. A divergence line can be a closed $\mu$-geodesic or an open $\mu$-geodesic arc of finite or infinite length. As a matter of fact, $\mu$-geodesics in an arbitrary surface have self-intersections or accumulation points but next result shows that this is not possible for divergence lines. It is worth mentioning that in other more specific cases in the literature (e.g., in $\mathbb{H}^2\times\R$~\cite{MRR}), this discussion is not pertinent because $\mu$-geodesics are properly embedded automatically.

\begin{lemma}\label{lem:divline-properly-embedded}
Each divergence line of a sequence of minimal graphs in $\Omega$ is properly embedded in $\overline\Omega$. In particular, such a line is either a closed $\mu$-geodesic or an open $\mu$-geodesic arc with finite $\mu$-length connecting two points of $\partial\Omega$.
\end{lemma}

\begin{proof}
First, if a divergence line $L$ has a self-intersection at $p\in\Omega$, we find a contradiction. Consider a compact subset $K\subset\pi^{-1}(L)$ that contains some $q\in\pi^{-1}(p)$ in the interior. Given a translated subsequence $\Sigma_{\sigma(n)}$ that uniformly converges to $K$, their unit normals also converge uniformly to the normal of $K$ at $q$. Since the self-intersection of $L$ is transverse (because $L$ is a $\mu$-geodesic), this contradicts the uniqueness of limit of $\nabla u_{\sigma(n)}/\|\nabla u_{\sigma(n)}\|$ as stated in~\eqref{eqn:divergence-line-normal}.

Now we shall assume that $L$ accumulates on some $p\in\Omega$ and find a contradiction again. Let $\Sigma_{\sigma(n)}$ be a translated subsequence that uniformly converges to $\pi^{-1}(L)$ on compact subsets and let $U_q$ be a neighborhood of $q=F_0(p)$ where harmonic coordinates exist (see Lemma~\ref{lem:uniform-graph-lemma}). Accumulation at $p$ gives a sequence $p_k\in L$ converging to $p$ and disjoint closed subarcs $L_k\subset L\cap \pi(U_q)$ of fixed length centered at $p_k\in L_k$ that converge to a limit $\mu$-geodesic arc $L_\infty$ through $p$. We can also assume that a $\mu$-geodesic arc $\Gamma\subset\Omega$ orthogonal to $L$ at $p$ intersects all the arcs $L_k$ transversely. For each $k\in\N$, consider $K_m=\cup_{k=1}^m\cup_{t\in[-1,1]}\phi_t(F_0(L_k))$, which is a compact subset of $\pi^{-1}(L)$. Since $\Sigma_{\sigma(n)}$ converges uniformly to $K_m$ for any fixed $m$, the curve $\Sigma_{\sigma(n)}\cap\pi^{-1}(\Gamma)$ must go up and down many times, so there must be local maxima  $q_k\in\pi^{-1}(\Gamma)\cap\Sigma_{\sigma(n)}$ of the height of this curve over $F_0$ (indeed as many as desired by making $m$ and $n$ large enough), see Figure~\ref{fig:divergence1}. Since $q$ is bounded away from the boundaries $\partial\Sigma_n$ (uniformly on $n$), we can translate vertically so that each $q_k$ lies in $U_q$ and the uniform graph lemma~\ref{lem:uniform-graph-lemma} implies that an intrinsic ball of $\Sigma_{\sigma(n)}$ centered at $q_k$ of uniform radius is an Euclidean graph in the harmonic coordinates in $U_q$ (that also contains all the points $F_0(p_k)$). This is clearly a contradiction when $m$ and $n$ are large because these uniform graphs cannot go up and down in arbitrarily narrow vertical strips. Note that they are also vertical graphs (not only graphs in the Euclidean sense over the tangent plane).

\begin{figure}
	\centering\includegraphics[width=12cm]{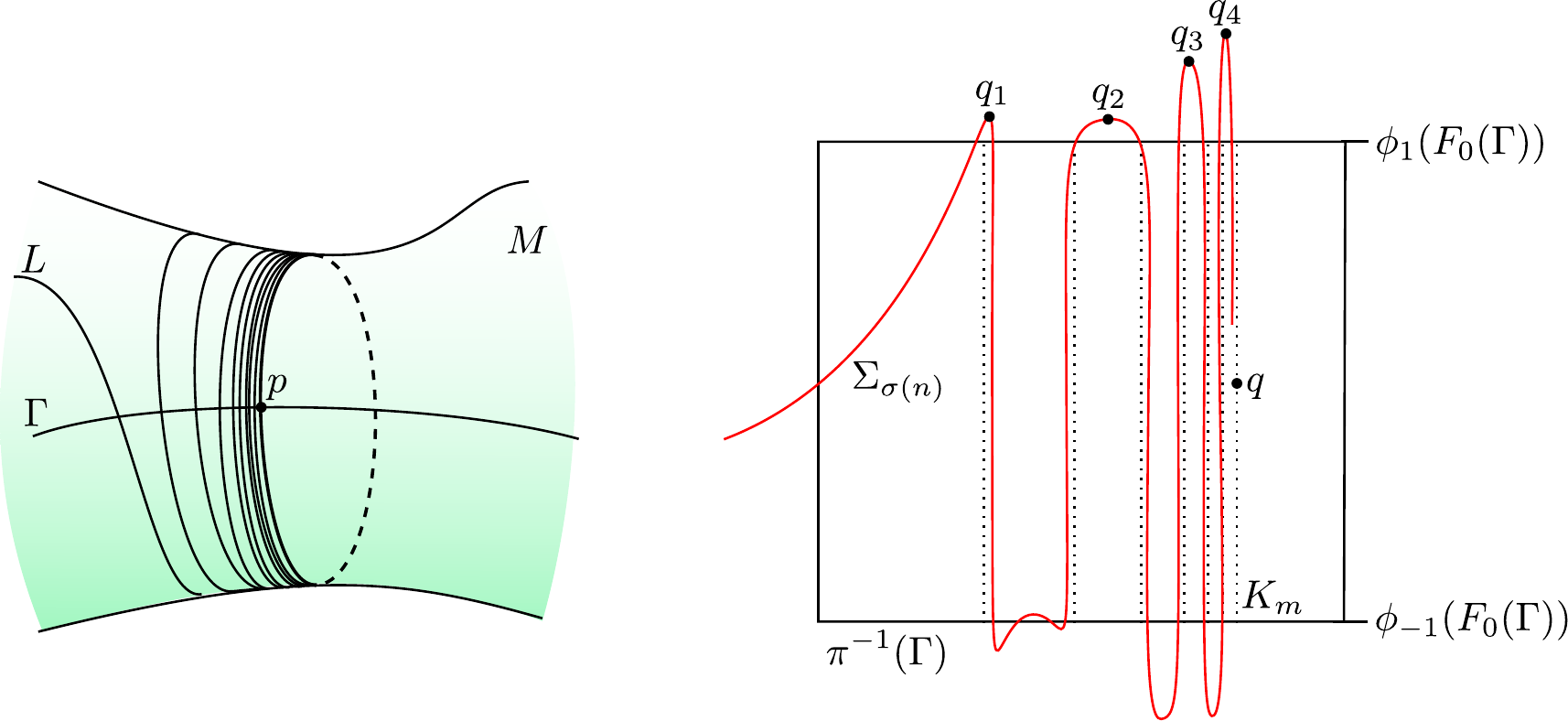}
	\caption{A divergence line that accumulates at some $p\in\overline\Omega$ (left) and the profile curve in the intersection $\pi^{-1}(\Gamma)\cap\Sigma_{\sigma(n)}$ (right). The dotted vertical lines represent the compact set $K_m$.}
	\label{fig:divergence1}
\end{figure} 

A similar argument discards the possibility that $L$ accumulates at some $p\in\partial\Omega$. In this case, $p$ belongs to the $\mu$-geodesic arc $L_\infty\subset\partial\Omega$ so the maxima $q_k$ in the above paragraph are bounded away from $\partial\Omega$  (the arcs $L_k$ converging to $L_\infty$ have fixed length). The contradiction arises again when $m$ is large because the uniform graph lemma implies that the Euclidean graphs in harmonic coordinates on $U_q$ must escape $\pi^{-1}(\Omega)$ by item (3) of Lemma~\ref{lem:uniform-graph-lemma} if $\pi(q_k)$ is close enough to $\partial\Omega$ (which is always the case for $m$ large).
\end{proof}

We show next that a divergence line cannot end at the interior of a component of $\partial\Omega$ where uniformly continuous boundary values have been prescribed. The proof of a similar result in $\mathbb{H}^2\times\mathbb{R}$~\cite[Prop.~4.8]{MRR} strongly relies on reflections about horizontal geodesics, so we will need a different argument giving a slightly more general result. Note that Lemma~\ref{lem:divline-boundary-values} applies if all the $u_n$ have continuous fixed values at $C$ but also when the value $\pm n$ is assigned to $u_n$ on $C$, which is the case of the sequence~\eqref{eqn:sequence} leading to the solution of the Jenkins--Serrin problem.

\begin{lemma}\label{lem:divline-boundary-values}
Let $\{u_n\}$ be a sequence of minimal graphs over $\Omega$ and let $C\subset\partial\Omega$ be an open $\mu$-convex arc (possibly $\mu$-geodesic). If each $u_n$ can be extended continuously to $\Omega\cup C$ and $\{u_n|_C-u_n(p)\}$ converges uniformly on $C$ to a continuous function $f:C\to\R$ for some $p\in C$, then no divergence line of $\{u_n\}$ ends at $p$.  
\end{lemma}

\begin{proof}
Assume by contradiction that a divergence line $L$ ends at $p$. Since $L$ is $\mu$-geodesic and $C$ is $\mu$-convex, their intersection at $p$ is transverse. We will parametrize $L$ as $\gamma:[0,\ell]\to L$ with unit speed and $\gamma(0)=p$, and define for $0<\varepsilon<\frac{\ell}{2}$ the compact sets
\begin{align*}
K_\varepsilon^-&=\bigcup_{t\in[-2,-1]}\phi_t(F_0(\gamma([\varepsilon,\ell-\varepsilon]))),&
K_\varepsilon^+&=\bigcup_{t\in[1,2]}\phi_t(F_0(\gamma([\varepsilon,\ell-\varepsilon]))).
\end{align*}
Let $v_n$ be a subsequence of $u_n-u_n(p_0)$, where $p_0\in L$, that uniformly converges to $\pi^{-1}(L)$ on compact subsets, in particular on $K_\varepsilon^+\cup K_\varepsilon^-$. Since each $v_n$ is continuous on $\Omega\cup C$, we can consider a further subsequence to assume without loss of generality that $v_n(p)>0$ for all $n$ (the case $v_n(p)\leq 0$ for all $n$ is similar). Choose a subarc $C'\subset C$ containing $p$ such that $v_n|_{C'}\geq-\frac{1}{2}$, which does not depend on $n$ by the uniform continuity on $C$ given by the statement. If $\Sigma_n$ denotes the graph of $v_n$, we can find a sequence $q_n\in\Sigma_n$ approaching $q_\infty=\phi_{-3/2}(\gamma(\epsilon))\in K_\epsilon^-$, see Figure~\ref{fig:divergence2}. This sequence verifies that $d_\Sigma(q_n,\partial\Sigma_{n})$ is uniformly bounded away from zero because $v_n\geq-\frac{1}{2}$ on $C'$ and the normal to $\Sigma_{n}$ at $q_n$ also approaches the normal to $\pi^{-1}(L)$ at $q_\infty$.

\begin{figure}
	\centering\includegraphics[width=6cm]{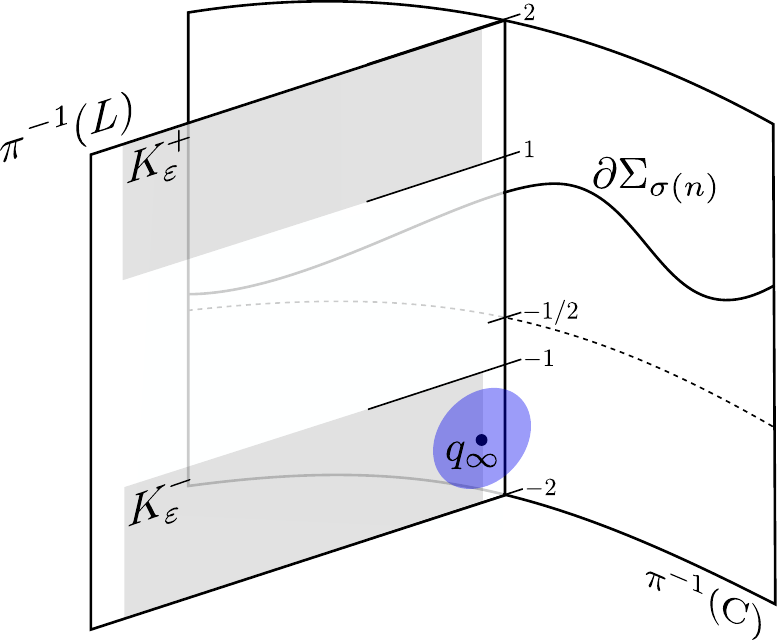}
	\caption{The compact sets $K_\varepsilon^{\pm}$ and the disk (in blue) that supports the graph which leaves the domain in the proof of Lemma~\ref{lem:divline-boundary-values}. The dashed line represents the uniform lower bound for $\partial\Sigma_{\sigma(n)}$.}
	\label{fig:divergence2}
\end{figure} 

If we start the above argument with $\varepsilon$ small enough so that $q_\infty$ lies in a prescribed harmonic coordinate chart centered at $\phi_{-3/2}(F_0(p))$, the uniform graph lemma~\ref{lem:uniform-graph-lemma} implies that $\Sigma_{n}$ is an Euclidean graph over an almost vertical plane transverse to $\pi^{-1}(C)$, so $\Sigma_{n}$ escapes $\pi^{-1}(\Omega)$ for small $\varepsilon$ (the uniform radius does not depend on $\varepsilon$), which is the desired contradiction. In this argument, we have used item (2) in Lemma~\ref{lem:uniform-graph-lemma} strongly, since it implies that the bend of the graphs (in harmonic coordinates) is uniformly bounded.
\end{proof}

Assume that the divergence lines of a sequence of minimal graphs $\{u_n\}$ are disjoint and denote by $\mathcal D$ the union of all such lines. We can find a subsequence $\{u_{\sigma(n)}\}$ such that items (A)-(C) below hold (this was proved in $\mathbb{H}^2\times\mathbb{R}$, see~\cite[Prop.~4.4, Lemma~4.6, Rmk.~4.7]{MRR} and the proof extends literally to the general case of Killing submersions). Let $\Omega_1$ be a connected component of $\Omega\sm\mathcal D$ and let $p_1\in\Omega_1$.
\begin{enumerate}[label=(\Alph*)]
 	\item The translated sequence $u_{\sigma(n)}-u_{\sigma(n)}(p_1)$ converges uniformly on compact subsets of $\Omega_1$ to a minimal graph $u_\infty^1$ over $\Omega_1$.
 
	\item If $L\subset\partial\Omega_1$ is a divergence line of $\{u_{\sigma(n)}\}$ and $\eta_L$ is the outer unit conormal to $\Omega_1$ along $L$, then $\eta_{\sigma(n)}\to\pm\eta_L$ and $u_{\sigma(n)}-u_{\sigma(n)}(p_1)\to\pm\infty$ uniformly on compact subsets of $L$ (the sign $\pm$ is the same for both limits). The flux of $u_{\sigma(n)}$ in $\Omega_1$ along $L$ with respect to $\eta_L$ gives (with the same choice of sign)
 	\[\lim_{n\to\infty}\Flux(u_{\sigma(n)},L)=\Flux(u_\infty^1,L)=\pm\Length_\mu(L).\]
 	
 	\item If $\Omega_2$ is an adjacent component of $\Omega\sm\mathcal D$ such that $L\subset\partial\Omega_1\cap \partial\Omega_2$, then $\{u_{\sigma(n)}-u_{\sigma(n)}(p_1)\}$ diverges uniformly to $\pm\infty$ on compact subsets of $\Omega_2$ provided that $\eta_{\sigma(n)}\to\pm\eta_L$ along $L$ (the sign $\pm$ is the same for both limits).
\end{enumerate}
In particular, two adjacent connected components of $\Omega\sm\mathcal D$ (at both sides of an isolated line of divergence $L$) cannot coincide. This topological obstruction discards some possible configurations of divergence lines. 

It is important to notice that some divergence lines can disappear after passing to a subsequence but no new ones are created. Our next goal is to refine a sequence of minimal graphs so that the divergence lines are disjoint, whence they enjoy the above properties (A)-(C). In $\mathbb{H}^2\times\mathbb{R}$, this is not difficult since there are finitely many vertices and any two vertices are joined by a unique geodesic. However, over a relatively compact Jenkins--Serrin domain in general Killing submersions there might be an uncountable infinite number of divergence lines (see Remark~\ref{rmk:infinite-divlines}) so we need again a new approach. We will have to deal with the two new situations depicted in Figure~\ref{fig:divergence-uncountable}:
\begin{itemize}
 	\item Infinitely many closed disjoint $\mu$-geodesics (as in the case of parallel circles in a round cylinder).
 	\item Infinitely many disjoint $\mu$-geodesics joining two fixed vertices (as in the case of meridians joining the north and south poles of the round sphere).
 \end{itemize}

\begin{remark}\label{rmk:infinite-divlines}
The above two situations can actually occur (we will prove later that this is not the case if the JS-conditions are satisfied). In $\mathbb{E}=\mathbb{S}^2\times\mathbb{R}$ with $\tau\equiv 0$ and $\mu\equiv 1$, choose $\Omega$ as a wedge of $\mathbb{S}^2$ bounded by two meridians and let $u_n$ take the value $n$ and $-n$ on these meridians, then $u_n$ spans a screw-motion invariant helicoid in $\mathbb{S}^2\times\mathbb{R}$. Therefore, the limit of $\{u_n\}$ is a foliation of $\Omega\times\R$ by vertical cylinders, i.e., all geodesics of $\Omega$ joining the its vertices are divergence lines of $\{u_n\}$.

Likewise, in $\mathbb{E}=(\mathbb{S}^1\times\R)\times\R$ with $\tau\equiv 0$ and $\mu\equiv 1$, consider the relatively compact domain $\Omega=\mathbb{S}^1\times(-1,1)$ and let $u_n$ take the values $\pm n$ on $\mathbb{S}^1\times\{\pm 1\}$. Then $u_n$ spans a graph over $\Omega$ which is totally geodesic: it is a plane in the Euclidean space $\R^2\times\R$, the universal cover of $\E$. These planes converge to vertical planes everywhere (i.e., tangent to the second factor $\R$), so the divergence lines of $\{u_n\}$ are the closed geodesics $\mathbb{S}^1\times\{t_0\}$ with $-1<t_0<1$.
\end{remark}

\begin{figure}
	\centering\includegraphics[width=9cm]{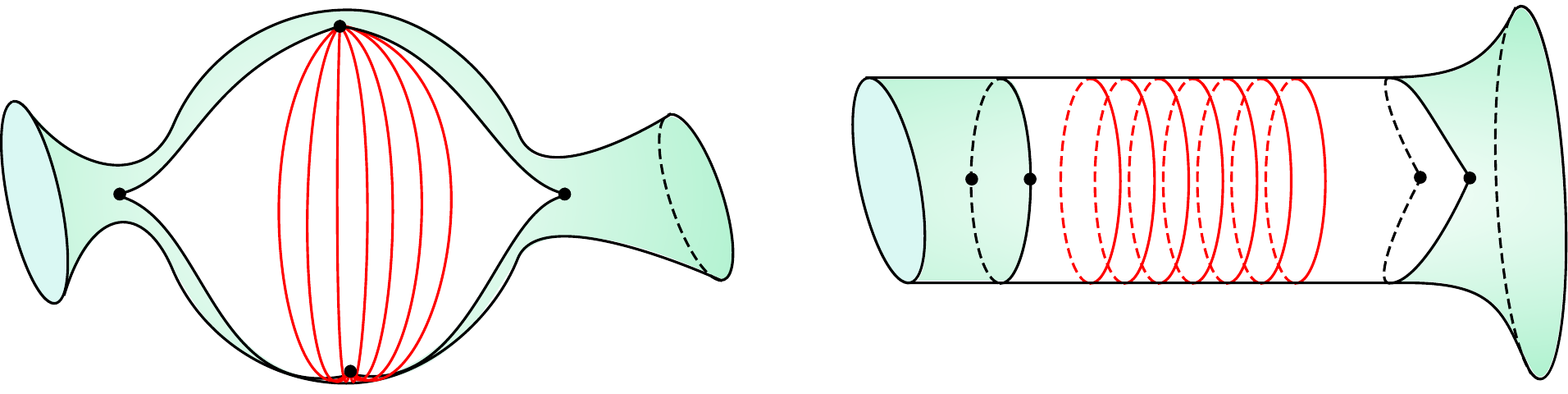}
	\caption{Two Jenkins--Serrin quadrilaterals containing open subsets isometric to part of a sphere (left) or a cylinder (right) so they have uncountably many potential divergence lines.}
	\label{fig:divergence-uncountable}
\end{figure} 

We will group the divergence lines in isotopy classes of closed $\mu$-geodesics or open $\mu$-geodesic arcs (with respect to their common endpoints in the latter case, i.e., the vertices remain fixed under the isotopy). Observe that, given such a class $\mathcal I$ and disjoint $L_1,L_2\in\mathcal I$,  the closed curve $\overline L_1\cup\overline L_2$ is the boundary of a topological annulus (resp.\ disk) contained in $\Omega$ if $\mathcal I$ consists of closed curves (resp.\ open arcs). 

\begin{definition}\label{defi:isotopy-region}
Assume that all divergence lines are disjoint and their union is $\mathcal D$. The connected components of $\Omega\sm\mathcal D$ will be called \emph{convergence components}. 

 Given a isotopy class of divergence lines $\mathcal I$ and $L_1,L_2\in\mathcal{I}$, we will denote by $R(L_1,L_2)\subset\Omega$ the open disk or annulus with boundary $\overline L_1\cup\overline L_2$. We will call \emph{isotopy region} the disk or annulus $R_{\mathcal I}=\cup_{L_1,L_2\in\mathcal I}R(L_1,L_2)$.
\end{definition}

The closure of the divergence set (proved next in Lemma~\ref{lemma:closure-divlines}) will play a crucial role. Recall that a $\mu$-geodesic $L$ is a limit of $\mu$-geodesics $L_n$ if and only if there is a sequence $p_n\in L_n$ converging to some $p\in L$ such that the unit tangent vectors to $L_n$ at $p_n$ converge to an unit tangent vector to $L$ at $p$. This convergence is uniform in compact subsets (of the common arc-length parameter of these $\mu$-geodesics) due to the smooth dependence of $\mu$-geodesics on their initial conditions.

\begin{lemma}\label{lemma:closure-divlines}
Let $\{u_n\}$ be a sequence of minimal graphs over $\Omega$.
\begin{enumerate}[label=\emph{(\arabic*)}]
	\item Any limit of divergence lines of $\{u_n\}$ is either a $\mu$-geodesic component of $\partial\Omega$ or again a divergence line of $\{u_n\}$.
	\item Each isotopy class of divergence lines not isotopic to any $\mu$-geodesic component of $\partial\Omega$ is closed (with respect to the convergence of $\mu$-geodesics).
\end{enumerate}
\end{lemma}

\begin{proof}
Let $L_n$ be a convergent sequence of divergence lines not converging to a component of $\partial\Omega$, so there exist $p_n\in L_n$ converging to some $p_\infty\in \Omega$ with unit tangent vectors $v_n$ to $L_n$ at $p_n$ that converge to a unit vector $v_\infty$ at $p_\infty$, and let $L_\infty$ be $\mu$-geodesic through $p_\infty$ with unit tangent vector $v_\infty$. Observe that $\pi^{-1}(L_n)$ converge as minimal surfaces to $\pi^{-1}(L_\infty)$ in the $\mathcal C^k$-topology on compact subsets for all $k$. Since we can find subsequences of $u_n$ spanning graphs converging to any $\pi^{-1}(L_n)$ on compact subsets, a diagonal argument shows that we can also find a subsequence of $u_n$ converging to $\pi^{-1}(L_\infty)$, so $L_\infty$ is a divergence line.

Item (2) follows readily from item (1) since a limit of simple closed $\mu$-geodesics in some isotopy class is a simple closed $\mu$-geodesic in the same isotopy class.
\end{proof}

Next, we will prove that a sequence of minimal graphs over a Jenkins--Serrin domain can be refined so that the divergence lines become disjoint and can be grouped into finitely many isotopy classes, and each isotopy class $\mathcal I$ defines the exclusive region $R_{\mathcal I}$ (see Definition~\ref{defi:isotopy-region}) containing the (possibly uncountably many) divergence lines of $\mathcal I$ but no other lines in other isotopy classes. Furthermore, the lines of $\mathcal I$ separate countably many regions whose \emph{divergence heights} are linearly ordered, that is, we only go up (or down) whenever we go through $R_{\mathcal I}$ transversely to the lines of $\mathcal I$. It was suggested in~\cite[Rmk.~4.5]{MRR} that disjoint divergence lines can be obtained even in the uncountable case, so next result settles this question.

\begin{proposition}\label{prop:disjoint-lines}
Given a sequence of minimal graphs $\{u_n\}$ over a Jenkins--Serrin domain $\Omega$, there is a subsequence $\{u_{\sigma(n)}\}$ whose divergence lines are pairwise disjoint, whence it has finitely many nonempty isotopy classes of divergence lines. 

Let $\mathcal I$ be one of such isotopy classes with at least two elements and assume that no $\mu$-geodesic component of $\partial\Omega$ is isotopic to the elements of $\mathcal I$. 
\begin{enumerate}[label=\emph{(\arabic*)}]
	\item There is a linear order $\prec$ in $\mathcal I$ such that $L_1\prec L\prec L_2$ if and only if $L\subset R(L_1,L_2)$.
	\item The ordered set $(\mathcal I,\prec)$ has maximum and minimum elements $L_+,L_-\in\mathcal I$.
	\item All the curves of $\mathcal I$ have the same $\mu$-length.
	\item The order $\prec$ can be choosen uniquely (and we will do so) by assuming that $\eta_{\sigma(n)}$ converges to the unit inner conormal $\eta_{L_-}$ to $R_{\mathcal I}=R(L_-,L_+)$ along $L_-$. 
	\item If $L\in\mathcal I$ is different from $L_-$, the normalized gradients $\eta_{\sigma(n)}$ converge to the outer unit conormal $\eta_L$ to $R(L_-,L)$ along $L$.
	\item Denote by $\mathcal D$ the union of all divergence lines of $\{u_{\sigma(n)}\}$. There are unique distinct convergence components $\Omega_\pm\subset\Omega\sm R_\mathcal I$ with $L_\pm\subset\partial\Omega_\pm$.
	\begin{enumerate}[label=\emph{(\alph*)}]
		\item Given $p\in\Omega_-$ (resp.\ $p\in\Omega_+$), the sequence $\{u_{\sigma(n)}-u_{\sigma(n)}(p)\}$ diverges uniformly to $+\infty$ (resp.\ $-\infty$) on compact subsets of $\Omega_+$ (resp.\ $\Omega_-$).
		\item If $\Omega_0=R(L_1,L_2)\subset R_\mathcal I$ is a convergence component with $L_1\prec L_2$ and $p\in\Omega_0$, then $\{u_{\sigma(n)}-u_{\sigma(n)}(p)\}$ diverges uniformly to $+\infty$ (resp.\ $-\infty$) on compact subsets of $\overline{R(L_2,L_+)}\cup\Omega_+$ (resp.\ $\Omega_-\cup\overline{R(L_-,L_1)}$).
	\end{enumerate}
\end{enumerate}
\end{proposition}

\begin{proof}
Let $\{p_m\}$ be a countable and dense subset of $\Omega$ and $\mathcal D_1$ be the union of all divergence lines of $\{u_n\}$, which is a relatively closed subset of $\Omega$. Let $L_1\subset\mathcal D_1$ be a divergence line closest to $p_1$ (possible not unique), and let $\{u_n^1\}$ be a subsequence of $\{u_n\}$ such that the graphs of $u_n^1-u_n^1(p)$ converge uniformly on compact subsets to $\pi^{-1}(L_1)$ for a fixed $p\in L_1$. Note that all divergence lines of $\{u_n^1\}$ (other than $L_1$) do not intersect $L_1$ because the normalized gradients $\eta_n^1$ converge along $L_1$ to an unit normal to $L_1$, and any other $\mu$-geodesic intersects $L_1$ transversely.

By induction, suppose that we have a subsequence $\{u_n^{k-1}\}$ and let $\mathcal D_k$ its (relatively closed) set of divergence lines. Consider a divergence line $L_k\subset\mathcal D_k$ closest to $p_k$ and define $\{u_n^k\}$ as a subsequence of $\{u_n^{k-1}\}$ such that the graphs of $u_n^k-u_n^k(p)$ converge uniformly on compact subsets to $\pi^{-1}(L_k)$ for a fixed $p\in L_k$. As in the above argument, this leads to divergence lines $L_1,\ldots,L_k$ which are pairwise disjoint (after removing possible repetitions) and also disjoint with any other divergence line of $\{u_n^k\}$. We will consider the diagonal sequence $u_{\sigma(n)}=u_n^n$ and the sequence of pairwise disjoint divergence lines $\{L_m\}$ we have constructed in this way: each $L_m$ is disjoint with any other divergence line of $\{u_{\sigma(n)}\}$. Observe that the limits of elements of $\{L_m\}$ are disjoint too (if two limits intersect, then there must be sufficiently close elements of $\{L_m\}$ that also intersect since the convergence of $\mu$-geodesics is uniform on compact subsets and their intersections are always transverse). 

If $\{L_m\}$ is either finite (after removing repetitions) or contains all divergence lines of $u_{\sigma(n)}$, then we are done since this means that all divergence lines are disjoint. So, we will assume that $L$ is divergence line not in the sequence and prove that it is a limit of elements of $\{L_m\}$, which also proves that all divergence lines are disjoint. No point of $L$ can be at positive distance from $\cup L_m$ (otherwise, as the $p_m$ are dense, another divergence line different from any of the $L_m$ should have been chosen in the process). Therefore, there is a sequence of points $x_k\in L_{m_k}$ converging to some $x_\infty\in L$. If $v_k$ is an unit normal to $L_{m_k}$ at $x_k$, then up to its sign it must converge to an unit normal to $L$ at $x_{\infty}$ (otherwise, some of the $L_{m}$ would intersect $L$). Therefore, $L$ is a limit of elements of $\{L_m\}$ and we are done.

Since $\Omega$ has finite topology, it is diffeomorphic to a surface of finite genus minus some points given by its boundary curves. It is well known that such a surface cannot have infinite non-homotopic disjoint closed curves (e.g., see~\cite[Prop.~2.3.3]{Arettines}); also, it is clear that there cannot be infinitely disjoint non-isotopic arcs joining vertices of $\Omega$. Since we have already shown that divergence lines are disjoint, the number of nonempty isotopy classes is finite. Observe that there are no nulhomotopic divergence lines by the maximum principle. 

Let $\mathcal I$ be an isotopy class as in the statement, and let us prove items (1)--(6):
\begin{enumerate}[label=(\arabic*)]
	\item Given fixed distinct $L_1,L_2\in\mathcal I$, we set $L_1\prec L_2$. Then, there are three possible scenarios for another $L\in\mathcal I$, namely $L\in R(L_1,L_2)$, $L_1\in R(L,L_2)$ or $L_2\in R(L,L_1)$, in which case we set $L_1\prec L\prec L_2$ or $L\prec L_1\prec L_2$ or $L_1\prec L_2\prec L$, respectively. Given another $L'\in\mathcal I$, we can use the same argument to compare $L$ and $L'$, which easily leads to a total order in $\mathcal I$.
	\item The region $R_{\mathcal I}=\cup_{L_1,L_2\in\mathcal I}R(L_1,L_2)$ is nonempty because $\mathcal I$ contains at least two elements, so $R_{\mathcal I}$ is again a topological disk (resp.\ annulus) if $\mathcal I$ consists of arcs (resp.\ closed curves). Given $p\in\partial R_{\mathcal I}$ not a vertex of $\Omega$, there is a (possibly constant) sequence $\{L_n\}$ in $\mathcal I$ which accumulates at $p$, and hence a limit $\mu$-geodesic $L$ through $p$. Since no component of $\partial\Omega$ is isotopic to the elements of $\mathcal I$ by hypothesis, Lemma~\ref{lemma:closure-divlines} ensures that $L\in\mathcal I$. This divergence line $L$ must also be either a maximum or a minimum of $\prec$ by construction. Since $\partial R_{\mathcal I}$ cannot consist of just one divergence line by the maximum principle, we infer the existence of both the maximum $L_+$ and the minimum $L_-$, whence $R_{\mathcal I}=R(L_-,L_+)$. 

	\item[(3)--(5)] Given distinct $L_1,L_2\in\mathcal I$, the divergence theorem on $R(L_1,L_2)$ gives
	\begin{equation}\label{eqn:lem-isotopy:eqn1}
	0=\int_{L_1\cup L_2}\langle X_{u_{\sigma(n)}},\eta\rangle=\Flux(u_{\sigma(n)},L_1)+\Flux(u_{\sigma(n)},L_2),
	\end{equation}
	where the flux is computed with respect to the outer unit conormal $\eta$ to $R(L_1,L_2)$ along its boundary. Taking limits in~\eqref{eqn:lem-isotopy:eqn1} as $n\to\infty$, we get that $0=\pm\Length_\mu(L_1)\pm\Length_\mu(L_2)$, where the signs depend on whether $\eta_{\sigma(n)}$ converges to $\eta$ or $-\eta$. Clearly, both signs must be different so the result of the sum is zero, which proves item (3). In the case of $R(L_-,L_+)$, this means that $\eta_{\sigma(n)}$ converges to the inner conormal to $R_{\mathcal I}$ along $L_-$ and to the outer conormal to $R_{\mathcal I}$ along $L_+$ up to reversing the order, so we have item (4). The very same argument proves item (5).
	\item[(6)] Assume by contradiction that there is no such component $\Omega_+$ or $\Omega_-$. This means that there is a sequence of divergence lines outside $R_{\mathcal I}$ that accumulate at some $p\in\partial R_{\mathcal I}$. Since there only finitely many non-empty isotopy classes of them, we can assume that they all belong to the same class, but this is clearly a contradiction since isotopy classes are closed and hence $\partial R_{\mathcal I}$ would intersect an element of an isotopy class other than $\mathcal I$. This proves the existence of the components $\Omega_\pm$ given in the statement.

	Subitems (a) and (b) can essentially proved in the same way and reflect the idea that all convergence components in $R_{\mathcal I}$ lie at different levels, which are linearly ordered, and this also applies to the adjacent ones $\Omega_\pm$. We will only consider the case $p\in\Omega_-$ as in item (a), since other cases are analogous. We first recall that $u_{\sigma(n)}-u_{\sigma(n)}(p)$ diverge uniformly to $+\infty$ on compact subsets of $L_-$ because $L_-\subset\partial\Omega_-$. Assume by contradiction that $u_{\sigma(n)}(p_0)-u_{\sigma(n)}(p)$ remains bounded from above for some $p_0\in R(L_-,L_+)\cup L_+\cup\Omega_+$ (after possibly taking a further subsequence). Let $\gamma:[0,1]\to\Omega$ be a regular curve joining $p$ and $p_0$ meeting transversely the elements of $\mathcal I$ following the order given by $\prec$. The value of $u_{\sigma(n)}(\gamma(t))-u_{\sigma(n)}(p)$ becomes arbitrarily high as $n\to\infty$ for points $\gamma(t)\in L_-$ and then remains bounded from above at $\gamma(1)=p_0$. This means that the graph of $u_{\sigma(n)}$ contains arbitrarily vertical directions that must subconverge to part of a cylinder over a divergence line (thus some $L\in\mathcal I$). However, since the value of the graph decreases from an arbitrarily high value as we cross $L$, the normalized gradient $\eta_{\sigma(n)}$ must converge to the inner conormal to $R(L_-,L)$, in contradiction to item (5).\qedhere
\end{enumerate}
\end{proof}

\begin{remark}
An interesting fact that may help understand the nature of the subsequence $\{u_{\sigma(n)}\}$ given by Proposition~\ref{prop:disjoint-lines} is that all its divergence lines are not \emph{removable}, in the sense that any further subsequence of $\{u_{\sigma(n)}\}$ has the same set of divergence lines. This is a consequence of the diagonal argument in the proof.
\end{remark}

Under the JS-conditions, there will not be divergence lines in the isotopy class of a boundary component of type $A_i$ or $B_i$ (Lemma~\ref{lemma:isotopy-boundary}). However, most of the ideas of Proposition~\ref{prop:disjoint-lines} can be adapted easily in the case that there is such a $\mu$-geodesic $\Gamma\subset\partial\Omega$ (recall that the sides of $\Omega$ are not divergence lines, which must be interior to $\Omega$ by definition). We can extend the order $\prec$ to $\mathcal I\cup\{\Gamma\}$ and $\Gamma$ acts as a maximum or minimum, in which case, one of the domains $\Omega_+$ or $\Omega_-$ is not defined. 

Also, an isotopy class $\mathcal I$ with just one element is not a problem since it can be understood using the above items (A)-(C) as in~\cite{MRR}. Because of the following corollaries, the structure of the divergence set is as depicted in Figure~\ref{fig:RI}.

\begin{figure}
	\centering\includegraphics[width=\textwidth]{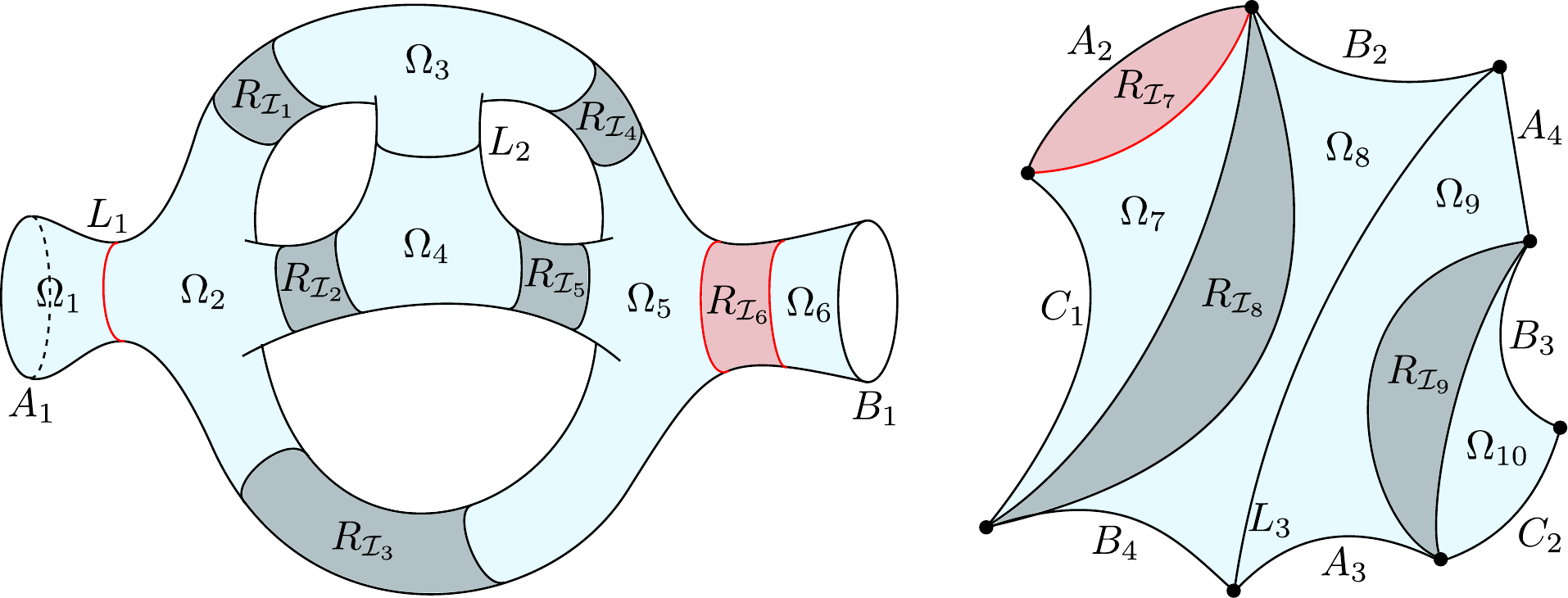}
	\caption{Two possible configurations of the divergence set in a surface of genus three (left) and in a genus zero octagon (right) with some isotopy classes with at least two elements (contained in the dark regions) plus several isolated divergence lines which are unique in their isotopy classes. Note that $L_1$, $R_{\mathcal I_6}$ and $R_{\mathcal I_7}$ (in red color) cannot exist under the JS-conditions by Lemma~\ref{lemma:isotopy-boundary}.}
	\label{fig:RI}
\end{figure}

\begin{corollary}\label{coro:components-are-polygons}
Under the assumptions of Proposition~\ref{prop:disjoint-lines}, any connected component of $\Omega\sm\mathcal D$ is an inscribed $\mu$-polygon.
\end{corollary}

\begin{proof}
A component $\Omega_0\subset\Omega\sm\mathcal D$ is bounded by disjoint $\mu$-geodesic lines that can be either arcs joining two vertices of $\Omega$ or closed curves. We only need to prove that the number of such lines is finite, so assume by contradiction that it is not. Since there are finitely-many isotopy classes of divergence lines and finitely-many isotopy classes of the sides of $\partial\Omega$, we conclude that there are at least three (infinitely many, indeed) isotopic distinct divergence lines $L_1,L_2,L_3$ that can be assumed to satisfy $L_1\prec L_2\prec L_3$. This is clearly a contradiction because the only two possible connected components of $\Omega\sm\mathcal D$ with $L_1\cup L_3$ in their boundary are $R(L_1,L_3)$ and $\Omega\sm\overline{R(L_1,L_3)}$, none of which has $L_2$ as part of the boundary.
\end{proof}

\begin{corollary}\label{coro:finite-convergence-components}
Under the assumptions of Proposition~\ref{prop:disjoint-lines}, $(\Omega\sm\mathcal D)\sm\cup_{\mathcal I}R_{\mathcal I}$ has finitely many connected components.
\end{corollary}

\begin{proof}
If $m$ is the number of nonempty isotopy classes of divergence lines, there are at most $2m$ lines that can act as boundary components of the connected components of $(\Omega\sm\mathcal D)\sm\cup_{\mathcal I}R_{\mathcal I}$. Since each of these lines can only be in the boundary of two such connected components, we conclude that the number is finite.
\end{proof}

\subsection{Existence of Jenkins--Serrin graphs}\label{subsec:existence-JS}

Consider a Jenkins--Serrin problem as in Definition~\ref{def:JS-problem}. Theorems~\ref{thm:existence} and~\ref{thm:generalized-existence} yield the existence of minimal graphs $u_n\in\mathcal C^\infty(\Omega)$ with the following boundary conditions:
\begin{equation}\label{eqn:sequence}
u_n=\begin{cases}
n&\text{on }\cup A_i,\\
-n&\text{on }\cup B_i,\\
f_{i,n}&\text{on }C_i,
\end{cases}
\end{equation}
where $f_i$ the continuous function prescribed on the side $C_i$, truncated as $f_{i,n}=\min\{\max\{f_i,-n\},n\}$.
Assume henceforth that the JS-conditions in Theorem~\ref{thm:JS} are satisfied and there is at least one side $A_i$ or $B_i$ (otherwise the JS-conditions become trivial and Theorem~\ref{thm:JS} follows from Theorems~\ref{thm:existence} and~\ref{thm:generalized-existence}).

We will find a subsequence of $\{u_n\}$ that converges uniformly on compact sets of all $\Omega$ (i.e., without divergence lines) and achieves the desired boundary values. By Proposition~\ref{prop:disjoint-lines}, we can start with a subsequence $\{u_{\sigma(n)}\}$ such that all divergence lines are disjoint and can be grouped in finitely-many isotopy classes of either closed $\mu$-geodesics or $\mu$-geodesic arcs joining a pairs of vertices of $\Omega$ (see also Lemmas~\ref{lem:divline-properly-embedded} and~\ref{lem:divline-boundary-values}). We will denote by $\mathcal D$ the union of all divergence lines of $\{u_{\sigma(n)}\}$.

The next lemma completes the picture of the divergence set given by Proposition~\ref{prop:disjoint-lines} by additionally assuming the JS conditions. Note that the flux limits still hold without such conditions (e.g., see~\cite[Lemma~3.6]{MRR}).

\begin{lemma}\label{lemma:isotopy-boundary}
Under the JS-conditions and computing the flux with respect to the outer conormal to $\Omega$ along its boundary, we have that
\begin{align*}
  \lim_{n\to\infty}\Flux(u_n,A_i)&=\Length_\mu(A_i),& \lim_{n\to\infty}\Flux(u_n,B_i)&=-\Length_\mu(B_i).
\end{align*}
Consequently, there are no divergence lines of $\{u_n\}$ isotopic to any $A_i$ or $B_i$.
\end{lemma}

\begin{proof}
We will reason for one of the components $A_i$ (the reasoning is completely analogous for a component $B_i$). We will first prove that divergence lines of $\{u_n\}$ cannot accumulate at $A_i$. By contraction, if they happen to accumulate, then $A_i$ is a limit of divergence lines $\{L_n\}$ in the same isotopy class as $A_i$ (isotopy classes of $\mu$-geodesics are closed). Since all the $L_n$ have the same $\mu$-length by Proposition~\ref{prop:disjoint-lines}, so does $A_i$ as a limit $\mu$-geodesic. Therefore, $\mathcal P=\overline A_i\cup \overline L_1$ is an inscribed $\mu$-polygon with $\gamma(\mathcal P)=2\alpha(\mathcal P)$, which is not possible by the JS-conditions.

This means that there is a convergence component $\Omega_i$ with $A_i\subset\partial\Omega_i$. The boundary $\mathcal P=\partial\Omega_i$ is an inscribed $\mu$-polygon and consists of curves of $\partial\Omega$ plus finitely-many divergence lines of $\{u_n\}$ (possibly none if $\Omega_i=\Omega$). The functions $v_n=u_n-n$ take values in $[-2n,0]$ and form a pointwise decreasing sequence by Proposition~\ref{prop:uniqueness}. Since $\{v_n\}$ has the same divergence lines as $\{u_n\}$, it follows that $\{v_n\}$ pointwise diverges to $-\infty$ on $(\cup B_i)\cap\mathcal P$ and on $\mathcal P\sm\partial\Omega$. The divergence theorem on $\Omega_i$ with respect to its outer unit conormal $\eta$ yields 
\begin{equation}\label{lemma:isotopy-boundary:eqn1}
0=\int_{(\cup A_i)\cap \mathcal P}\langle X_{u_n},\eta\rangle+\int_{(\cup B_i)\cap \mathcal P}\langle X_{u_n},\eta\rangle+\int_{\mathcal P\sm\partial\Omega}\langle X_{u_n},\eta\rangle.
\end{equation}
Notice that this computation can be done indistinctly for $u_n$ or $v_n$ since they differ in a vertical translation. The first summand in~\eqref{lemma:isotopy-boundary:eqn1} is bounded by $\alpha(\mathcal P)$ in absolute value, whereas the second and the third summands converge to $-\beta(\mathcal P)$ and $\alpha(\mathcal P)+\beta(\mathcal P)-\gamma(\mathcal P)$, respectively, as $n\to\infty$ (by the same argument as in the proof of Lemma~\ref{lemma:flux-boundary}). If $\Omega_i\neq\Omega$, the limit of~\eqref{lemma:isotopy-boundary:eqn1} as $n\to\infty$ gives $\gamma(\mathcal P)-2\alpha(\mathcal P)\leq 0$, which is not possible by the JS-conditions, whence $\Omega_i=\Omega$ and $\mathcal P=\partial\Omega$. In this case, because of the condition $\alpha(\partial\Omega)=\beta(\partial\Omega)$, the limit of~\eqref{lemma:isotopy-boundary:eqn1} gives 
\begin{equation}\label{lemma:isotopy-boundary:eqn2}
\lim_{n\to\infty}\sum_i\Flux(u_n, A_i)=\beta(\partial\Omega)=\alpha(\partial\Omega)=\sum_i\Length_\mu(A_i).
\end{equation}
Since $\Flux(u_n,A_i)\leq\Length_\mu(A_j)$ for all $i$, this identity easily particularizes to each component $A_i$, that is, the sum in~\eqref{lemma:isotopy-boundary:eqn2} can be removed.

Finally, assume by contradiction that there is a divergence line $L$ isotopic to $A_i$, so the open region $R(A_i,L)\subset\Omega$ is either a disk or an annulus depending on whether $A_i$ is an open arc or a closed curve. The very same argument as in item (3) of Proposition~\ref{prop:disjoint-lines} implies that $\Length_\mu(A_i)=\Length_\mu(L)$, where we now take into account~\eqref{lemma:isotopy-boundary:eqn2}. This in turn implies that the inscribed $\mu$-polygon $\mathcal P=\partial R(A_i,L)$ verifies $\gamma(\mathcal P)=2\alpha(\mathcal P)$, which is the desired contradiction.
\end{proof}

We prove next that, under the JS-conditions, that any subsequence of $\{u_n\}$ can be further refined to get rid of all divergence lines.

\begin{lemma}\label{lemma:convergence-JS}
For each $p\in\Omega$, any subsequence of $\{u_n\}$ has a further subsequence $\{u_{\sigma(n)}\}$ such that $\{u_{\sigma(n)}-u_{\sigma(n)}(p)\}$ uniformly converges on compact subsets of $\Omega$ to a solution of the minimal surface equation.
\end{lemma}

\begin{proof}
Take a subsequence $\{u_{\sigma(n)}\}$ of the original subsequence having only disjoint divergence lines using Proposition~\ref{prop:disjoint-lines}, and assume by contradiction that the union of all divergence lines $\mathcal D$ is nonempty. Note that $\mathcal D\neq\Omega$ because, by Lemmas~\ref{lemma:closure-divlines} and~\ref{lemma:isotopy-boundary}, isotopy classes of $\mu$-geodesics are closed and $\mathcal D=\Omega$ would imply the existence of divergence lines isotopic to some $A_i$ or $B_i$. Therefore, $\Omega\sm\mathcal D\neq\emptyset$ and Lemma~\ref{lemma:isotopy-boundary} ensures the existence of a convergence component $\Omega_1\subset\Omega\sm\mathcal D$ whose boundary contains one of the sides $B_i$ (we can argue similarly if we assume that it contains one of the sides $A_i$); in particular, $\Omega_1$ is disjoint with any of the regions $R_{\mathcal I}$. Note that $\{u_{\sigma(n)}-u_{\sigma(n)}(p_1)\}$ converges uniformly on compact subsets of $\Omega_1$ for a fixed $p_1\in\Omega_1$ and $\partial\Omega_1$ is an inscribed $\mu$-polygon by Corollary~\ref{coro:components-are-polygons}. If there is a divergence line $L_1\subset\partial\Omega_1$ such that $\{u_{\sigma(n)}-u_{\sigma(n)}(p_1)\}$ diverges to $+\infty$ along $L_1$, then the normalized gradients $\eta_{\sigma(n)}$ converge to the outer conormal to $\Omega_1$ along $L_1$. We can define a new convergence component $\Omega_2$ as follows:
\begin{enumerate}
	\item If $L_1=L_-$ and $\Omega_1=\Omega_-$ for some isotopy class of divergence lines $\mathcal I$ with at least two elements, then define $\Omega_2=\Omega_+$ (with the notation of Proposition~\ref{prop:disjoint-lines}). Hence, by item (6) of that Proposition, $\{u_{\sigma(n)}-u_{\sigma(n)}(p_2)\}$ converges uniformly on compact subsets of $\Omega_2$ for any $p_2\in\Omega_2$ and diverges to $-\infty$ on $\Omega_1\cup\overline R_{\mathcal I}$.
	\item Otherwise, $L_1$ is unique in its isotopy class so there is an adjacent component $\Omega_2\subset\Omega\sm\mathcal D$ such that $L_1\subset\partial\Omega_1\cap\partial\Omega_2$. By the already discussed results in~\cite{MRR}, it follows that $\{u_{\sigma(n)}-u_{\sigma(n)}(p_2)\}$ converges uniformly to a minimal graph on compact subsets of $\Omega_2$ for any $p_2\in\Omega_2$ and diverges to $-\infty$ on $\Omega_1\cup L_1$.
\end{enumerate}
Either way, we found a component $\Omega_2$ at a higher level than $\Omega_1$. This process can be repeated to produce a sequence of convergence components $\Omega_1,\Omega_2,\ldots$ which are pairwise disjoint and disjoint with any of the regions $R_{\mathcal I}$. There is a finite number of such convergence components by Corollary~\ref{coro:finite-convergence-components}, so the aforesaid process must end after finitely many steps. This means that we can find a component $\Omega_k\subset\Omega\sm\mathcal D$ and $p_k\in\Omega_k$ such that $\{u_{\sigma(n)}-u_{\sigma(n)}(p_k)\}$ goes to $-\infty$ along all the divergence lines in $\mathcal P_k=\partial\Omega_k$. The flux of $u_{\sigma(n)}-u_{\sigma(n)}(p_k)$ along any $B_i\subset\mathcal P_k$ tends to $-\Length_\mu(B_i)$ by Lemma~\ref{lemma:isotopy-boundary}, so a computation similar to~\eqref{lemma:isotopy-boundary:eqn1} gives $\gamma(\mathcal P_k)-2\beta(\mathcal P_k)\leq 0$ in the limit as $n\to\infty$, which gives a contradiction with the JS-conditions.
\end{proof}

Now we have all ingredients to finish the proof of the main theorem.

\begin{proof}[Proof of the existence in Theorem~\ref{thm:JS}]
Fix $p\in\Omega$ and let $\{u_{\sigma(n)}-u_n(p)\}$ be the subsequence given by Lemma~\ref{lemma:convergence-JS}. We will assume first that $u_n(p)$ is bounded, so $\{u_{\sigma(n)}(p)\}\to a\in\R$ (up to a subsequence). This easily implies that $\{u_{\sigma(n)}-a\}$, or equivalently $\{u_{\sigma(n)}\}$, converges uniformly on compact subsets of $\Omega$ to a minimal graph $u$. Propositions~\ref{prop:boundary-value} and~\ref{prop:boundary-infinity} ensure that $u$ achieves the desired boundary values along the components of $\partial\Omega$.

Now suppose that $\{u_n(p)\}$ is unbounded and also that $\{u_{\sigma(n)}(p)\}\to+\infty$ up to considering a further subsequence (the case $\{u_{\sigma(n)}(p)\}\to-\infty$ follows similarly). Let $u$ be the limit of $\{u_{\sigma(n)}-u_{\sigma(n)}(p)\}$, which has the correct asymptotic value $-\infty$ on $\cup B_i$ by Proposition~\ref{prop:boundary-infinity}. Let $a_n=n-u_n(p)$ be the value that each graph $u_n-u_n(p)$ takes on $\cup A_i$, which is a sequence of positive numbers by the maximum principle ($F_0$ is a minimal section). Notice that we need that $a_n\to+\infty$ in order to get the desired solution. We will distinguish two cases:
\begin{enumerate}
 	\item If $\cup C_i\neq\emptyset$, then $u$ takes the value $-\infty$ on each $C_i$. If $\{a_{\sigma(n)}\}$ is bounded, we can pass to a subsequence such that $\{a_{\sigma(n)}\}\to a\in\R$ so that $u$ takes the constant value $a$ on $\cup A_i$; otherwise, we can pass a subsequence such that $\{a_{\sigma(n)}\}\to+\infty$ increasingly and $u$ takes the value $+\infty$ on $\cup A_i$. Either way, by computing the flux of $u$ across $\partial\Omega$ we get $2\alpha(\partial\Omega)\geq \gamma(\partial\Omega)$ which is not compatible with the JS-conditions.
 	\item If $\cup C_i=\emptyset$, then we apply a similar argument. If $\{a_{\sigma(n)}\}$ is not bounded, we get the desired solution with the correct boundary values. If $\{a_{\sigma(n)}\}$ is bounded, we find a graph $u$ in all $\Omega$ with constant value on each $A_i$ and $-\infty$ value on each $B_i$. This leads to $\alpha(\partial\Omega)>\beta(\partial\Omega)$, which is a contradiction.\qedhere
 \end{enumerate}
\end{proof}

\begin{remark}
Assume the JS-conditions hold. If $\cup C_i\neq\emptyset$, the above argument shows that any subsequence of $\{u_n\}$ has a further subsequence that converges uniformly on compact subsets of $\Omega$ to a solution of the Jenkins--Serrin problem. Since the solution is unique by Proposition~\ref{prop:uniqueness}, this easily implies that the original sequence $\{u_n\}$ given by~\eqref{eqn:sequence} converges itself to the solution. If $\cup C_i=\emptyset$, the same is true for $\{u_n-u_n(p)\}$ for any prescribed $p\in\Omega$ (no need of subsequences).
\end{remark}

\begin{remark}
Our approach also gives information if the JS-conditions do not hold or there are two adjacent arcs of type $A_i$ or $B_i$ by analysing the behavior of $\{u_n\}$. There are two possible scenarios:
\begin{enumerate}[label=(\alph*)]
	\item Every subsequence of $\{u_n\}$ has divergence lines. In particular, we can find a subsequence of $\{u_n\}$ where these divergence lines are disjoint and hence they are grouped in isotopy classes that behave as in Proposition~\ref{prop:disjoint-lines} and its corollaries (see Figure~\ref{fig:RI}). 
	\item There is a subsequence of $\{u_n\}$ without divergence lines, in which case we produce a minimal graph over all $\Omega$ with different boundary values. The rectangle of $\R^2$ in Figure~\ref{fig:culete} (left) cannot have divergence lines by symmetry and uniquness of solution, so $u_n(p)\to+\infty$ at any point $p$ of the rectangle. This means that $\{u_n-u_n(p)\}$ converges uniformly on compact subsets of the rectangle but we have performed an infinite translation downwards, so that the prescribed boundary values $0$ become $-\infty$ whilst the values $+\infty$ become $0$.
\end{enumerate}

We also point out that divergence lines cannot end at convex corners where two of the $C_i$ with finite values meet (we can use the small Scherk graphs as barriers at such a corner). However, there do exist examples in which divergence lines actually end on reentrant corners where two curves of type $C_i$ meet. The example in $\R^2$ given in Figure~\ref{fig:culete} (right) cannot converge after bounded or unbounded translations because the JS-conditions are not satisfied. It is not difficult to see that the divergence lines are those in dashed line that end at the concave vertex.
\end{remark}

\begin{figure}
\centering\includegraphics[width=0.7\textwidth]{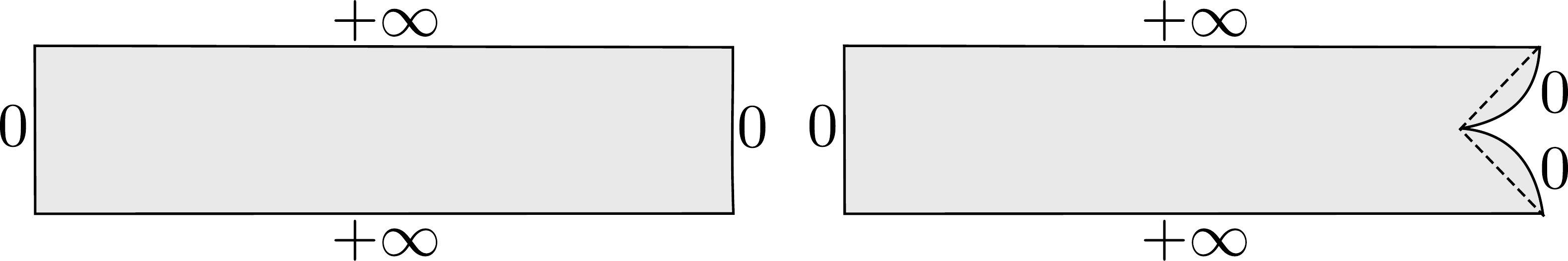}
\caption{On the left, an Jenkins--Serrin problem which has a solution if we change the boundary values. On the right, the dashed segments show up as divergence lines of the sequence $\{u_n\}$.}\label{fig:culete} 
\end{figure}

\section{Examples and further remarks}\label{sec:examples}

\subsection{New minimal surfaces in $\mathbb{R}^3$}

Let $\ell$ be the $z$-axis in $\R^3$, so that $\R^3\sm\ell$ can be seen as a Killing submersion with the Killing vector field $\xi=y\partial_x-x \partial_y$ generated by rotations about $\ell$. The affine planes of $\R^3$ containing $\ell$ are everywhere orthogonal to $\xi$, so the horizontal distribution associated to this Killing submersion is integrable. The metric in the orbit space $M=\{(x,z)\in\R^2:x>0\}$ that makes the projection $\pi:\R^3\sm\ell\to M$ Riemannian is the Euclidean one, and we also infer that $\tau(x,z)=0$ and $\mu(x,z)=x$ on $M$. The Killing submersion is completely determined in this way by also taking into account that $\R^3-\ell$ is not simply connected: it is the quotient by a vertical translation of the simply connected space $\E(M,\tau,\mu)$ that fibers over $M$ with bundle curvature $\tau$ and Killing length $\mu$. Recall that \emph{vertical} in $\E(M,\tau,\mu)$ is not the same as vertical in $\R^3$. Consequently, if $\Omega\subset M$ is an admissible Jenkins--Serrin domain, an eventual solution $\Sigma$ of the Jenkins--Serrin problem in $\E(M,\mu,\tau)$ that diverges in $\alpha\subset\partial\Omega$ will be embedded around $\alpha$ (since it is a Killing graph), but not properly embedded since it accumulates at $\pi^{-1}(\alpha)$. Also, the vertices of $\Omega$ always give rise to self-intersections of the boundary of the graph.

\begin{figure}
	\centering\includegraphics[width=0.9\textwidth]{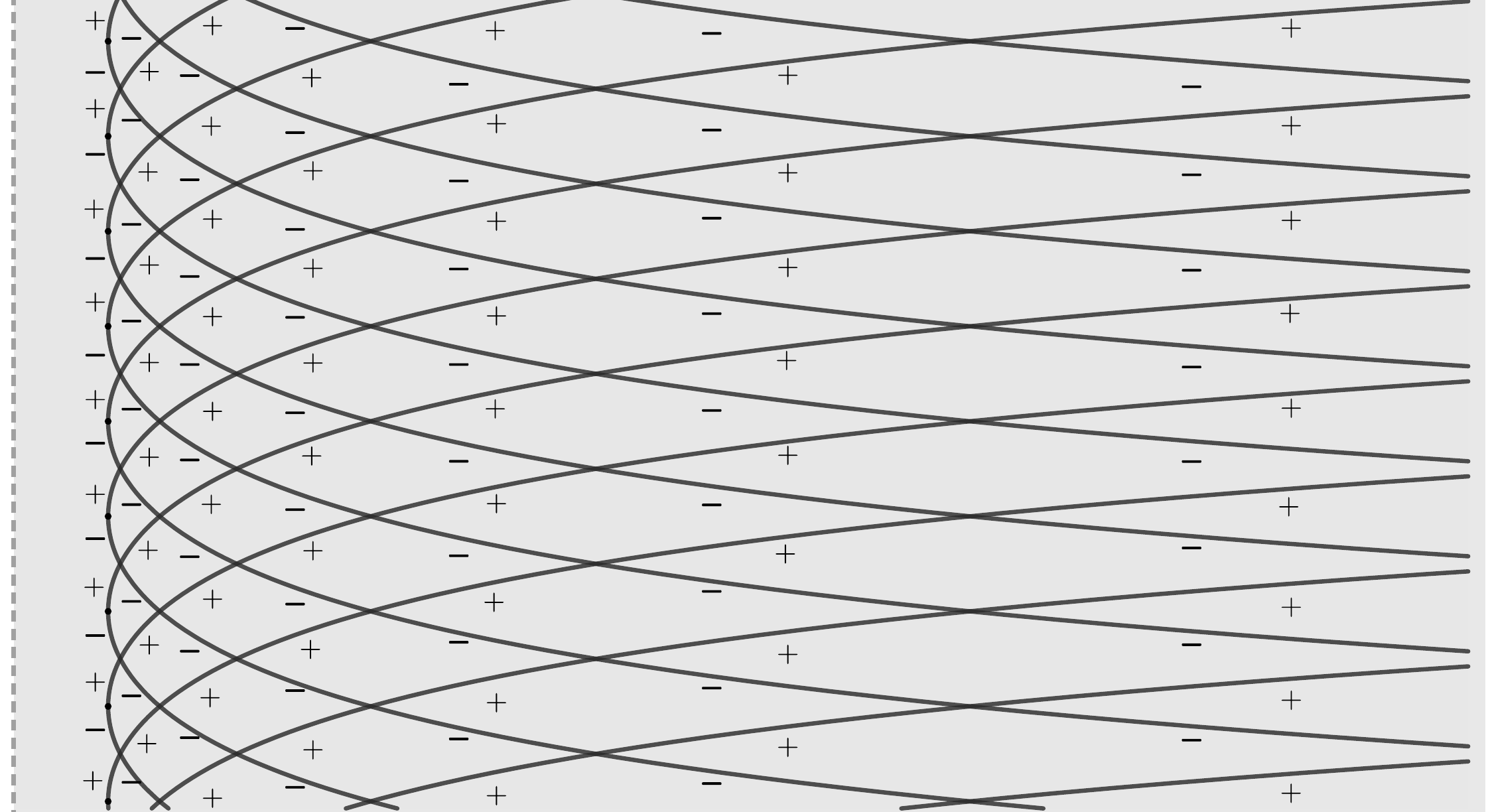}
	\caption{Tessellation of a vertical halfplane by catenaries that produce catenoids of $\R^3$ by rotation about the axis (in dashed line).}
	\label{TassellationM}
\end{figure}

Rotational minimal surfaces in $\R^3$ are catenoids and planes, from where we infer that $\mu$-geodesics in $M$ are catenaries (with respect to the $z$-axis) and straight lines (orthogonal to the $z$-axis). This gives the following $\mu$-geodesics depending on two parameters $a,b\in\R$:
\begin{itemize}
	\item $\alpha_a(t)=(t,a)$, which is defined for $t>0$ and hence noncomplete;
	\item $\beta_{a,b}(t)=(a\cosh(t/a),t+b)$ with $a>0$.
\end{itemize} 
For fixed $b\in\R$ and $a,c>0$, the curves $\left\{\beta_{a,b+kc}\right\}_{k\in\Z}$ produce a tiling of $M$ as shown in Figure~\ref{TassellationM}. Each $\mu$-geodesic $\beta_{a,b+kc}(t)$ is marked with the values $+\infty$ if $t>0$ and $-\infty$ if $t<0$. This produces a Jenkins--Serrin problem in each tile, which satisfies the JS-conditions since each tile is a $\mu$-quadrilateral symmetric with respect to the horizontal line passing through two of its vertices. The solution viewed in $\R^3\sm\ell$ is an embedded graph in the rotational direction that accumulates on closed subsets of four catenoids and whose boundary consists of four circumferences. 

It is natural to ask is if there exist values of $a,b,c$ such that the solutions on two tiles that are opposite by a vertex continue analytically each other. This is not trivial since there is no Schwarz reflection across circumferences of $\R^3$.

We also remark that constructions in the same spirit can be also done respect to screw motions in $\R^3$ by taking advantage of symmetric configuration of the $\mu$-geodesics. The same applies to the rest of $\E(\kappa,\tau)$-spaces.

\subsection{Scherk-like minimal surfaces in $\mathrm{Nil}_3$ and Berger spheres}

Consider the unit square $\Omega=(0,1)\times(0,1)\subset\R^2$ and assign values $\pm\infty$ to opposite sides of $\Omega$, as in the classical Scherk graph of $\R^3$. This gives rise trivially to a minimal graph in $\mathrm{Nil}_3$ over $\Omega$ with these assigned values. The resulting surface $\Sigma_0$ has boundary four vertical lines projecting to the vertices of $\Omega$, so it can be extended to a complete minimal surface by successive axial symmetries about its boundary components. 

The minimal set of such axial symmetries needed to go back to the tile $\Omega$ consists of the symmetries about $(0,0)$, $(0,-1)$, $(1,-1)$ and $(1,0)$, as shown in Figure~\ref{fig:wild-scherk}. The axial symmetry of $\mathrm{Nil}_3$ about the vertical axis $\{x=x_0,y=y_0\}$ reads
\[R_{(x_0,y_0)}(x,y,z)=(2x_0-x,2y_0-y,z+2\tau(y_0x-x_0y)),\]
and it follows that $R_{(1,0)}\circ R_{(1,-1)}\circ R_{(0,-1)}\circ R_{(0,0)}$ is the vertical translation $(x,y,z)\mapsto (x,y,z+4\tau)$. This means that on each shaded tile of the infinite chessboard, one finds infinitely many copies of $\Sigma_0$ evenly distributed at vertical distance $4\tau$ from the neighboring ones. Therefore, a Scherk-like surface in $\mathrm{Nil}_3$ is neither proper nor embedded.

\begin{figure}
\includegraphics[width=0.4\textwidth]{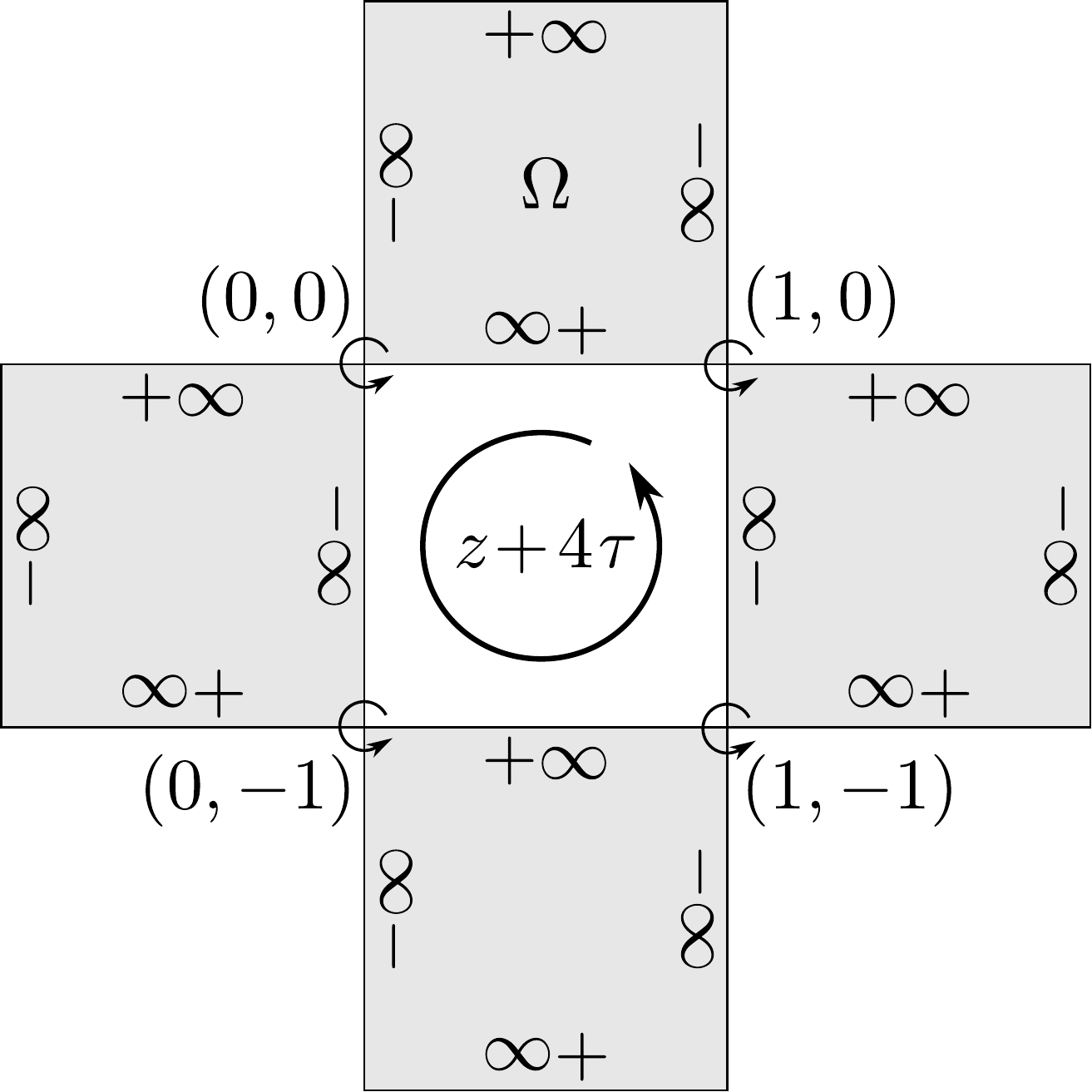}\qquad\includegraphics[width=0.4\textwidth]{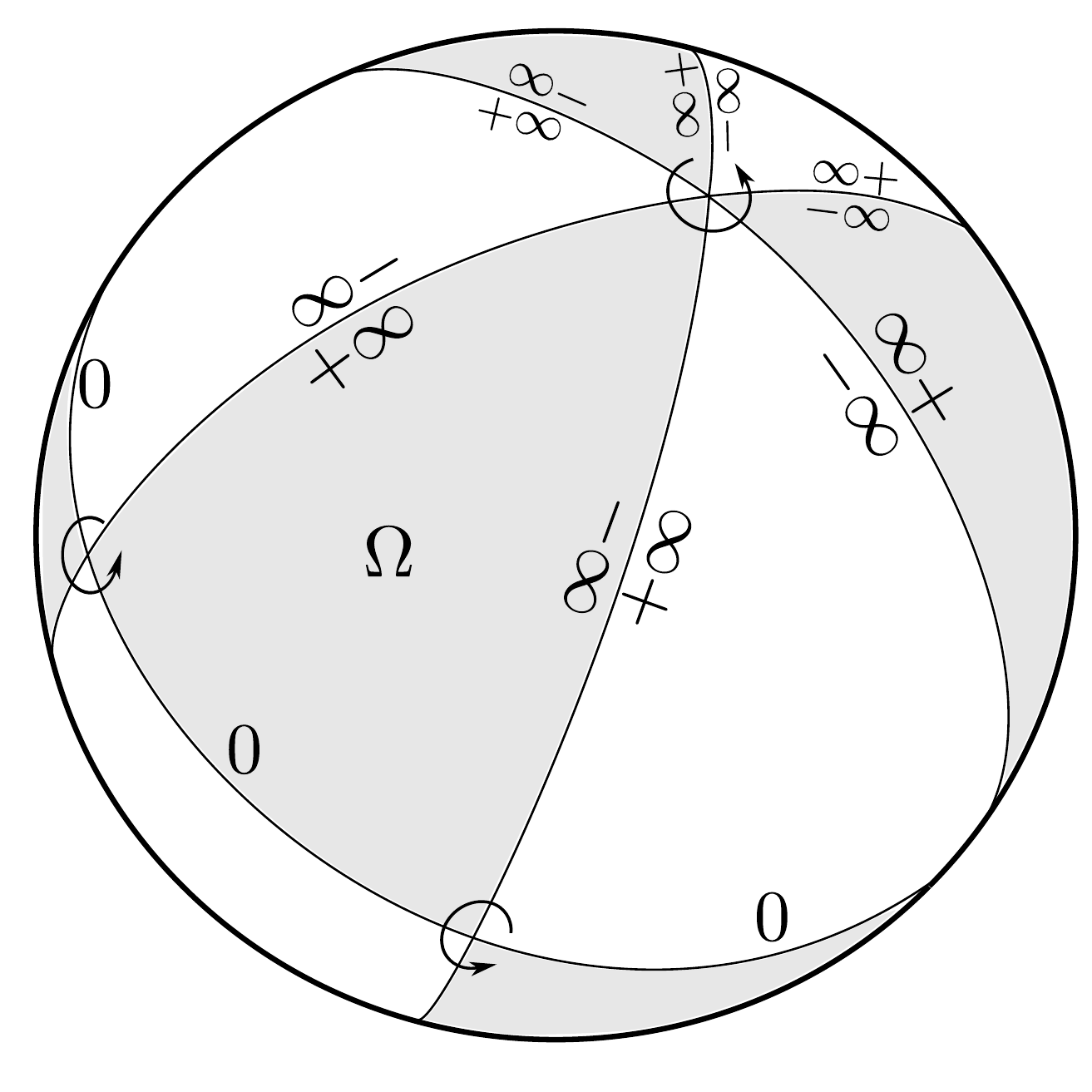}
\caption{Fundamental domains of a Scherk-like surface in $\mathrm{Nil}_3$ and the effect of the holonomy (left). Beach ball tessellation of $\mathbb{S}^2$ that leads to another complete surface in $\mathbb{S}^2\times\R$ or in Berger spheres. The values $0$ actually mean \emph{horizontal geodesic}.}\label{fig:wild-scherk}
\end{figure}

Similar constructions can be done in other $\E(\kappa,\tau)$-spaces by taking tessellations of $\mathbb{H}^2(\kappa)$ by regular $2m$-gons such that $2k$ of them meet at each vertex (such a tessellation exists if and only if $\frac{1}{m}+\frac{1}{k}\leq 1$). The above construction in $\mathrm{Nil}_3$ can be mimicked to get Scherk surfaces in $\mathbb{H}^2(\kappa)\times\mathbb{R}$ or $\widetilde{\mathrm{SL}}_2(\mathbb{R})$; in the latter case, we will find the same holonomy problem as in $\mathrm{Nil}_3$, so the resulting complete surface is invariant by a vertical translation and it is neither embedded nor proper. 

There is a special case which is worth mentioning, namely considering a \emph{beach ball} tessellation of $\mathbb{S}^2(\kappa)$ consisting of $2m$ sectors (or $2$-gons) whose sides are split in two arcs by adding the midpoints. Each sector becomes a quadrilateral in this way in which we can solve a Jenkins--Serrin problem (in $\mathbb{S}^2(\kappa)\times\R$ or in a Berger sphere) by prescribing alternating boundary values $\pm\infty$, see Figure~\ref{fig:wild-scherk} (right). The solution $\Sigma_0$ exists if $m\geq 2$ and has an additional axial symmetry that has been marked as zero in the Figure (this equator spans a horizontal geodesic $\Gamma\subset\Sigma_0$).
\begin{itemize}
	\item If $\tau=0$, then $\Sigma_0$ is completed by successive axial symmetries about the vertical geodesics projecting to $\Gamma$, since this process also provides an extension of $\Sigma_0$ beyond the geodesics projecting to the poles. All in all, we obtain a complete surface that consists of $4m$ copies of $\Sigma$, each of them projecting to one of the triangles (shaded or not) in Figure~\ref{fig:wild-scherk}. This surface is properly immersed in $\mathbb{S}^2(\kappa)\times\R$ with $2m$ annular ends asymptotic to vertical planes, since it takes $+\infty$ (resp.\ $-\infty$) values along $m$ great circles. 
	\item If $\tau\neq 0$, then the same ideas still apply, though we need $8m$ copies of $\Sigma_0$. The holonomy makes the horizontal geodesic $\Gamma$ project two-to-one to a great circle of $\mathbb{S}^2(\kappa)$, so that the complete surface projects two-to-one onto the interior of each of the triangles in Figure~\ref{fig:wild-scherk} and we have $4m$ annular ends.
\end{itemize}
There are other possible configurations that lead to interesting minimal surfaces consisting of finitely many isometric copies of a solution to a Jenkins--Serrin problem. It is likely that all these surfaces have finite total curvature in $\mathbb{S}^2(\kappa)\times\R$ or in a Berger sphere but this is an open question.

\subsection{Some topological observations}

In the above examples, we have seen that the condition that fibers have infinite length is not actually necessary for practical purposes, since one can work in a universal cover and then pass to the quotient. There are other three scenarios that is worth mentioning.

First, it is not necessary that the domain $\Omega\subset M$ is embedded. Assume that $\Omega'\subset M'$ is a relatively compact domain on some simply connected Riemannian surface $M'$ and let $\psi:M'\to M$ be an isometric immersion. Then we can consider the Killing submersion $\pi':\E'\to M'$ with bundle curvature and Killing length the pullback of $\tau$ and $\mu$ by, respectively. Since $\psi$ lifts to an isometric immersion $\Psi:\E'\to\E$ such that $\pi\circ\Psi=\psi\circ\pi'$, the solution of a Jenkins--Serrin problem over $\Omega'$ can mapped by $\Psi$ to a solution of a Jenkins--Serrin problem over the (possibly not embedded) domain $\psi(\Omega)\subset M$.

Second, an extremal case in our Jenkins--Serrin problem is the construction of minimal annuli over annular domains bounded by two closed geodesics in $M$. For instance, in Figure~\ref{fig:unduloid}, we have a rotational unduloid $M$, where we assume that $\mu\equiv 1$ and $\tau$ is arbitrary. Also, $A_1,A_2,B_1,B_2,B_3$ are closed embedded $\mu$-geodesics corresponding to maximal or minimal radii. In the following problems, we prescribe $+\infty$ (resp.\ $-\infty$) values in the components $A_i$ (resp.\ $B_i$) when they lie in the boundary of the domains under consideration.
\begin{itemize}
	\item In the domain bounded by $A_1$ and $B_1$, the Jenkins--Serrin problem has solution, because any possible closed simple $\mu$-geodesic has $\mu$-length larger than the (common) $\mu$-length of $A_1$ and $B_1$ (they minimize lengths in their isotopy class).
	\item In the domain bounded by $A_2$ and $B_2$, there is no solution because the inscribed polygon $A_2\cup B_1$ does not satisfy the JS-conditions.
	\item In the domain bounded by $A_1$ and $B_3$ there is no solution either, because the inscribed polygon $A_1\cup B_1$ does not satisfy the JS-conditions. 
\end{itemize}

\begin{figure}
\includegraphics[width=0.75\textwidth]{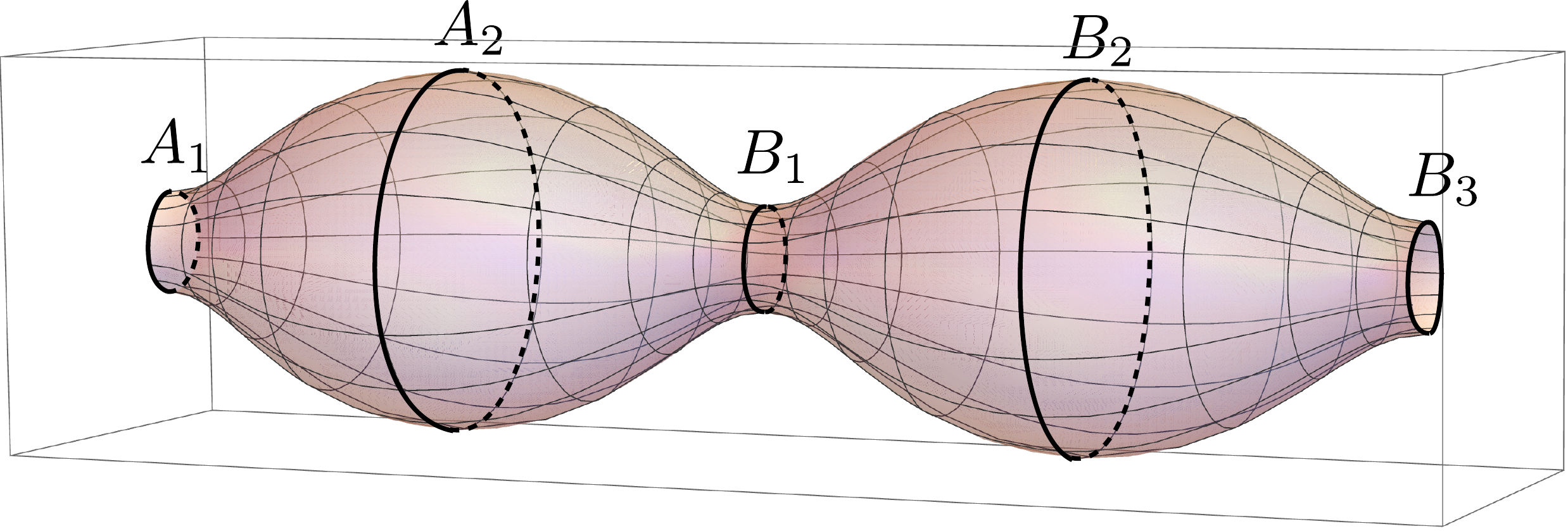}
\caption{Unduloid-like domains for Jenkins--Serrin problems.}\label{fig:unduloid}
\end{figure}

Third and last, we would like to point out some issue related to the condition (C2) in the Introduction, which is assumed in $M\times\R$ in order to adapt the original ideas by Jenkins and Serrin (see the Fourth case in the proof of~\cite[Thm.~3.3]{MRR}). Note also that this assumption does not give any trouble in $\mathbb{H}^2\times\R$ so it does not affect other results in~\cite{MRR}. In the case $\cup C_i=\emptyset$, take the sequence $v_n$ with values $0$ at the $A_i$ and $n$ at the $B_i$, and define the sets $E_c=\{p\in\Omega:v_n(p)>c\}$ and $F_c=\{p\in\Omega:v_n(p)<c\}$, which are disconnected when $c$ or $n-c$ are close enough to zero by condition (C2). The classical approach defines $\mu_n$ as the infimum of $c\in (0,n)$ such that $F_c$ is connected and claims that $E_{\mu_n}$ and $F_{\mu_n}$ are both disconnected. To see that this is not true in general, consider a sphere in which add four necks with boundary geodesics $A_1,A_2,B_1,B_2$ disposed symmetrically, as shown in Figure~\ref{fig:counterexample}. By uniqueness, the solution of the Jenkins--Serrin problem given by Theorem~\ref{thm:JS} has a symmetry with respect to a horizontal geodesic. Note that a similar example can be produced by removing four small polygons (with reentrant corners) in the round sphere $\mathbb{S}^2$.

\begin{figure}
\includegraphics[width=0.4\textwidth]{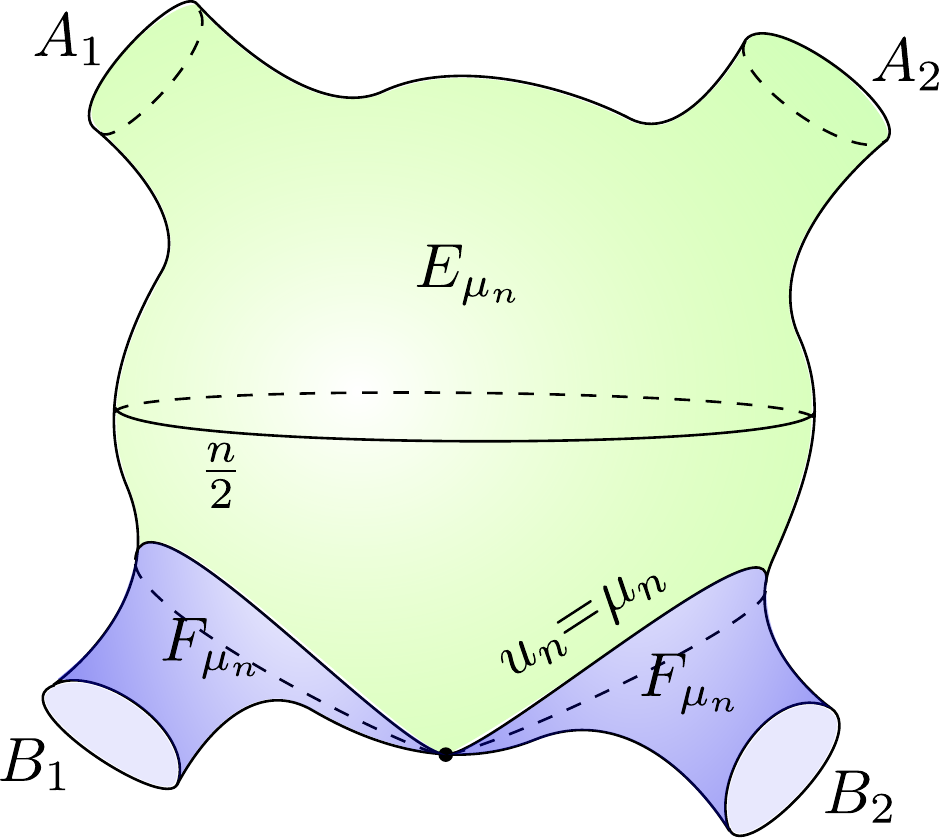}
\caption{The set $E_{\mu_n}$ (green) is connected.}\label{fig:counterexample}
\end{figure}

The aforesaid symmetry implies that $v_n$ has value $\frac{n}{2}$ along the symmetry curve (as shown in the figure), so it takes values larger (resp.\ smaller) than $\frac{n}{2}$ on the upper (resp.\ lower) half of the surface. At the first instant that the purple set $F_c$ gets disconnected, the green set $E_c$ is still connected.

\end{document}